\documentclass[11pt,reqno]{amsart}

\usepackage{amssymb,amsmath,graphicx,amsfonts,euscript}
\usepackage{color}

\newtheorem{thm}{Theorem}[section]

\newtheorem{prop}[thm]{Proposition}

\newtheorem{rem}[thm]{Remark}

\newtheorem{lemma}[thm]{Lemma}

\newcommand{\p}{\partial}

\newcommand{\na}{\nabla}
\newcommand{\De}{\Delta}
\newcommand{\ddt}{\frac{d}{dt}}
\newcommand{\R}{\mathbb{R}}
\newcommand{\Z}{\mathbb{Z}}

\newcommand{\cd}{\cdot}
\newcommand{\U}{\vec{u}}
\newcommand{\B}{\vec{b}}
\newcommand{\xe}{\vec{\xi}}
\newcommand{\F}{\mathcal{F}}
\newcommand{\G}{\mathcal{G}}

\newcommand{\LH}{\langle t-\tau\rangle}
\newcommand{\LHH}{\langle t\rangle}

\numberwithin{equation}{section}
\subjclass[2010]{35Q35, 76W05}
\keywords{MHD equations,   global small solution, decay estimate}
\begin{document}
\title[On the temporal decay for MHD]{On the temporal decay for the 2D non-resistive incompressible MHD equations}

\author[  R. Wan]{ Renhui Wan}
\address{ School of Mathematical Sciences, Nanjing Normal University, Nanjing 210023, China}

\email{rhwanmath@163.com,wrh@njnu.edu.cn}

\vskip .2in
\begin{abstract}
Califano-Chiuderi \cite{CC} gave the numerical observation
that the energy of the MHD equations   is dissipated at a rate
independent of the ohmic resistivity, which was first proved by \cite{RWXZ}[Ren et al., J. Funct. Anal., 2014] (the initial data near $(0,\vec{e}_1)$, $\vec{e}_1=(1,0)$). Precisely, they showed some explicit decay rates of solutions in $L^2$ norm. So a nature question is whether the obtained decay rates in \cite{RWXZ} is optimal. In this paper, we 
 aim at giving  the explicit decay rates of solutions in both $L^2$ norm and $L^\infty$ norm. In particular, our decay rate in terms of $L^2$ norm  improves the previous work \cite{RWXZ}.
\end{abstract}

\maketitle

\section{Introduction}
\label{intro}
\vskip .1in
In this paper, we are concerned with the Cauchy problem for the
two-dimensional (2D) non-resistive  incompressible   MHD equations  given by
\begin{equation} \label{Oringal}
\left\{
\begin{array}{l}
\partial_t \vec{u} + \vec{u}\cdot\nabla \vec{u} -  \Delta \vec{u} +\nabla p = \vec{H}\cdot\na \vec{H},\  (t,x,y)\in \R^+\times \R\times \R, \\
\partial_t \vec{H} + \vec{u}\cdot\nabla \vec{H}=\vec{H}\cd\na \vec{u},\\
{\rm div}\vec{u}={\rm div}\vec{H}=0, \\
\vec{u}|_{t=0} =\vec{u}_0(x,y), \quad \vec{H}|_{t=0}=\vec{H}_0(x,y),
\end{array}
\right.
\end{equation}
where $\vec{u}=(u,v)\in \R^2$ stand for the 2D velocity field,
$p$ the pressure and
$\vec{H}=(H_1,H_2)\in \R^2$  the magnetic field.
(\ref{Oringal}) can
be applied to model  plasmas when the plasmas are strongly collisional, or the resistivity since
these collisions are extremely small, see \cite{Cab} for more explanations to this model.
 For the MHD equations with both velocity dissipation and magnetic diffusion, \cite{DL} and \cite{ST} obtained  the  local and global well-posedness of solutions to that model, respectively.
 In both 2D and 3D, Chemin et al \cite{CMRR} showed the local existence of solutions to (\ref{Oringal})  with the initial data in critical Besov space (see \cite{Wan} for the uniqueness of solutions in 2D).
 However,
since there is no dissipation or damping in the equation for $\vec{H}$,  global well-posedness of smooth solutions to (\ref{Oringal}) even under  small assumption of the initial data has become  an issue that needed to be resolved.
\vskip .1in
Based on  Lagrangian coordinates and the techniques on anisotropic Besov spaces,
Lin, Xu and Zhang \cite{LXZ} first established the global well-posedness of small solutions
after translating the magnetic field by a constant vector and assuming that the initial magnetic field satisfies sort of admissible condition.
If we set $\vec{H}=\B+\vec{e}_1$, where $\B=(b,B)$, then the investigated model in \cite{LXZ} is
 \begin{equation} \label{mhd}
\left\{
\begin{array}{l}
\partial_t \vec{u} + \vec{u}\cdot\nabla \vec{u} -  \Delta \vec{u} +\nabla p = \vec{b}\cdot\na \vec{b}+\p_x\vec{b},\  (t,x,y)\in \R^+\times \R\times \R, \\
\partial_t \vec{b} + \vec{u}\cdot\nabla \vec{b}=\vec{b}\cd\na \vec{u}+\p_x\vec{u},\\
{\rm div}\vec{u}={\rm div}\vec{b}=0, \\
\vec{u}|_{t=0} =\vec{u}_0(x,y), \quad \vec{b}|_{t=0}=\vec{b}_0(x,y).
\end{array}
\right.
\end{equation}
Meanwhile,
\begin{equation}\label{1010}
\int(\vec{b}_0-\vec{e_1})(Z(t,\alpha))dt=0\ \ {\rm for\ \  all} \ \ \alpha\in \R^2\times\{0\}
\end{equation}
is the admissible condition, where $Z(t,\alpha)$ is determined by
$$\frac{d}{dt}Z(t,\alpha)=\vec{b}_0(Z(t,\alpha)),
\ Z(t,\alpha)|_{t=0}=\alpha.$$
 Later, by carefully exploiting the divergence structure of the velocity, Ren, Wu, Xiang and Zhang \cite{RWXZ} removed (\ref{1010}) and
obtained some decay estimates of solutions as follows:
\begin{equation}\label{1.1}
\|\p_x^k \B(t)\|_{L^2(\R^2)}
+\|\p_x^k \U(t)\|_{L^2(\R^2)}\lesssim\ \LHH^{-\frac{1+2k}{4}+\epsilon},
\end{equation}
where $\epsilon\in (0,1/2)$ and $k=0,1,2$.   (\ref{1.1})  confirms the numerical observation
that the energy of the MHD equations is dissipated at a rate
independent of the ohmic resistivity, see \cite{CC}.
 Zhang \cite{Zhangt14} gave a more elementary proof
for the global existence and uniqueness of solutions. Motivated by \cite{BSS},   Zhang \cite{Zhangt16} also proved
 global well-posedness  with  large background magnetic field by using the techniques in \cite{Zhangt14}  and the classical method for the oscillatory integrals.
 \vskip.1in
  Global well-posedness and large time behavior  of solutions to the 3D case have been recently treated in  Abidi-Zhang \cite{AZ} and Deng-Zhang \cite{DZ}, where the method also works for the 2D case. As a matter of fact, \cite{AZ} showed
  $$\|\U(t)\|_{H^2(\R^3)}+\|\B(t)\|_{H^2(\R^3)}\lesssim\ \LHH^{-\frac{1}{4}},$$
  which  corresponds  to the case $\epsilon=0$ in (\ref{1.1}).  By exploiting H\"{o}rmander's version of Nash-Moser iteration scheme,
 \cite{DZ} derived the decay rate of  solutions in both $L^\infty$ and $L^2$ norms. In particular,  the decay rate in the $L^2$ norm  is optimal in sense that it coincides with that of the linear system. Indeed, the decay rate of solutions in \cite{DZ} can be given as follows:
 \begin{equation}\label{dzd}
 \begin{aligned}
 \|\U(t)\|_{W^{2,\infty}(\R^3)}\le\ C_\kappa\LHH^{-\frac{5}{4}+\kappa},&
 \ \|\B(t)\|_{W^{2,\infty}(\R^3)}\le\ C_\kappa\LHH^{-\frac{3}{4}+\kappa},\\
 \|\U(t)\|_{H^2(\R^3)}+\|\B(t)\|_{H^2(\R^3)}\le\ C\LHH^{-\frac{1}{2}},&\
 \|\na \U(t)\|_{L^2(\R^3)}\le\ C\LHH^{-1},
 \end{aligned}
 \end{equation}
 where the positive constant $\kappa$ is sufficiently small provided that the regularity of solutions is large enough.
 We refer the interested reader to \cite{PZZ,Wu15,Wu17,XZ15} and references therein  for other related works.
 \vskip.1in
\vskip .1in
Let
$$V=(u,v,b,B),\ \ V_0=(u_0,v_0,b_0,B_0).$$
We define $S_1\times S_2$ with their norm as follows:
$$
\begin{aligned}
\|\U\|_{S_1}=\sup_{t\ge 0}\Big\{&\|\vec{u}\|_{H^N}+\|\na\vec{u}\|_{L^2_t(H^N)}
+\langle t\rangle^\frac{1}{2}\|\vec{u}\|_{L^2}
+\langle t\rangle^\frac{3}{4}\|\p_yu\|_{L^2}\\
&+\langle t\rangle(\|\p_x\vec{u}\|_{H^2}
+\|\vec{u}\|_{\F L^1})+\langle t\rangle^\frac{5}{4}\|\p_xu\|_{\F L^1}\Big\};\\
\|\B\|_{S_2}=\sup_{t\ge 0}
\Big\{&\|\vec{b}\|_{H^N}+\|\p_x\vec{b}\|_{L^2_t(H^{N-1})}
+\langle t\rangle^\frac{1}{4}\|b\|_{L^2}
+\langle t\rangle^\frac{1}{2}(\||\na|^{-1}\langle\na\rangle b\|_{\F L^1}+\|B\|_{L^2})\\
&+\langle t\rangle^\frac{3}{4}\|\p_xb\|_{H^1}
+\langle t\rangle(\|\p_xB\|_{L^2}+\|B\|_{\F L^1}+\|\mathcal{R}_1\langle\na\rangle b\|_{\F L^1})\Big\};\\
\|V\|_{3}=\|\U\|_{S_1}&+\|\B\|_{S_2},\ \|V_0\|_3=\|V_0\|_{H^N}
+\|V_0\|_{W^{5,1}},
\end{aligned}
$$
where $\|f\|_{\F L^1}=\|\widehat{f}\|_{L^1}$,  $\mathcal{R}_1$ stands for the Riesz transform, the operators $|\na|$ and $\langle\na\rangle$ are standard.
\vskip.1in
Now, we give the main result of this paper.
\begin{thm}\label{t1}
Let $N\ge 8$ and $(\U_0,\B_0)\in H^N(\R^2)\cap W^{5,1}(\R^2)$ satisfying ${\rm div} \vec{u}_0={\rm div} \vec{b}_0=0$. Then there exists a sufficiently small positive constant $c_0$ such that if
$$\|V_0\|_3\le \ c_0,$$
then (\ref{mhd}) has a unique global solution $(\U,\B)\in S_1\times S_2$. Moreover,
\begin{equation}\label{resu1}
\|V\|_3\lesssim\ c_0.
\end{equation}
\end{thm}
\begin{rem}\label{r100}
(1)\  Due to $\|f\|_{L^\infty}\le\|\widehat{f}\|_{L^1}$,
we can get the decay rates in  $L^\infty$ norm, and then fulfill
 the decay rates in $L^p$ $(2\le p\le \infty)$ norm by interpolation.    In particular,
\begin{equation}\label{1.3}
\|\p_x^k\U(t)\|_{L^2(\R^2)}+
\|\p_x^kB(t)\|_{L^2(\R^2)}\lesssim\ c_0\LHH^{-\frac{1+k}{2}},\ \|\p_x^kb(t)\|_{L^2(\R^2)}\lesssim\ c_0\LHH^{-\frac{1+2k}{4}},\ k=0,1,
\end{equation}
which obviously improves (\ref{1.1}). In fact,  our method works for the case $k\ge2$ in (\ref{1.3}), but we choose not to pursue on this direction here.\\
(2)\  Here the $L^p$ $(2\le p\le\infty)$ norm of $B$ decays faster than the associated  norm of $b$, whereas this type result is not proved in  \cite{AZ,DZ,RWXZ}. \\
(3)\ Our idea seems hard to be applied for the 3D case,  since the divergence structure of the velocity field in this case  can not be effectively used.
\end{rem}
\begin{rem}\label{rr1}
Comparing with the result on the 2D compressible MHD equations \cite{Wu17}, we can see that the estimates in $H^N$ norm do not grow over time  is the principal difference.
\end{rem}
Formally, the approach in the present work is similar to the works \cite{Wu15} and \cite{Wu17}, but there are many differences in the proof. Let us now outline some principal differences between \cite{Wu15,Wu17} and the present work.
\vskip.1in
In \cite{Wu15}, Wu-Wu-Xu considered 2D incompressible  MHD equations with only a  velocity damping term when the initial data is close to $(0,\vec{e}_1)$. Not only did they obtain  the global well-posedness of solutions, but also some decay estimates of solutions. The  part on the velocity (i.e., $\B=0$) is the 2D incompressible damped Euler equations, whose solution even in  $L^2$ norm has
exponential decay estimate by only using energy method provided that the initial data is sufficiently small.  Here the  part on the velocity (i.e., $\B=0$ in (\ref{Oringal}))  is the 2D incompressible Navier-Stokes equations.  However,  the decay rate
of solutions to this model   are  polynomial. In particular, the decay rate of  solution in  $L^2$ norm is
slower than $\LHH^{-1}$. By  the virtue that it is not integrable, it seems more difficult than \cite{Wu15} to obtain global well-posedness of solutions.
\vskip.1in
 Besides the way dealing with the pressure, there are some other differences between our work and \cite{Wu17}.
 In our work, we can make use of the structure of incompressibility of the fluid to control the $H^N$ estimate of solutions  by some special  norms, the decay rate of which is integrable, and then achieve the goal that the $H^N$ estimates of solutions  do not grow over time. In this process,  it is the new decay estimate of $\p_x u$ in $\F L^1$ norm (or $L^\infty$ norm) that plays an important role, while this idea also works for the model studied in \cite{Wu15}.   On the other hand, we need to establish some other new decay estimates including
 \begin{equation}\label{11111}
\||\na|^{-1}\langle\na\rangle b(t)\|_{\F L^1}
\lesssim\ \LHH^{-\frac{1}{2}},
\ \|\mathcal{R}_1\langle\na\rangle b(t)\|_{\F L^1}+\|B(t)\|_{\F L^1}\lesssim\ \LHH^{-1}.
\end{equation}
However, this type goal seems  hard  to be fulfilled for the compressible model in \cite{Wu17}.
\vskip .1in
Let us do some comments on the proof and our idea. Firstly, all previous works dealing with the 2D incompressible or compressible  cases applied the magnetic potential equation for $\phi$ defined by
  $
 \vec{H}=(\p_y\phi,-\p_x\phi),
$
 here we do not introduce this magnetic potential equation any more.
Our idea is  considering (\ref{mhd}) as two subsystems (\ref{ub1}) and (\ref{vB}),   and then using  the   method of diagonalization via the eigenvalues and eigenvectors to these subsystems.  Secondly, by using the structure of system and integration by parts many times, we can use the  integral
 $$
 \int_0^t\|v\|_{L^\infty}^2+\|B\|_{L^\infty}^2+\|\p_xu
\|_{L^\infty} d\tau
 $$
 to control the $H^N$ estimate of solutions. Thanks to the fast decay rate of these special norms: $\|v\|_{L^\infty}^2$, $\|B\|_{L^\infty}^2$ and $\|\p_xu
\|_{L^\infty}$, we can obtain   the $H^N$ estimate  of solutions (uniformly in time). Thirdly, we shall establish some new decay estimate like (\ref{11111}) to get the different large-time behavior of $b$ and $B$ in $\F L^1$ norm (or $L^\infty$ norm), which  is not obtained for the compressible MHD equations in \cite{Wu17}.
 \vskip .1in
\begin{proof}[Proof of Theorem \ref{t1}]
Thanks to \cite{CMRR} and \cite{Wan}, one can easily get the local well-posedness
of solutions to (\ref{mhd}).
Claim :
\begin{equation}\label{endend}
\|V\|_3\lesssim\ \|V_0\|_3+\|V_0\|_3^2+\|V\|_3^\frac{3}{2}+\|V\|_3^3,
\end{equation}
the proof of which is provided at the end of the ninth section,
then we can conclude the proof of Theorem \ref{t1} by the standard continuity argument.
\end{proof}
\vskip.1in
The present paper is structured as follows:\\
 In the second section,  we provide the definitions of some operators  and  some spaces. The third section devotes to giving  the integral representation of solutions. The fourth section gives several  decay estimates on some operators  and nonlinear decay estimates, which is an essential part in this paper. From the fifth section to the ninth section, we   devotes to showing (\ref{endend}).
 In the Appendix,  we   give the proofs of some lemmas which are used in the previous parts.
\vskip .1in
Let us complete this section by describing the notations we shall use in this paper.\\
{\bf Notations} We use $A\lesssim B$ to denote the statement that $A\le CB$ for some absolute constant $C>0$.
$A\thickapprox B$  means   $A\lesssim B$ and $B\lesssim A$.
 $\langle t\rangle$ means $\sqrt{1+t^2}$. We  use $\mathcal{R}_{ij}=\mathcal{R}_i\mathcal{R}_j$, where $\mathcal{R}_i$ and $\mathcal{R}_j$  stand for the Riesz transform.
 We shall denote by $(a|b)$  the $L^2$ inner product
of $a$ and $b$, and
$$(a|b)_{\dot{H}^s}\stackrel{\rm def}{=}(|\na|^s a\big||\na|^s b),
 \ {\rm and}\ \ (a|b)_{\dot{H}^m}\stackrel{\rm def}{=} (\p^m a|\p^mb)\  (m {\rm \ is\  an\  integer}),
$$
$$(a|b)_{H^s}\stackrel{\rm def}{=}(a|b)+(a|b)_{\dot{H}^s}.$$
\section{Preliminaries}
\label{pre}
\vskip .1in
The fractional Laplacian operator $|\na|^\alpha=(-\Delta)^\frac{\alpha}{2}$ is defined through the Fourier transform, namely,
$$\widehat{|\na|^\alpha f}(\xi,\eta)\stackrel{\rm def}{=}|\xe|^\alpha \widehat{f}(\xi,\eta),$$
where $\xe=(\xi,\eta)$ and the  Fourier transform  is given by
$$\widehat{f}(\xi,\eta)\stackrel{\rm def}{=}\int_{\mathbb{R}^2}e^{-i(x\xi+y\eta)}f(x,y)d\xi d\eta.$$
 We also use $\F \{f\}$ to stand for the Fourier transform for some convenience. We define
 $$\|f\|_{\F L^p}\stackrel{\rm def}{=}\|\widehat{f}\|_{L^p},\ 1\le p\le\infty.$$
\vskip.1in
Let $\psi(\xe)$ be a smooth bump function adapted to $\{|\xe|\le 2\}$ and equal to 1 on $\{|\xe|\le 1\}$. For $N>0$,  we define the Fourier multipliers
$$\widehat{P_{\le N}f}(\xi,\eta)=\psi(\xi/N,\eta/N)\widehat{f}(\xi,\eta),
\ \widehat{P_{>N}f}(\xi,\eta)=\big(1-\psi(\xi/N,\eta/N)\big)\widehat{f}(\xi,\eta),$$
$$\widehat{P_{N}f}(\xi,\eta)=\big(\psi(\xi/N,\eta/N)
-\psi(2\xi/N,2\eta/N)\big)\widehat{f}(\xi,\eta),$$
and $P_{<N}$ and $P_{\ge N}$ can be  defined similarly. We also define $$P_{M<\cdot\le N}=P_{\le N}-P_{\le M}$$
when $M<N$.  We will usually apply these multipliers when $N$ and $M$ are dyadic numbers (i.e., of the form $2^\Z$ in general). In particular, all summation over $N$ are understood to be over dyadic numbers.
\vskip.1in
When $1\le p\le \infty$, we  define
$$\|f\|_{W^{k,p}(\R^2)}\stackrel{\rm def}{=}\| f\|_{L^p(\R^2)}+\||\na|^{k} f\|_{L^p(\R^2)}\ (k>0),\ \|f\|_{\dot{W}^{k,p}(\R^2)}\stackrel{\rm def}{=}\||\na|^{k} f\|_{L^p(\R^2)},\ (k\in\Z).
$$
For the special case $p=2$,  $W^{k,p}(\R^2)$ and $\dot{W}^{k,p}(\R^2)$ reduces to $H^{k}(\R^2)$  and $\dot{H}^k(\R ^2)$, respectively.
\vskip.1in
The following two lemmas provide Bernstein's inequality and product estimate.
\begin{lemma}
For $1\le p\le q\le \infty$ and $N>0$, then
$$\||\na|^{\pm s}P_Nf\|_{L^p(\R^2)}\thickapprox N^{\pm s}\|P_N f\|_{L^p(\R^2)},$$
$$
\|(P_N,P_{\le N})f\|_{L^q(\R^2)}
\lesssim\ N^{\frac{2}{p}-\frac{2}{q}}\|(P_N,P_{\le N})f\|_{L^p(\R^2)}.$$
\end{lemma}
\begin{lemma}[Product estimate \cite{KP}]
Let $n\ge 1$, $s>0$, $1\le p,r\le \infty,$ then
\begin{equation}\label{kp}
\||\na|^s( fg)\|_{L^p(\R^n)}\le C\left\{\|f\|_{L^{p_{1}}(\R^n)}
\||\na|^s g\|_{L^{p_2}(\R^n)}
+\|g\|_{L^{r_{1}}(\R^n)}\||\na|^sf\|_{L^{r_2}(\R^n)}\right\},
\end{equation}
where $1\le p_{1},r_{1}\le \infty$ such that $\frac{1}{p}=\frac{1}{p_{1}}+\frac{1}{p_{2}}=\frac{1}{r_{1}}+\frac{1}{r_{2}}$.
\end{lemma}
At last, we list some basic inequalities including classical estimates of solution to the  Heat equation. Since the proof  is easy, we omit it.
\begin{lemma}\label{l2.1}
Let $n=1,2$.
(1)\  If $f\in L^1(\R^n)\cap L^2(\R^n)$, then
\begin{equation}\label{Heat}
\|e^{ct\De}f\|_{L^2(\R^n)}\lesssim t^{-\frac{n}{4}}\|f\|_{L^1(\R^n)},
\ \|e^{ct\De}f\|_{\F L^1(\R^n)}\lesssim\ \min\{t^{-\frac{n}{2}}\|f\|_{L^1(\R^n)},t^{-\frac{n}{4}}
\|f\|_{L^2(\R^n)}\}.
\end{equation}
(2)
Let $\epsilon>0$, $\sigma>\frac{n}{2}$ and $1\le r\le2$, then
$$\|f\|_{L^\frac{r}{r-1}(\R^n)}\lesssim\  \|\widehat{f}\|_{L^r(\R^n)},\ \|f\|_{L^2(\R^n)}\thickapprox \|\widehat{f}\|_{L^2(\R^n)},$$
$$\|fg\|_{H^\sigma(\R^n)}\lesssim \|f\|_{H^\sigma(\R^n)}\|g\|_{H^\sigma(\R^n)},\ \
\|\widehat{fg}\|_{L^1(\R^n)}
\le \|\widehat{f}\|_{L^1(\R^n)}\|\widehat{g}\|_{L^1(\R^n)},$$
$$\|\widehat{f}\|_{L^r(\R^2)}
\lesssim\ \|\langle \na\rangle^{\frac{2}{r}-\frac{1}{2}
  +\epsilon}f\|_{L^1_x(L^2_y)}.$$
\end{lemma}
\section{The integral representation of solutions}
\label{spec}
\vskip.1in
In this section, we shall obtain the integral representation of solutions to (\ref{mhd}).
Let us investigate the spectrum properties to the following two systems:
\begin{equation} \label{ub1}
\left\{
\begin{array}{l}
\partial_t u -\Delta u   =\p_x b+F^1,  \\
\partial_t b   =\p_x u+G^1,\\
G^1=-\U\cd\na b+\B\cd\na u,\\
 F^1=-\U\cd\na u-\p_xp+\B\cd\na b
\end{array}
\right.
\end{equation}
and
\begin{equation} \label{vB}
\left\{
\begin{array}{l}
\partial_t v -\Delta v   =\p_x B+F^2,  \\
\partial_t B   =\p_x v+G^2,\\
G^2=-\U\cd\na B+\B\cd\na v,\\
 F^2=-\U\cd\na v-\p_yp+\B\cd\na B,
\end{array}
\right.
\end{equation}
where
$$p=-\De^{-1}{\rm div}(\U\cd\na \U-\B\cd\na\B).$$
Denote $\vec{\xi}=(\xi,\eta)$ and
$$A=
\left(
  \begin{array}{cc}
    -|\xe|^2 & -i\xi \\
    -i\xi & 0 \\
  \end{array}
\right),
$$
then  the eigenvalues of the matrix $A$ can be given by
 \begin{equation*}
\lambda_\pm=\left\{
\begin{array}{l}
 \frac{-|\xe|^2\pm \sqrt{|\xe|^4-4\xi^2}}{2},\ \ {\rm when} \ \    |\xi|<\frac{|\xe|^2}{2},\\
\frac{-|\xe|^2\pm i\sqrt{4\xi^2-|\xe|^4}}{2},\ \ {\rm when} \ \  |\xi|\ge\frac{|\xe|^2}{2},
\end{array}
\right.
\end{equation*}
where $i=\sqrt{-1}$. After using Fourier transform, (\ref{ub1}) and
(\ref{vB}) reduces to
\begin{equation}\label{Fub}
\p_t\left(
        \begin{array}{c}
          \widehat{u} \\
          \widehat{b} \\
        \end{array}
      \right)(\xe)
=A \left(
        \begin{array}{c}
          \widehat{u} \\
          \widehat{b} \\
        \end{array}
      \right)(\xe)
      +\left(
        \begin{array}{c}
          \widehat{F^1} \\
          \widehat{G^1} \\
        \end{array}
      \right)(\xe)
\end{equation}
and
\begin{equation}\label{FvB}
\p_t\left(
        \begin{array}{c}
          \widehat{v} \\
          \widehat{B} \\
        \end{array}
      \right)(\xe)
=A \left(
        \begin{array}{c}
          \widehat{v} \\
          \widehat{B} \\
        \end{array}
      \right)(\xe)
      +\left(
        \begin{array}{c}
          \widehat{F^2} \\
          \widehat{G^2} \\
        \end{array}
      \right)(\xe).
\end{equation}
\vskip.1in
It follows by using  the  method of diagonalization via the eigenvalues and eigenvectors,   (\ref{Fub}) and (\ref{FvB}) that
\begin{equation}\label{B}
\B(t)=M_{1}(\p,t)\U_0+M_{2}(\p,t)\B_0
+\int_0^t M_{1}(\p,t-\tau)\vec{F}d\tau
+\int_0^t M_{2}(\p,t-\tau)\vec{G}d\tau
\end{equation}
and
\begin{equation}\label{U}
\U(t)=M_{3}(\p,t)\U_0+M_{1}(\p,t)\B_0
+\int_0^t M_{3}(\p,t-\tau)\vec{F}d\tau
+\int_0^t M_{1}(\p,t-\tau)\vec{G}d\tau,
\end{equation}
 where
$$\vec{F}=(F^1,F^2),\ \vec{G}=(G^1,G^2),\ \widehat{M_{i}f}(\xe,t)=
\widehat{M_{i}}(\xe,t)\widehat{f}(\xe),\ i=1,2,3$$
and
\begin{equation}\label{MM}
\big(
  \begin{array}{ccc}
     \widehat{M_{1}}(\xe,t),& \widehat{M_{2}}(\xe,t), & \widehat{M_{3}}(\xe,t)\\
  \end{array}
  \big)
\stackrel{\rm def}{=}
  \big(
    \begin{array}{ccc}
i\xi\frac{e^{\lambda_-t}-e^{\lambda_+t}}{\lambda_+-\lambda_-}, &  \frac{\lambda_+ e^{\lambda_-t}-\lambda_- e^{\lambda_+t}}
{\lambda_+-\lambda_-}, & \frac{\lambda_- e^{\lambda_-t}-\lambda_+e^{\lambda_+t}}
      {\lambda_--\lambda_+}\\
    \end{array}
  \big).
\end{equation}
Notice that
\begin{equation}\label{tm}
\p_tM_{2}=\p_xM_{1},\ \p_tM_{1}=\p_xM_{3},
\end{equation}
which is useful in the following context.
To bound $M_{i}(\p,t)$, we split
the whole space $\R^2$ into   four   regions:
\begin{equation}\label{domain}
  \begin{aligned}
   D_1\stackrel{\rm def}{=}&\{\xe\in \R^2:\ |\xi|\ge |\xe|^2\},\\
D_2\stackrel{\rm def}{=}&\{\xe\in \R^2:\ \frac{|\xe|^2}{2}\le|\xi|<|\xe|^2\},\\
D_3\stackrel{\rm def}{=}&\{\xe\in \R^2:\ \frac{|\xe|^2}{4}\le |\xi|<\frac{|\xe|^2}{2}\},\\
D_4\stackrel{\rm def}{=}&\{\xe\in \R^2:\ |\xi|< \frac{|\xe|^2}{4}\}.
  \end{aligned}
\end{equation}
In order to help us establish   some estimates of solutions,
 $D_4$ will be seen as two sets in many times, namely,
\begin{equation}\label{D4f}
D_4=D_{41}\cup D_{42},\ D_{41}=D_4\cap \{\xe:|\xe|\ge 1\},\ D_{42}=D_4\cap \{\xe:|\xe|< 1\}.
\end{equation}
Due to the definition  in (\ref{domain}), it is easy to get
\begin{equation}\label{proper1}
|\xe|\lesssim 1\ {\rm  when}\   \xe\in D_1\cup D_2\cup D_3,
\end{equation}
so that we can   bound some estimates on $D_i$ $(i=1,2,3)$ under  low regularity assumption.
\vskip .1in
Next, a proposition devoting  to the estimates of $\widehat{M_{i}}(\xe,t)$  is given.
\begin{prop}\label{p3.1}
 $\widehat{M_{i}}(\xe,t)$ $(i=1,2,3)$ defined by (\ref{MM}) satisfies the following estimates:\\
(1) if $\xe\in D_1$,
\begin{equation}\label{rf1}
\begin{aligned}
|\widehat{M_{1}}(\xe,t)|\lesssim\ e^{-\frac{|\xe|^2}{2}t}&,\
|\widehat{M_{2}}(\xe,t)|\lesssim\ e^{-\frac{|\xe|^2}{2}t},\
|\widehat{M_{3}}(\xe,t)|\lesssim\ e^{-\frac{|\xe|^2}{2}t};
\end{aligned}
\end{equation}
(2) if $\xe\in D_2$,
\begin{equation*}
\begin{aligned}
|\widehat{M_{1}}(\xe,t)|\lesssim\ e^{-\frac{|\xe|^2}{4}t}&,\
|\widehat{M_{2}}(\xe,t)|\lesssim\ e^{-\frac{|\xe|^2}{4}t},\
|\widehat{M_{3}}(\xe,t)|\lesssim\ e^{-\frac{|\xe|^2}{4}t};
\end{aligned}
\end{equation*}
(3) if $\xe\in D_3$,
\begin{equation*}
\begin{aligned}
|\widehat{M_{1}}(\xe,t)|\lesssim\ e^{-\frac{|\xe|^2}{32}t}&,\
|\widehat{M_{2}}(\xe,t)|\lesssim\ e^{-\frac{|\xe|^2}{32}t},\
|\widehat{M_{3}}(\xe,t)|\lesssim\ e^{-\frac{|\xe|^2}{32}t};
\end{aligned}
\end{equation*}
(4) if $\xe\in D_4$,
\begin{equation}\label{rf2}
\begin{aligned}
|\widehat{M_{1}}(\xe,t)|\lesssim\ \frac{|\xi|}{|\xe|^2}e^{-\frac{\xi^2}{|\xe|^2}t},&\
|\widehat{M_{2}}(\xe,t)|\lesssim\ e^{-\frac{\xi^2}{|\xe|^2}t},\
|\widehat{M_{3}}(\xe,t)|\lesssim\ e^{-\frac{|\xe|^2}{2}t}+\frac{\xi^2}{|\xe|^4}e^{-\frac{\xi^2}
{|\xe|^2}t}.
\end{aligned}
\end{equation}
\end{prop}
\begin{proof}
 (1)\ Since $\xe\in D_1$, we have
 $$|\lambda_+-\lambda_-|\thickapprox |\lambda_\pm|\thickapprox |\xi|,\
 |e^{\lambda_\pm t}|\lesssim\ e^{-\frac{|\xe|^2}{2}t},$$
  which yields the desired estimate (\ref{rf1}) by some basic computations.
  \vskip .1in
  (2) In $D_2$, we have $|\lambda_\pm|=|\xi|$ and $|\xi|\thickapprox|\xe|^2$. For $\widehat{M}_{1}(\xe,t)$, using $|\sin x|\le |x|$, we have
  \begin{equation*}
  \begin{aligned}
  |\widehat{M_{1}}(\xe,t)|
  \lesssim\  |\xi|e^{-\frac{|\xe|^2}{2}t}
  \Big|\frac{\sin(\frac{\sqrt{4\xi^2-|\xe|^4}}{2}t)}
  {\sqrt{4\xi^2-|\xe|^4}}\Big|
  \lesssim\ |\xi|te^{-\frac{|\xe|^2}{2}t}
  \lesssim\ \frac{|\xi|}{|\xe|^2}e^{-\frac{|\xe|^2}{4}t}
  \lesssim\ e^{-\frac{|\xe|^2}{4}t},
  \end{aligned}
  \end{equation*}
  which,
  together with
  \begin{equation}\label{100}
  |\widehat{M_{2}}(\xe,t)|\le
  |e^{\lambda_-t}|+\frac{|\lambda_-|}{|\xi|}|\widehat{M_{1}}
  (\xe,t)|,\ |\widehat{M_{3}}(\xe,t)|\le
  |e^{\lambda_-t}|+\frac{|\lambda_+|}{|\xi|}|\widehat{M_{1}}
  (\xe,t)|,
  \end{equation}
  yields the desired result.
\vskip .1in
(3) In $D_3$, one has
$$|\xi|\thickapprox|\xe|^2,\ \lambda_-\in (-|\xe|^2,-\frac{|\xe|^2}{2}),\ \lambda_+\ge \lambda_-$$
and
$$e^{\lambda_+t}-e^{\lambda_-t}\le
e^{\lambda_+t}(\lambda_+-\lambda_-)t,\ \lambda_+
=-\frac{2\xi^2}{|\xe|^2+\sqrt{|\xe|^4-4\xi^2}}\in (-\frac{|\xe|^2}{2}, -\frac{|\xe|^2}{16}],$$
which follows
$$|\widehat{M_{1}}(\xe,t)|\lesssim |\xi||\frac{e^{\lambda_+t}-e^{\lambda_-t}}{\lambda_+-\lambda_-}|
\lesssim\ |\xi|e^{\lambda_+t}t
\lesssim\ \frac{|\xi|}{|\xe|^2}e^{-\frac{|\xe|^2}{16}t}(t|\xe|^2)
\lesssim\ e^{-\frac{|\xe|^2}{32}t}.
$$
Thanks to (\ref{100}), we can get  the desired  estimates.
  \vskip .1in
  (4) In $D_4$, we have
  $$
  \lambda_+-\lambda_-\in (\frac{|\xe|^2}{\sqrt{2}},|\xe|^2),
  \ \lambda_-\in (-|\xe|^2,-\frac{|\xe|^2}{2}),\
  \lambda_+\in (-\frac{2\xi^2}{|\xe|^2},-\frac{\xi^2}{|\xe|^2}),
  \ -\frac{\xi^2}{|\xe|^2}\ge -\frac{|\xe|^2}{2},
  $$
 yielding $e^{\lambda_\pm t}\le e^{-\frac{\xi^2}{|\xe|^2}t}$. So we have
  $$
  \begin{aligned}
  |\widehat{M_{1}}(\xe,t)|\lesssim&\ \frac{|\xi|}{|\xe|^2}(e^{\lambda_+t}+e^{\lambda_-t})
  \lesssim\  \frac{|\xi|}{|\xe|^2}e^{-\frac{\xi^2}{|\xe|^2}t},\\
  |\widehat{M_{2}}(\xe,t)|\lesssim&\ e^{\lambda_-t}
  +|\lambda_-|\frac{|e^{\lambda_+t}-e^{\lambda_-t|}}
  {|\lambda_+-\lambda_-|}
  \lesssim\ e^{-\frac{\xi^2}{|\xe|^2}t},\\
  |\widehat{M_{3}}(\xe,t)|\lesssim&\ e^{\lambda_-t}
  +|\lambda_+|\frac{|e^{\lambda_+t}-e^{\lambda_-t|}}
  {|\lambda_+-\lambda_-|}
  \lesssim\ e^{-\frac{|\xe|^2}{2}t}+\frac{\xi^2}{|\xe|^4}
  e^{-\frac{\xi^2}{|\xe|^2}t}.
  \end{aligned}
  $$
\end{proof}
\begin{rem}\label{rpp}
Since the estimates of $\widehat{M_{i}}(\xe,t)$ $(i=1,2,3)$ in $D_1$, $D_2$ and $D_3$  are  similar, we will give the detailed estimate on $D_1$ and omit the details of the estimate on $D_2\cup D_3$.
\end{rem}
\section{Decay estimates and Nonlinear estimates}
\label{sdecay1}
\vskip.1in
\subsection{Decay estimate I}
\begin{lemma}\label{l4.1}
Let $k\ge0$, $c>0$ and $0<\alpha\le1$, there holds
\begin{equation} \label{decay1}
\begin{aligned}
1)\ \||\na|^k e^{ct\De}f\|_{\F L^2(D_1)}
\lesssim&\ \min\big\{\langle t\rangle^{-\frac{k+1}{2}}
\|f\|_{L^1}, \langle t\rangle^{-\frac{2k+1}{4}}\|f\|_{L^1_y(L^2_x)},
\langle t\rangle^{-\frac{k}{2}} \|f\|_{L^2}\big\}; \\
2)\ \||\na|^k e^{ct\De}f\|_{\F L^2(D_4)}
\lesssim&\  \min\big\{\langle t\rangle^{-\frac{k}{2}} \|f\|_{H^k}, \langle t\rangle^{-\frac{k+1}{2}}(\|f\|_{L^1}+\||\na|^k f\|_{L^2})\}; \\
3)\ \||\na|^k e^{ct\De}f\|_{\F L^1(D_1)}
\lesssim&\ \min\big\{\langle t\rangle^{-\frac{k+2}{2}}
\|f\|_{L^1}, \langle t\rangle^{-\frac{k+2}{2}}
\||\na|^{-1}f\|_{L^2},\\
&\ \langle t\rangle^{-\frac{2k+3}{4}}
\|f\|_{L^1_y(L^2_x)},  \langle t\rangle^{-\frac{k+1}{2}}
\|f\|_{L^2},\langle t\rangle^{-\frac{k}{2}} \|f\|_{\F L^1}\big\}; \\
4)\ \||\na|^k e^{ct\De}f\|_{\F L^1(D_4)}
\lesssim&\ \langle t\rangle^{-\frac{k}{2}}\min\big\{\langle t\rangle^{-1}
(\|f\|_{L^1}+ \||\na|^kf\|_{\F L^1}),\\
&\langle t\rangle^{-\frac{1}{2}}(\|f\|_{L^2}+ \||\na|^kf\|_{\F L^1}), \|\widehat{f}\|_{L^1}+\||\na|^kf\|_{\F L^1}\big\};\\
5)\ \||\na|^{-\alpha} e^{ct\De}f\|_{\F L^1(D_1)}
\lesssim&\ \LHH^{-\frac{2-\alpha}{2}}\|f\|_{L^1},
\end{aligned}
\end{equation}
where $D_1$ and $D_4$ are defined by (\ref{domain}).
\end{lemma}
\begin{proof}
It suffices to consider the case $k=0$, since we can easily get
$$\| |\na|^ke^{ct\De}f\|_{\F L^r(D_1)}
\lesssim\ \LHH^\frac{k}{2}\|e^{\frac{c}{2}t\De}f\|_{\F L^r(D_1)},$$
 $$\| |\na|^ke^{ct\De}f\|_{\F L^r(D_4)}
\lesssim\ \LHH^\frac{k}{2}(\|e^{\frac{c}{2}t\De}f\|_{\F L^r(D_4)}+\||\na|^kf\|_{\F L^r}).$$
If we can show (\ref{decay1})$_5$ and
\begin{equation}\label{q1}
\| e^{ct\De}f\|_{\F L^r(D_1)}
\lesssim\ \LHH^{\frac{1}{4}-\frac{1}{r}}\|f\|_{L^1_y(L^2_x)},\ 1\le r\le2,
\end{equation}
then other estimates can be proved  by using (\ref{Heat}),  H\"{o}lder's inequality, Plancherel's identity and (\ref{q1}). Next, we focus on the estimates of (\ref{decay1})$_5$ and (\ref{q1}).
Using polar coordinate
$$\xi=r\cos\theta,\ \eta=r\sin\theta,\ \theta\in [0,2\pi],$$
we have
$$
\begin{aligned}
\| |\na|^{-\alpha}e^{ct\De}f\|_{\F L^1(D_1)}
\lesssim&\  \int_{\theta\in [0,2\pi]}d\theta\int_{0\le r\lesssim1} e^{-ctr^2}r^{1-\alpha}|\widehat{f}(r,\theta)|dr\\
\lesssim&\ \min\{1,t^{-\frac{1-\alpha}{2}}\int_{0\le r\lesssim1} e^{-\frac{c}{2}tr^2}dr\}\|\widehat{f}\|_{L^\infty}\\
\lesssim&\ \LHH^{-\frac{2-\alpha}{2}}\|f\|_{L^1}.
\end{aligned}
$$
By $|\xe|^2=\xi^2+\eta^2$ and H\"{o}lder's inequality,
$$
\begin{aligned}
\| e^{ct\De}f\|_{\F L^r(D_1)}
\lesssim&\ \|e^{-ct\eta^2}e^{-ct\xi^2}\widehat{f}\|_{L^r(D_1)}\\
\lesssim&\ \|e^{-ct\eta^2}\|_{L^r_\eta(|\eta|\lesssim 1)}\|e^{-ct\xi^2}\|_{L^\frac{2r}{2-r}_\xi(|\xi|\lesssim 1)}\|\widehat{f}\|_{L^2_\xi(L^\infty_\eta)}\\
\lesssim&\ \LHH^{\frac{1}{4}-\frac{1}{r}}\|f\|_{L^1_y(L^2_x)}.
\end{aligned}
$$
Thanks to the above estimates, we conclude the estimates of (\ref{decay1})$_5$ and (\ref{q1}).
\end{proof}
\subsection{Decay  estimate II}
\begin{lemma}\label{l4.2}
Let $1\le r\le2$, $1\le q\le 2$ and $1/p+1/q>1$. For all $k\ge 0$ and $\delta>0$, there holds
\begin{equation}\label{D21}
\begin{aligned}
 1)\ \mathcal{I}_1=&
 \|\G_{k,k}
 e^{-\G_{2,2}t}\widehat{f}\|_{L^r
  (D_4)}
  \lesssim\  \LHH^{-\frac{k}{2}}\min\big\{
  \|\widehat{f}\|_{L^r},\langle t\rangle^{-\frac{1}{2r}}\|\langle \na\rangle^{\frac{2}{r}-\frac{1}{2}
  +\delta}f\|_{L^1_x(L^2_y)}\big\};\\
2)\ \mathcal{I}_2=&
\|\G_{k,k+1}
 e^{-\G_{2,2}t}\widehat{f}\|_{L^1
  (D_{4})}\lesssim\ \langle t\rangle^{-\frac{k+1}{2}}
  (\|\langle \na\rangle^{\frac{3}{2}+\delta} f\|_{L^1_x(L^2_y)}+\|f\|_{L^1});\\
 3)\ \mathcal{I}_2'=&
\|\G_{k+1,k+2}
 e^{-\G_{2,2}t}\widehat{f}\|_{L^1
  (D_{4})}\lesssim\ \langle t\rangle^{-\frac{k+1}{2}}
  (\|\widehat{f}\|_{L^1}+\|f\|_{L^2});\\
 4)\ \mathcal{I}_3=&
\|\G_{k+1,k+2}
e^{-\G_{2,2}t}\widehat{f}\|_{L^1
  (D_{42})}\lesssim\ \langle t\rangle^{-\frac{k+1}{2}}\min\{
\|\widehat{f}\|_{L^1}
  +\|f\|_{L^2},\ \langle t\rangle^{-\frac{1}{2}}\|f\|_{L^1}\};\\
 5)\  \mathcal{I}_4=&
 \|\G_{k+2,k+4}
 e^{-\G_{2,2}t}\widehat{f}\|_{L^1
  (D_4)}\lesssim\ \langle t\rangle^{-\frac{k+2}{2}}(\|\widehat{f}\|_{L^1}+
  \|f\|_{L^1_y(L^\frac{4}{3}_x)});\\
6)\ \mathcal{I}_5=&
\|\G_{k+1,k+3}
e^{-\G_{2,2}t}\widehat{f}\|_{L^1
  (D_{42})}\\
  &\lesssim\
  \min\{ t^{-\frac{k+1}{2}+\frac{1}{2p}-\frac{1}{2q}}\|f
 \|_{L^1_y(L^q_x)},
 \|f\|_{L^1},\LHH^{-\frac{k+2}{2}}\||\na|^{-\delta}f\|_{L^1}
 \};\\
 7)\ \mathcal{I}_6=&
 \|\G_{k+1,k+2}
 e^{-\G_{2,2}t}\widehat{f}\|_{L^2
  (D_4)}\lesssim\ \langle t\rangle^{-\frac{k+1}{2}}
(\|f\|_{L^1}+\|f\|_{L^2}).
  \end{aligned}
\end{equation}
where $\mathcal{G}_{k,l}=\frac{|\xi|^k}{|\xe|^l}$, $D_4$ and $D_{42}$ are defined by (\ref{domain}) and (\ref{D4f}).
\end{lemma}
\begin{proof}
\underline{(\ref{D21})$_1$}
Using
$A^ke^{-A^2t}\lesssim t^{-\frac{k}{2}},\ \forall\ A>0,$
it is easy to get
\begin{equation}\label{q6}
\mathcal{I}_1\lesssim\ \langle t\rangle^{-\frac{k}{2}}
\|e^{-\frac{\xi^2}{2|\xe|^2}t}\widehat{f}\|_{L^r
  (D_4)}\lesssim\ \langle t\rangle^{-\frac{k}{2}}
\|\widehat{f}\|_{L^r
  }.
  \end{equation}
Using dyadic decomposition and H\"{o}lder's inequality, we have
$$
\begin{aligned}
\|e^{-\frac{\xi^2}{2|\xe|^2}t}\widehat{f}\|_{L^r
  (D_4)}
  \lesssim&\ \sum_{M\ge 0}
  \|e^{-c\xi^2M^{-2}t}\widehat{P_Mf}\|_{L^r}
  \lesssim\ \sum_{M\ge 0}\|e^{-c\xi^2M^{-2}t}\|_{L^r_\xi}
  \|\widehat{P_Mf}\|_{L^r_\eta(L^\infty_\xi)}\\
  \lesssim&\ t^{-\frac{1}{2r}}\sum_{M\ge 0}M^\frac{1}{r}
  \|\widehat{P_Mf}\|_{L^2_\eta(L^\infty_\xi)}(\int_{|\eta|\lesssim
  M}d\eta)^\frac{2-r}{2r}\\
  \lesssim&\ t^{-\frac{1}{2r}}\sum_{M\ge 0}M^{\frac{2}{r}-\frac{1}{2}}
  \|\widehat{P_Mf}\|_{L^2_\eta(L^\infty_\xi)}\\
  \lesssim&\
  t^{-\frac{1}{2r}}\Big(\sum_{M\ge 1}M^{\frac{2}{r}-\frac{1}{2}}
  \langle M\rangle^{-\frac{2}{r}+\frac{1}{2}-\delta}
  \|\langle|\xe|\rangle^{\frac{2}{r}-\frac{1}{2}+\delta}
  \widehat{f}\|_{L^2_\eta(L^\infty_\xi)}\\
  &+\sum_{0\le M<1} M^{\frac{2}{r}-\frac{1}{2}}\|\widehat{f}\|_{L^2_\eta(L^\infty_\xi)}\Big)\\
  \lesssim&\
  t^{-\frac{1}{2r}}\|\langle\na\rangle^{\frac{2}{r}-\frac{1}{2}+\delta}f\|_{L^1_x(L^2_y)}
  \Big(\sum_{M\ge 1}
  M^{-\delta}
  +\sum_{0\le M<1}M^{\frac{2}{r}-\frac{1}{2}}\Big)\\
  \lesssim&\ t^{-\frac{1}{2r}}\|\langle\na \rangle^{\frac{2}{r}-\frac{1}{2}
  +\delta}f\|_{L^1_x(L^2_y)}.
\end{aligned}
$$
So
$$
 \mathcal{I}_1 \lesssim\  t^{-\frac{k}{2}-\frac{1}{2r}}\|\langle \na\rangle^{\frac{2}{r}-\frac{1}{2}
  +\delta}f\|_{L^1_x(L^2_y)},
$$
which, along with (\ref{q6}) and
$\|\widehat{f}\|_{L^r}
\lesssim\ \|\langle \na\rangle^{\frac{2}{r}-\frac{1}{2}
  +\delta}f\|_{L^1_x(L^2_y)}$
 yields (\ref{D21})$_1$.\\
\underline{(\ref{D21})$_2$} Thanks to (\ref{D4f}),
 we have
$$\mathcal{I}_2\le\ \|\G_{k,k+1}
 e^{-\frac{\xi^2}{|\xe|^2}t}\widehat{f}\|_{L^1
  (D_{41})}+\|\G_{k,k+1}
 e^{-\frac{\xi^2}{|\xe|^2}t}\widehat{f}\|_{L^1
  (D_{42})}=\mathcal{I}_{21}+\mathcal{I}_{22}.$$
By (\ref{D21})$_1$ for   $r=1$, one has
$$\mathcal{I}_{21}\lesssim
\|\G_{k,k}
 e^{-\frac{\xi^2}{|\xe|^2}t}\widehat{f}\|_{L^1
  (D_{41})}
  \lesssim \langle t\rangle^{-\frac{k+1}{2}}\|\langle\na\rangle^{\frac{3}{2}+\delta}f\|_{L^1_x(L^2_y)}.$$
By dyadic decomposition and H\"{o}lder's inequality, we infer
$$
\begin{aligned}
 \mathcal{I}_{22}
\lesssim&\ \sum_{M\lesssim1}
\||\xi|^{k}M^{-k-1}e^{-c\xi^2M^{-2}t}\widehat{P_Mf}\|_{L^1(D_{42})}\\
\lesssim&\ t^{-\frac{k}{2}}\sum_{M\lesssim1}
M^{-1}\|e^{-\frac{c}{2}\xi^2M^{-2}t}
\widehat{P_Mf}\|_{L^1(D_{42})}\\
\lesssim&\ t^{-\frac{k}{2}}\sum_{M\lesssim1}
M^{-1}\|e^{-\frac{c}{2}\xi^2M^{-2}t}\|_{L^1_\xi}\|
\widehat{P_Mf}\|_{L^\infty}\int_{|\eta|\lesssim M}d\eta\\
\lesssim&\ t^{-\frac{k+1}{2}}\|\widehat{f}\|_{L^\infty}
\sum_{M\lesssim1}M
\lesssim\ t^{-\frac{k+1}{2}}\|f\|_{L^1},
\end{aligned}
$$
which, together with
$\mathcal{I}_{22}\lesssim \|\widehat{f}\|_{L^\infty}\||\xe|^{-1}
\|_{L^1(|\xe|<1)}\lesssim\ \|f\|_{L^1}$
leads
\begin{equation}\label{I22}
\mathcal{I}_{22}
\lesssim\ \langle t\rangle^{-\frac{k+1}{2}}
\|f\|_{L^1}.
\end{equation}
Combining with the estimates of  $\mathcal{I}_{21}$ and (\ref{I22}) can yield
the desired result.\\
\underline{(\ref{D21})$_3$} Using (\ref{D4f}) again,
$$\mathcal{I}_2'\le\ \|\G_{k+1,k+2}
 e^{-\frac{\xi^2}{|\xe|^2}t}\widehat{f}\|_{L^1
  (D_{41})}+\|\G_{k+1,k+2}
 e^{-\frac{\xi^2}{|\xe|^2}t}\widehat{f}\|_{L^1
  (D_{42})}=\mathcal{I}_{21}'+\mathcal{I}_{22}'.$$
It is easy to obtain
$$\mathcal{I}_{21}'
\le\ \|\G_{k+1,k+1}
 e^{-\frac{\xi^2}{|\xe|^2}t}\widehat{f}\|_{L^1
  (D_{41})}\le \langle t\rangle^{-\frac{k+1}{2}}\|\widehat{f}\|_{L^1}.$$
Using dyadic decomposition, $|\xi|\lesssim |\xe|^2$ in $D_4$ and H\"{o}lder's inequality, we have
$$
\begin{aligned}
\mathcal{I}_{22}'
\lesssim&\ \sum_{M\lesssim1}
\||\xi|^{k+1}M^{-k-2}e^{-c\xi^2M^{-2}t}\widehat{P_Mf}\|_{L^1(D_{42})}\\
\lesssim&\ t^{-\frac{k+1}{2}+\frac{1}{6}}
\sum_{M\lesssim1}
\||\xi|^\frac{1}{3}
M^{-\frac{4}{3}}
e^{-\frac{c}{2}
\xi^2M^{-2}t}\widehat{P_Mf}\|_{L^1(D_{42})}\\
\lesssim&\
t^{-\frac{k+1}{2}+\frac{1}{6}}
\sum_{M\lesssim1} M^{-\frac{2}{3}}
\|e^{-\frac{c}{2}
\xi^2M^{-2}t}\|_{L^2_\xi}\|\widehat{P_Mf}\|_{L^2}
(\int_{|\eta|\lesssim M}d\eta)^\frac{1}{2}\\
\lesssim&\
t^{-\frac{k+1}{2}-\frac{1}{12}}\sum_{M\lesssim1}M^{-\frac{2}{3}+\frac{1}{2}
+\frac{1}{2}}\|P_M f\|_{L^2}\\
\lesssim&\ t^{-\frac{k+1}{2}-\frac{1}{12}}\|f\|_{L^2}.
\end{aligned}
$$
Due to $\mathcal{I}_{22}'\lesssim \|\widehat{f}\|_{L^1}$, one has
$$\mathcal{I}_{22}'\lesssim \langle t\rangle^{-\frac{k+1}{2}}
(\|f\|_{L^2}+\|\widehat{f}\|_{L^1}).$$
As a result, we deduce the desired bound
by combining the estimate of $\mathcal{I}_{21}'$.\\
\underline{ (\ref{D21})$_4$}
 Since the estimate of $\mathcal{I}_3$ is similar to the estimate of $\mathcal{I}_{22}'$, it suffices to prove
 \begin{equation}\label{adi}
\mathcal{I}_3\lesssim \langle t\rangle^{-\frac{k+2}{2}}\|f\|_{L^1}
,\ k\ge -1.
\end{equation}
Using dyadic decomposition and H\"{o}lder's inequality, we have
$$
\begin{aligned}
\mathcal{I}_3
\lesssim&\ t^{-\frac{k+1}{2}}\sum_{M\lesssim1} M^{-1}
\|e^{-\frac{c}{2}\xi^2M^{-2}t}\widehat{P_Mf}\|_{L^1(D_{42})}\\
\lesssim&\ t^{-\frac{k+1}{2}}\sum_{M\lesssim1} M^{-1}
\|e^{-\frac{c}{2}\xi^2M^{-2}t}\|_{L^1_\xi}
\|\widehat{P_Mf}\|_{L^\infty_{\xi,\eta}}\int_{|\eta|\lesssim M}d\eta\\
\lesssim&\ t^{-\frac{k+2}{2}}\sum_{M\lesssim1} M^{-1+1+1}\|P_Mf\|_{L^1}
\lesssim\ t^{-\frac{k+2}{2}}\|f\|_{L^1},
\end{aligned}
$$
which,
along with $\mathcal{I}_3\lesssim\ \|\widehat{f}\|_{L^\infty}\lesssim\ \|f\|_{L^1}$
can lead (\ref{adi}).\\
\underline{ (\ref{D21})$_5$ and (\ref{D21})$_6$}
Like the previous arguments, we have
$$\mathcal{I}_4\le\ \|\G_{k+2,k+4}
 e^{-\frac{\xi^2}{|\xe|^2}t}\widehat{f}\|_{L^1
  (D_{41})}+\|\G_{k+2,k+4}
 e^{-\frac{\xi^2}{|\xe|^2}t}\widehat{f}\|_{L^1
  (D_{42})}=\mathcal{I}_{41}+\mathcal{I}_{42}.$$
Using $|\xe|\ge 1$ in $D_{41}$, we can get
$$\mathcal{I}_{41}\le \|\G_{k+2,k+2}e^{-\frac{\xi^2}{|\xe|^2}t}\widehat{f}\|_{L^1}
\le \langle t\rangle^{-\frac{k+2}{2}}\|\widehat{f}\|_{L^1}.$$
 By $|\xi|\lesssim |\xe|^2$ in $D_{4}$,
 dyadic   decomposition, H\"{o}lder's inequality,
 $1\le q\le 2$ and $1/p+1/q>1$,  one has
\begin{equation}\label{q7}
\begin{aligned}
\mathcal{I}_{42}
\lesssim&\ \sum_{M\lesssim 1}
\||\xi|^{k+2}M^{-k-4}e^{-c\xi^2M^{-2}t}\widehat{P_Mf}
\|_{L^1(D_{42})}\\
\lesssim&\ \sum_{M\lesssim 1}
\|(\xi^2M^{-2})^{\frac{k+2}{2}-\frac{1}{2p}}
(\xi^2M^{-2})^\frac{1}{2p}M^{-2}
e^{-c\xi^2M^{-2}t}\widehat{P_Mf}
\|_{L^1(D_{42})}\\
\lesssim&\ t^{-\frac{k+2}{2}+\frac{1}{2p}}
\sum_{M\lesssim1} M^{-2+\frac{1}{p}}
\|e^{-\frac{c}{2}\xi^2M^{-2}t}\widehat{P_Mf}
\|_{L^1(D_{42})}\\
\lesssim&\ t^{-\frac{k+2}{2}+\frac{1}{2p}}
\sum_{M\lesssim1} M^{-2+\frac{1}{p}}
\|e^{-\frac{c}{2}\xi^2M^{-2}t}\|_{L^q_\xi}
\|\widehat{P_Mf}
\|_{L^{\frac{q}{q-1}}_\xi(L^\infty_\eta)}\int_{|\eta|\lesssim M}d\eta\\
\lesssim&\ t^{-\frac{k+2}{2}+\frac{1}{2p}-\frac{1}{2q}}
\sum_{M\lesssim1}M^{-1+\frac{1}{p}+\frac{1}{q}}
\|P_Mf\|_{L^1_y(L^q_x)}\\
\lesssim&\ t^{-\frac{k+2}{2}+\frac{1}{2p}-\frac{1}{2q}}\|f\|_{L^1_y(L^q_x)},\ (k\ge -1),
\end{aligned}
\end{equation}
which,
together with $\mathcal{I}_{42}\le \|\widehat{f}\|_{L^1}$, yield
\begin{equation}\label{q2}
\mathcal{I}_{42}\le\ \min\{\|\widehat{f}\|_{L^1},
t^{-\frac{k+2}{2}+\frac{1}{2p}-\frac{1}{2q}}\|f\|_{L^1_y(L^q_x)}\}.
\end{equation}
Using  (\ref{q2}) for $p=q=\frac{4}{3}$ and  the estimate of
$\mathcal{I}_{41}$, we can get the desired estimate (\ref{D21})$_5$.
It follows from using
 $|\xi|\lesssim|\xe|^2$ that
$\mathcal{I}_5\lesssim\ \||\xe|^{-1}\widehat{f}\|_{L^1(D_{42})}
\lesssim\ \|f\|_{L^1}$. Notice that (\ref{q7}) holds for $k\ge -1$, so we can obtain
$$\mathcal{I}_5\lesssim\ t^{-\frac{k+1}{2}+\frac{1}{2p}-\frac{1}{2q}}\|f\|_{L^1_y(L^q_x)},
\ k\ge 0.$$
Finally, to complete the estimate of (\ref{D21})$_6$, it suffices to prove
\begin{equation}\label{q4}
 \mathcal{I}_5\lesssim\ \LHH^{-\frac{k+2}{2}}\||\na|^{-\delta}f\|_{L^1}.
 \end{equation}
It is easy to get
  \begin{equation}\label{q5}
\mathcal{I}_5\lesssim\ \LHH^{-\frac{k+1}{2}}
\|\frac{1}{|\xe|^2}e^{-\frac{\xi^2}{2|\xe|^2}t}
\widehat{f}\|_{L^1(D_{42})}.
\end{equation}
Using dyadic decomposition and H\"{o}lder's inequality, one obtains
$$
\begin{aligned}
\|\frac{1}{|\xe|^2}e^{-\frac{\xi^2}{2|\xe|^2}t}
\widehat{f}\|_{L^1(D_{42})}\lesssim&
\sum_{M\lesssim 1}M^{-2}
\|e^{-\frac{c}{2}\xi^2M^{-2}t}\widehat{P_Mf}\|_{L^1(D_{42})}\\
\lesssim&\ \sum_{M\lesssim 1}M^{-2}\|e^{-\frac{c}{2}\xi^2M^{-2}t}\|_{L^1_\xi}
\|\widehat{P_Mf}\|_{L^\infty_{\xi,\eta}}\int_{|\eta|\lesssim M}d\eta\\
\lesssim&\ t^{-\frac{1}{2}}\sum_{M\lesssim1}\|P_Mf\|_{L^1}
\lesssim\ t^{-\frac{1}{2}}\||\na|^{-\delta}f\|_{L^1},
\end{aligned}
$$
which, along with
$$\|\frac{1}{|\xe|^2}e^{-\frac{\xi^2}{2|\xe|^2}t}
\widehat{f}\|_{L^1(D_{42})}
\le \|\frac{1}{|\xe|^2}\widehat{f}\|_{L^1(D_{42})}
\lesssim\ \|\frac{1}{|\xe|^{2-\delta}}\|_{L^1(D_{42})}
\||\xe|^{-\delta}\widehat{f}\|_{L^\infty(D_{42})}
\lesssim \||\na|^{-\delta}f\|_{L^1}$$
follows
$$\|\frac{1}{|\xe|^2}e^{-\frac{\xi^2}{2|\xe|^2}t}\widehat{f}
\|_{L^1(D_{42})}\lesssim
\ \LHH^{-\frac{1}{2}}\||\na|^{-\delta}f\|_{L^1}.$$
Hence   by (\ref{q5}), we can get  the  desired estimate.\\
\underline{ (\ref{D21})$_7$}
Like the previous process, we have
$$\mathcal{I}_6\le\ \|\G_{k+1,k+1}
 e^{-\frac{\xi^2}{|\xe|^2}t}\widehat{f}\|_{L^2
  (D_{41})}+\|\G_{k+1,k+2}
 e^{-\frac{\xi^2}{|\xe|^2}t}\widehat{f}\|_{L^2
  (D_{42})}=\mathcal{I}_{61}+\mathcal{I}_{62}.$$
 $\mathcal{I}_{61}$ can be easily bounded by $C\LHH^{-\frac{k+1}{2}}\|f\|_{L^2}$. For $\mathcal{I}_{62}$, using dyadic decomposition and H\"{o}lder's inequality, we have
 $$
 \begin{aligned}
 \mathcal{I}_{62}\lesssim&\ \sum_{M\lesssim1}\||\xi|^{k+1}
 M^{-k-2}e^{-c\xi^2M^{-2}t}\widehat{P_Mf}\|_{L^2(D_{42})}\\
 \lesssim&\ t^{-\frac{k}{2}}
 \sum_{M\lesssim1}\|(|\xi|M^{-1})^\frac{1}{2}
 (|\xi|M^{-3})^\frac{1}{2}
 e^{-\frac{c}{2}\xi^2M^{-2}t}\widehat{P_Mf}\|_{L^2(D_{42})}\\
 \lesssim&\ t^{-\frac{2k+1}{4}} \sum_{M\lesssim1}M^{-\frac{1}{2}}
 \| e^{-\frac{c}{2}\xi^2M^{-2}t}\|_{L^2_\xi}\|\widehat{P_Mf}
 \|_{L^\infty}(\int_{|\eta|\lesssim M}d\eta)^\frac{1}{2}\\
 \lesssim&\ t^{-\frac{k+1}{2}}\sum_{M\lesssim1}M^\frac{1}{2}\|P_M f\|_{L^1}\lesssim\ t^{-\frac{k+1}{2}}\|f\|_{L^1},
 \end{aligned}
 $$
 which, together with the estimate of $\mathcal{I}_{61}$ implies
 $$\mathcal{I}_{6}\lesssim\ t^{-\frac{k+1}{2}}
 (\|f\|_{L^1}+\|f\|_{L^2}).$$
In addition,  $\mathcal{I}_{6}\lesssim\ \|f\|_{L^2}$. So
we can get the desired estimate.
\end{proof}
\subsection{Decay  estimate III}
Thanks to Lemma \ref{l4.1} and Lemma \ref{l4.2}, we can get the following lemmas, the detailed proofs of which  are showed in the Appendix.
 \begin{lemma}\label{M1}
 Let $M_1(\p,t)$ be the operator defined by (\ref{MM}), $1\le r\le2$, $\delta>0$, then there holds
 \begin{equation}\label{End1}
 \begin{aligned}
 1)\ &\|M_1(\p,t) f\|_{\F L^2(D_4)}
 \lesssim\ \LHH^{-\frac{1}{2}}\min\{\|f\|_{L^1\cap L^2},\||\na|^{-1}f\|_{L^2}\};\\
 2)\  &\||\na|M_1(\p,t) f\|_{\F L^r(D_4)}
 \lesssim\ \LHH^{-\frac{1}{2}}\min\{\|\widehat{f}\|_{L^r},\LHH^{-\frac{1}{2r}}
 \|\langle\na\rangle^{\frac{2}{r}-\frac{1}{2}+\delta}f\|_{L^1_x(L^2_y)}\};\\
 3)\ &\|\p_xM_1(\p,t) f\|_{\F L^r(D_4)}
 \lesssim\ \LHH^{-1}\min\{\|\widehat{f}\|_{L^r},\LHH^{-\frac{1}{2r}}
 \|\langle\na\rangle^{\frac{2}{r}-\frac{1}{2}+\delta}f\|_{L^1_x(L^2_y)}\};\\
4)\ &\|M_1(\p,t) f\|_{\F L^1(D_4)}
 \lesssim\ \LHH^{-1}(
 \|\langle\na\rangle^{1.51}f\|_{L^1_x(L^2_y)}+\|f\|_{L^1});\\
 5)\ &\||\na|^{-1}M_1(\p,t) f\|_{\F L^1(D_{42})}
 \lesssim\ \LHH^{-\frac{1}{2}}\min\{\|f\|_{L^1}+\|f\|_{L^1_y(L^\frac{4}{3}_x)},\LHH^{-\frac{1}{2}}
 \||\na|^{-\delta}f\|_{L^1}\};\\
 6)\ &\|\mathcal{R}_1M_1(\p,t) f\|_{\F L^1(D_{42})}
 \lesssim\ \LHH^{-\frac{3}{2}}\|f\|_{L^1};\\
 7)\ &\|\mathcal{R}_1M_1(\p,t) f\|_{\F L^1(D_4)}
 \lesssim\ \LHH^{-1}
 (\|\widehat{f}\|_{L^1}+\|f\|_{L^2})\\
 8)\ &\||\na|\mathcal{R}_1M_1(\p,t) f\|_{\F L^1(D_4)}
 +\|\mathcal{R}_{11}M_1(\p,t) f\|_{\F L^1(D_4)}
 \lesssim\ \LHH^{-1}\|\widehat{f}\|_{L^1};\\
 9)\ &\|\p_x\mathcal{R}_1M_1(\p,t) f\|_{\F L^1(D_4)}
+\||\na|\mathcal{R}_{11}M_1(\p,t) f\|_{\F L^1(D_4)}
 \lesssim\ \LHH^{-\frac{3}{2}}\|\widehat{f}\|_{L^1},
 \end{aligned}
 \end{equation}
 where  $D_4$ and $D_{42}$ are defined by (\ref{domain}) and (\ref{D4f}).
 \end{lemma}

 \begin{lemma}\label{M2}
 Let $M_2(\p,t)$ be the operator defined by (\ref{MM}),  $1\le r\le2$, $\delta>0$, then there holds
\begin{equation}\label{End2}
\begin{aligned}
 1)\ & \|M_2(\p,t) f\|_{\F L^r(D_4)}
 \lesssim\ \LHH^{-\frac{1}{2r}}
 \|\langle\na\rangle^{\frac{2}{r}-\frac{1}{2}+\delta}f\|_{L^1_x(L^2_y)};\\
 2)\ &\|\p_xM_2(\p,t) f\|_{\F L^r(D_4)}
 \lesssim\ \LHH^{-\frac{1}{2}}\min\{\|\na f\|_{\F L^r},\LHH^{-\frac{1}{2r}}
 \|\langle\na\rangle^{\frac{2}{r}+\frac{1}{2}+\delta}f\|_{L^1_x(L^2_y)}\};\\
  3)\ & \|\p_x^2M_2(\p,t) f\|_{\F L^2(D_4)}
 \lesssim\ \LHH^{-1}\|\na^2f\|_{L^2};\\
  4)\ & \||\na|^{-1}M_2(\p,t) f\|_{\F L^1(D_{42})}
 \lesssim\ \LHH^{-\frac{1}{2}}\|f\|_{L^1};\\
  5)\ & \|\mathcal{R}_1^lM_2(\p,t) f\|_{\F L^1(D_4)}
 \lesssim\ \LHH^{-\frac{l}{2}}\min\{\|\widehat{f}\|_{L^1},
 \LHH^{-\frac{1}{2}}
 \|\langle\na\rangle^{1.51}f\|_{L^1_x(L^2_y)}\}\ (l=1,2);\\
  6)\ & \|\p_x\mathcal{R}_1M_2(\p,t) f\|_{\F L^1(D_4)}
 \lesssim\ \LHH^{-1}
 \|\na f\|_{\F L^1},
\end{aligned}
\end{equation}
 where $\mathcal{R}_1^1=\mathcal{R}_1$, $\mathcal{R}_1^2=\mathcal{R}_{11}$
 and
 $D_4$ and $D_{42}$ are defined by (\ref{domain}) and (\ref{D4f}).
\end{lemma}

 \begin{lemma}\label{M3}
 Let $M_3(\p,t)$ be the operator defined by (\ref{MM}),  $1\le r\le2$, then there holds
\begin{equation}\label{End3}
\begin{aligned}
1)\ &\|M_3(\p,t) f\|_{\F L^2(D_4)}
 \lesssim\ \LHH^{-\frac{1}{2}}\min\{
\||\na|^{-1}f\|_{H^1},\|f\|_{L^1\cap L^2}\};\\
2)\ &\|M_3(\p,t) f\|_{\F L^1(D_4)}
 \lesssim\ \LHH^{-1}\min\{
\|f\|_{ L^1\cap L^1_y(L^\frac{4}{3}_x)}+\|\widehat{f}\|_{L^1},\\
&\ \ \ \ \ \ \ \ \ \ \ \ \ \ \ \ \ \ \ \ \ \ \ \ \ \ \ \ \||\na|^{-1}f\|_{L^2\cap \F L^1}+\|\widehat{f}\|_{L^1}\};\\
3)\ &\|\p_yM_3(\p,t) f\|_{\F L^2(D_4)}
 \lesssim\ \min\big\{\LHH^{-\frac{1}{2}}\|f\|_{H^1},
\LHH^{-1}
\||\na|^{-1}f\|_{H^3},\LHH^{-1}\|f\|_{L^1\cap H^1}\big\};\\
4)\ &\|\p_xM_3(\p,t) f\|_{\F L^r(D_4)}
\lesssim\ \LHH^{-1}(\|f\|_{\F L^r}+\|\p_xf\|_{\F L^r});\\
5)\ & \|\p_xM_3(\p,t) f\|_{\F L^1(D_4)}
\lesssim\ \min\big\{\LHH^{-1}(\|\p_xf\|_{\F L^1\cap L^1_y(L^\frac{4}{3}_x)}+\|\widehat{f}\|_{L^1})
,\\
&\ \ \ \ \ \ \LHH^{-\frac{3}{2}}(\|f\|_{L^2}+\|\p_xf\|_{\F L^1}+\|\widehat{f}\|_{L^1}),
\LHH^{-\frac{3}{4}}(\|\p_xf\|_{L^1_y(L^2_x)}+\|\p_xf\|_{\F L^1}
)
\big\};\\
6)\ &\|\De M_3(\p,t) f\|_{\F L^2(D_4)}
\lesssim \ \LHH^{-1}\|f\|_{H^2}\\
7)\ &\|\p_x^2\mathcal{R}_1 M_3(\p,t) f\|_{\F L^1(D_4)}
\lesssim \ \LHH^{-\frac{5}{2}}(\|f\|_{\F L^1}+\|\p_x^2f\|_{\F L^1});\\
8)\ &\||\na|\p_x M_3(\p,t) f\|_{\F L^r(D_4)}
\lesssim \ \LHH^{-\frac{3}{2}}(\|f\|_{\F L^r}+\|\p_x\na f\|_{\F L^r});\\
9)\ &\|\p_x^2 M_3(\p,t) f\|_{\F L^1(D_4)}
\lesssim \ \LHH^{-2}(\|f\|_{\F L^1}+\|\p_x^2 f\|_{\F L^1});\\
10)\ &\|\mathcal{R}_1 M_3(\p,t) f\|_{\F L^1(D_4)}
\lesssim \ \LHH^{-1}(\||\na|^{-1}f\|_{\F L^1}+\|\widehat{f}\|_{L^1});\\
11)\ &\|\p_x\mathcal{R}_1 M_3(\p,t) f\|_{\F L^1(D_4)}
\lesssim \ \LHH^{-\frac{3}{2}}(\|\p_xf\|_{\F L^1}+\|\widehat{f}\|_{L^1}),
\end{aligned}
\end{equation}
where $D_4$ is defined by (\ref{domain}).
\end{lemma}
\subsection{Nonlinear decay estimate}
In this subsection, we give some Lemmas devoting to estimating the nonlinear part, the proofs of which are given in the Appendix. Let $\|V\|_3$ be the norm defined in section \ref{intro}, $P_\backsim$, $P_{\thickapprox}$ and $\vec{R'}$ be the operator defined by
\begin{equation}\label{ppp}
P_\backsim f=P_{\langle t\rangle^{-8}\le\cd \le 2\langle t\rangle^{-0.05}}f,
\end{equation}
\begin{equation}\label{Add}
P_{\thickapprox}f=P_{\langle \frac{t}{2}\rangle^{-8}\le\cd\le2\langle \frac{t}{2}\rangle^{-0.05}}f,
\end{equation}
\begin{equation}\label{rrrr}
\mathcal{R}_1'=-\p_x(-\De)^{-1},\ \mathcal{R}_2'=\p_y(-\De)^{-1},\ \vec{R'}=(\mathcal{R}_2',\mathcal{R}_1'),
\end{equation}
 respectively.
\begin{lemma}\label{ne1}
Let $(u,v,b,B)$ be sufficiently  smooth solution
 solving (\ref{mhd}), then
\begin{equation}\label{ine1}
\begin{aligned}
\|\na P_\backsim(\vec{R'}\cd\B) b\|_{L^1}+\|\B b\|_{L^1}\lesssim\ \LHH^{-\frac{1}{2}}\|V\|_3^2,\ \||\p_x|^\beta(bb)\|_{L^1}
\lesssim&\ \LHH^{-\frac{1+\beta}{2}}\|V\|_3^2;\\
\|\p_yv_{<\LHH^{-8}}b\|_{L^1}
+\|\p_yv_{>2\LHH^{-0.05}}b\|_{L^1}+
\|\p_xP_\backsim(\vec{R'}\cd\B)\p_yu\|_{L^1}
\lesssim&\ \LHH^{-\frac{5}{4}}\|V\|_3^2;\\
\||\na|^{0.99}(\U\otimes\U,\B B)\|_{L^1}
\lesssim&\ \LHH^{-0.6}\|V\|_3^2,
\end{aligned}
\end{equation}
where $\F \{|\p_x|^\beta f\}=|\xi|^\beta \widehat{f}$ and $0< \beta\le 1$.
\end{lemma}
\vskip.1in
\begin{lemma}\label{ne2}
Under the conditions in Lemma \ref{ne1},
let $1\le p\le 2$, then
\begin{equation}\label{ine2}
\begin{aligned}
\|b\p_xb\|_{L^1_y(L^p_x)}
\lesssim&\ \LHH^{-\frac{3}{2}+\frac{1}{2p}}\|V\|_3^2;\\
\|\langle\na\rangle^3(u\p_xb,\B\cd\na u)\|_{L^1_x(L^2_y)}
+\|\langle\na\rangle^3(\p_xP_\backsim(\vec{R'}\cd\B)\p_yu)
\|_{L^1_x(L^2_y)}&\\
\|\langle\na\rangle^2(\p_xP_\backsim u b)\|_{L^1_x(L^2_y)}+\|\langle\na\rangle^2(\p_xP_\backsim(\vec{R'}\cd\B)u)
\|_{L^1_x(L^2_y)}\lesssim&\ \LHH^{-1.01}\|V\|_3^2;\\
\|\langle\na\rangle^3(v_{<\LHH^{-8}}b,v_{>
2\LHH^{-0.05}}b)\|_{L^1_x(L^2_y)}&\\
+\|\langle\na\rangle^3(v_{<\LHH^{-8}}\p_yb,v_{>
2\LHH^{-0.05}}\p_yb)\|_{L^1_x(L^2_y)}&\\
+\|\langle\na\rangle^2(\p_yv_{<\LHH^{-8}}b,\p_yv_{>
2\LHH^{-0.05}}b)\|_{L^1_x(L^2_y)}
\lesssim&\ \LHH^{-1.01}\|V\|_3^2;\\
\|\langle\na\rangle^2(\U\cd\na\U,\B\cd\na\B)\|_{L^1_x(L^2_y)}
\lesssim&\ \LHH^{-0.75}\|V\|_3^2,\\
\|\langle\na\rangle^2(\p_xP_\backsim ub)\|_{L^1_x(L^2_y)}
\lesssim&\ \LHH^{-1.05}\|V\|_3^2.
\end{aligned}
\end{equation}
\end{lemma}
\begin{lemma}\label{ne3}
Under the conditions in Lemma \ref{ne1},
there holds
\begin{equation}\label{ine3}
\begin{aligned}
\|\U\otimes \B\|_{L^2}+\|\B B\|_{L^2}+
\|u\p_xb\|_{L^2}+\|b\p_xb\|_{L^2}+\|\B\cd\na u\|_{L^2}\lesssim&\ \LHH^{-\frac{5}{4}}\|V\|_3^2;\\
\|\U\otimes\U\|_{L^2}
\lesssim\ \LHH^{-\frac{3}{2}}\|V\|_3^2,\ \|\B b\|_{L^2}+\|\p_yP_\backsim(\vec{R'}\cd\B)b\|_{L^2}
\lesssim&\ \LHH^{-\frac{3}{4}}\|V\|_3^2;\\
\|(\U\cd\na\U,\B\cd\na \B)\|_{H^2}+\|(\U\cd\na b,\B\cd\na u)\|_{H^2}
\lesssim&\ \LHH^{-1.1}\|V\|_3^2;\\
\|v_{<\LHH^{-8}}b\|_{H^1}
+\|v_{>2\LHH^{-0.05}}b\|_{H^1}+
\|uB\|_{H^2}+\|bv\|_{H^2}+\|B\B\|_{H^2}&\\
+\|v_{<\LHH^{-8}}\p_yb\|_{L^2}
+\|v_{>2\LHH^{-0.05}}\p_yb\|_{L^2}+\|\p_xP_\backsim(\vec{R'}\cd \B)\p_yb\|_{L^2}&\\
\|\p_xP_\backsim(\vec{R'}\cd\B)b\|_{H^2}
+\|\p_x(\p_yP_\backsim(\vec{R'}\cd\B)b)\|_{H^1}+
\|P_\backsim(\vec{R'}\cd\B)\p_yu\|_{H^1}\lesssim&\ \LHH^{-1.1}\|V\|_3^2;\\
\|\p_x(P_\backsim(\vec{R'}\cd\B)\p_yb)\|_{H^1}+\|P_\backsim(\vec{R'}\cd\B)\p_xb\|_{H^2}
\lesssim&\ \LHH^{-1.01}\|V\|_3^2;\\
\|\p_x(\U\cd\na\U,\B\cd\na\B)\|_{H^2}
\lesssim\ \LHH^{-1}\|V\|_3^2,\ \|P_\backsim(\vec{R'}\cd\B)\p_yu\|_{H^3}
\lesssim&\ \LHH^{-0.9}\|V\|_3^2;\\
\|P_\backsim(\vec{R'}\cd\B)\p_yb\|_{L^2}
\lesssim&\ \LHH^{-0.6}\|V\|_3^2;\\
\end{aligned}
\end{equation}
\end{lemma}
\vskip.1in
\begin{lemma}\label{ne4}
Under the conditions in Lemma \ref{ne1},
there holds
\begin{equation}\label{ine4}
\begin{aligned}
\|[\U\otimes\U,\B\otimes\p_x\B,\B B,\p_x(\p_yP_\backsim(\vec{R'}\cd\B)b),\na P_\backsim(\vec{R'}\cd\B)u]\|_{\F L^1}\lesssim&\  \LHH^{-\frac{3}{2}}\|V\|_3^2;\\
\|\p_x(b\p_xu)\|_{\F L^1}+\|\p_y(b\p_xv,v\p_xb)\|_{\F L^1}\lesssim&\ \LHH^{-\frac{5}{4}}\|V\|_3^2;\\
\|P_\backsim(\vec{R'}\cd\B)b\|_{\F L^1}
+\|P_\thickapprox(\vec{R'}\cd\B)b\|_{\F L^1}
\lesssim&\ \LHH^{-0.99}\|V\|_3^2;\\
\|\p_x\p_y(P_\backsim(\vec{R'}\cd
\B)b)\|_{\F L^1}+\| \langle\na\rangle^2(P_\backsim(\vec{R'}\cd
\B)\p_yu)\|_{\F L^1}&\\
\|\langle\na\rangle\p_x(P_\backsim(\vec{R'}\cd
\B)b)\|_{\F L^1}+\|\langle\na\rangle(P_\backsim(\vec{R'}\cd
\B)u)\|_{\F L^1}\lesssim&\ \LHH^{-1.01}\|V\|_3^2;\\
\|\U\cd\na \B\|_{\F L^1}+\|\B\cd\na \U\|_{\F L^1}
+\|\p_x(b\p_xb)\|_{\F L^1}+\|u\p_x\B\|_{\F L^1}\lesssim&\
\LHH^{-1.3}\|V\|_3^2;\\
\|\na (B\B)\|_{\F L^1}+ \|\langle\na\rangle(\U\cd\na \U)\|_{\F L^1}+\|\langle\na\rangle(\U\otimes\U,\B B)\|_{\F L^1}\lesssim&\
\LHH^{-1.2}\|V\|_3^2;\\
 \|\langle\na\rangle^2(\B\cd\na \B)\|_{\F L^1}\lesssim\ \LHH^{-1.1}\|V\|_3^2,\ \|P_\backsim(\vec{R'}\cd
\B)\p_yb\|_{\F L^1}\lesssim&\ \LHH^{-0.9}\|V\|_3^2.
\end{aligned}
\end{equation}
\end{lemma}
\begin{rem}\label{r11}
We do not  focus on the optimal decay rate for the nonlinear terms like (\ref{ine2})$_2$, since it is sufficient to help us achieve the final goal in the present paper.
\end{rem}
\section{Energy estimate in $H^N$}
\label{s4}
\vskip.1in
In this section, we show the following energy  estimate of  solutions:
\begin{equation}\label{H8}
\begin{aligned}
&\ \|\vec{u}(t)\|_{H^N}^2+\|\vec{b}(t)\|_{H^N}^2+\|\na \U\|_{L^2_t(H^N)}^2
+\|\p_x \B\|_{L^2_t(H^{N-1})}^2\\
\lesssim&\ \|V_0\|_3^2+\|V\|_3^4+\|V\|_3^3,\ \forall\ t>0.
\end{aligned}
\end{equation}
Using
$$(\U\cd\na \U|\U)
=(\B\cd\na \B|\U)+(\B\cd\na \U|\B)=(\U\cd\na \B|\B)=0,$$
 we get the  $L^2$ energy estimate:
\begin{equation}\label{L2}
\ddt(\|\vec{u}\|_{L^2}^2+\|\vec{b}\|_{L^2}^2)+\|\na \vec{u}\|_{L^2}^2=0.
\end{equation}
The $\dot{H}^N$ estimate  of the solution reads:
\begin{equation}\label{H7}
\begin{aligned}
  \ddt(\|\vec{u}\|_{\dot{H}^N}^2+\|\vec{b}\|_{\dot{H}^N}^2)
  +\|\na \vec{u}\|_{\dot{H}^N}^2
  =&-(\vec{u}\cdot\na \vec{u}|\vec{u})_{\dot{H}^N}+(\vec{b}\cd\na \vec{b}|\vec{u})_{\dot{H}^N}\\
  &+(\vec{b}\cd\na \vec{u}|\vec{b})_{\dot{H}^N}-(\vec{u}\cd\na \vec{b}|\vec{b})_{\dot{H}^N}\\
  =&I_1+I_2+I_3+I_4.
\end{aligned}
\end{equation}
From (\ref{mhd})$_1$, using $(\na p|\p_x\B)=0$,
we have
$$(\p_x\vec{b}|\p_x\vec{b})_{H^{N-1}}
=-(\De\vec{u}|\p_x\vec{b})_{H^{N-1}}
-(\vec{b}\cd\na \vec{b}|\p_x\vec{b})_{H^{N-1}}
+(\U\cd\na\U|\p_x\B)_{H^{N-1}}+(\p_t\U|\p_x\B)_{H^{N-1}}.$$
Using the formation of (\ref{mhd})$_2$ and integration by parts, we can get
$$
\begin{aligned}
(\p_t\U|\p_x\B)_{H^{N-1}}
=&\ddt(\U|\p_x\B)_{H^{N-1}}+(\p_x\U|\p_t\B)_{H^{N-1}}\\
=&\ddt(\U|\p_x\B)_{H^{N-1}}+\|\p_x\U\|_{H^{N-1}}^2-(\p_x\U|\U\cd\na \B)_{H^{N-1}}+(\p_x\U|\B\cd\na \U)_{H^{N-1}}.
\end{aligned}
$$
So
$$
\begin{aligned}
\|\p_x\B\|_{H^{N-1}}^2
=&-(\De\U|\p_x\B)_{H^{N-1}}
-(\vec{b}\cd\na \vec{b}|\p_x\vec{b})_{H^{N-1}}
+(\U\cd\na\U|\p_x\B)_{H^{N-1}}\\
&+\ddt(\U|\p_x\B)_{H^{N-1}}+\|\p_x\U\|_{H^{N-1}}^2-(\p_x\U|\U\cd\na \B)_{H^{N-1}}+(\p_x\U|\B\cd\na \U)_{H^{N-1}},
\end{aligned}
$$
together with the application of Young's inequality
$$|(\De\U|\p_x\B)_{H^{N-1}}|\le \frac{1}{2}\|\p_x\B\|_{H^{N-1}}^2+\frac{1}{2}\|\De\U\|_{H^{N-1}}^2,$$
yields
\begin{equation}\label{couple}
\begin{aligned}
&\ \ -\ddt(\U|\p_x\B)_{H^{N-1}}+\frac{1}{2}\|\p_x\B\|_{H^{N-1}}^2-
\|\p_x\U\|_{H^{N-1}}^2\\
\le&\ \frac{1}{2}\|\De\U\|_{H^{N-1}}^2
+|(\vec{b}\cd\na \vec{b}|\p_x\vec{b})_{H^{N-1}}|
+|(\U\cd\na\U|\p_x\B)_{H^{N-1}}|\\
&+|(\p_x\U|\U\cd\na \B)_{H^{N-1}}|+|(\p_x\U|\B\cd\na \U)_{H^{N-1}}|\\
=&\ \frac{1}{2}\|\De\U\|_{H^{N-1}}^2+I_5+I_6+I_7+I_8.
\end{aligned}
\end{equation}
Multiplying (\ref{couple}) by $\frac{1}{4}$, and adding the resulting inequality, (\ref{L2}) and (\ref{H7}) together, we have
$$
\begin{aligned}
\ddt\Big(\|\vec{u}\|_{H^N}^2+\|\vec{b}\|_{H^N}^2
-\frac{1}{4}(\U|\p_x\B)_{H^{N-1}}\Big)+\frac{1}{2}\|\na \vec{u}\|_{H^N}^2
+\frac{1}{8}\|\p_x\B\|_{H^{N-1}}^2
\le \sum_{i=1}^4I_i+\frac{1}{4}\sum_{i=5}^8I_i.
  \end{aligned}
  $$
  Using product estimate and Young's inequality, we have
  $$
\begin{aligned}
I_1\le&\ \|\vec{u}\cd\na \vec{u}\|_{\dot{H}^N}
\|\vec{u}\|_{\dot{H}^N}\lesssim\ \|\U\|_{H^N}\|\na \U\|_{H^N}
\|\na \U\|_{H^{N-1}}\\
\le&\ C\|\U\|_{H^N}^2\|\na \U\|_{H^{N-1}}^2+0.001\|\na \U\|_{H^N}^2,\\
I_6\le&\ \|\vec{u}\cd\na \vec{u}\|_{H^{N-1}}
\|\p_x\vec{b}\|_{H^{N-1}}\lesssim\ \|\vec{u}\|_{H^{N-1}}\|\na \vec{u}\|_{H^{N-1}}\|\p_x\vec{b}\|_{H^{N-1}}\\
\le& C \|\vec{u}\|_{H^{N-1}}^2\| \na\vec{u}\|_{H^{N-1}}^2+0.001\|\p_x \B\|_{H^{N-1}}^2,\\
I_8\le& \ \|\B\cd\na\U\|_{H^{N-1}}\|\p_x\U\|_{H^{N-1}}\lesssim\
\|\B\|_{H^{N-1}}\|\na \U\|_{H^{N-1}}\|\p_x\U\|_{H^{N-1}}\\
\le&\ C\|\B\|_{H^{N-1}}^2\|\p_x\U\|_{H^{N-1}}^2+0.001\|\na \U\|_{H^{N-1}}^2.
\end{aligned}
$$
Since
$$
\begin{aligned}
\|\p^{N-1}(B\p_y\B)\|_{L^2}
\lesssim&\ \|\p_x^{N-1}(B\p_y\B)\|_{L^2}+\|\p_y^{N-1}(B\p_y\B)\|_{L^2}\\
\lesssim&\ \big\| \|B\|_{L^\infty_x}\|\p_y\p_x^{N-1}\B\|_{L^2_x}
+\|\p_y\B\|_{L^\infty_x}\|\p_x^{N-1}B\|_{L^2_x}\|\big\|_{L^2_y}\\
&+\big\| \|B\|_{L^\infty_y}\|\p_y^N\B\|_{L^2_y}
+\|\p_y\B\|_{L^2_y}\|\p_y^{N-1}B\|_{L^\infty_y}\|\big\|_{L^2_x}\\
\lesssim&\ \|\B\|_{H^3}\|\p_x\B\|_{H^{N-1}}+\|B\|_{L^\infty}\|\B\|_{H^N}\ (\p_yB=-\p_xb),
\end{aligned}
$$
together with $\|B\p_y\B\|_{L^2}\le\|B\|_{L^\infty}\|\B\|_{H^1}$ yields
\begin{equation}\label{f1}
\|B\p_y\B\|_{H^{N-1}}\lesssim\ \|\B\|_{H^3}\|\p_x\B\|_{H^{N-1}}+\|B\|_{L^\infty}\|\B\|_{H^N}.
\end{equation}
Thanks to (\ref{f1}), we have
$$
\begin{aligned}
I_5\le&\ (\|b\p_x\B\|_{H^{N-1}}+\|B\p_y\B\|_{H^{N-1}})\|\p_x\B\|_{H^{N-1}}\\
\lesssim&\ (\|\B\|_{H^{N-1}}\|\p_x\B\|_{H^{N-1}}+\|B\|_{L^\infty}\|\B\|_{H^N})
\|\p_x\B\|_{H^{N-1}}\\
\le&\ C(\|\B\|_{H^{N-1}}^2\|\p_x\B\|_{H^{N-1}}^2+\|B\|_{L^\infty}^2\|\B\|_{H^N}^2)
+0.001\|\p_x\B\|_{H^{N-1}}^2
\end{aligned}
$$
Following the arguments yielding (\ref{f1}), we can also get
$$\|v\p_y\B\|_{H^{N-1}}\lesssim\ \|\B\|_{H^3}\|\p_x\U\|_{H^{N-1}}+\|v\|_{L^\infty}\|\B\|_{H^N},$$
which deduce that
$$
\begin{aligned}
I_7\le&\ (\|u\p_x\B\|_{H^{N-1}}+\|v\p_y\B\|_{H^{N-1}})\|\p_x\U\|_{H^{N-1}}\\
\lesssim&\ (\|\U\|_{H^{N-1}}\|\p_x\B\|_{H^{N-1}}+
\|v\|_{L^\infty}\|\B\|_{H^N}+\|\B\|_{H^3}\|\p_x \U\|_{H^{N-1}})\|\p_x\U\|_{H^{N-1}}\\
\le&\ C(\|\U\|_{H^{N-1}}^2\|\p_x\B\|_{H^{N-1}}^2+\|\B\|_{H^3}^2\|\p_x \U\|_{H^{N-1}}^2+\|v\|_{L^\infty}^2\|\B\|_{H^N}^2)
+0.001\|\p_x\U\|_{H^{N-1}}^2.
\end{aligned}
$$
For $I_4$, we have
$$
\begin{aligned}
I_4=&-\int\  \p^N(\U\cd\na \B)\cd\p^N \B dxdy
=-\sum_{\iota+\kappa=N,0\le \kappa\le {N-1}}C_N^\kappa\int\ \p^\iota\U\cd\na \p^\kappa \B\cd\p^N \B dxdy\\
=&-\sum_{\iota+\kappa=N,0\le \kappa\le {N-1}}C_N^\kappa\int\ \p^\iota u\p_x \p^\kappa \B\cd\p^N \B dxdy-\sum_{\iota+\kappa=N,0\le \kappa\le {N-1}}C_N^\kappa\int\ \p^\iota v\p_y \p^\kappa \B\cd\p^N \B dxdy\\
\end{aligned}
$$
The first term can be bounded by
\begin{equation}\label{cha1}
C\|\na u\|_{H^N}\|\p_x\B\|_{H^{N-1}}\|\B\|_{H^N}\le\
C\|\B\|_{H^N}^2\|\na u\|_{H^N}^2+0.001\|\p_x\B\|_{H^{N-1}}^2.
\end{equation}
For the second integral, we consider it as two types based on that whether $\p^N$ contains $\p_x$. If $\p^N=\p_x^i\p_y^j$, $i+j=N$ and $i\ge 1$, we can bound this case by
$$C\|\na v\|_{H^N}\|\B\|_{H^N}\|\p_x\B\|_{H^{N-1}}\le\
C\|\na v\|_{H^N}^2\|\B\|_{H^N}^2+0.001\|\p_x\B\|_{H^{N-1}}^2.$$
Otherwise, we need to estimate  $\int \p_y^\iota v\ \p_y^{\kappa+1}\B \cd \p_y^N\B dxdy$. In fact, when $0\le \kappa\le N-2$, using $\p_yv=-\p_x u$ and integrating by parts twice, we have
$$
\begin{aligned}
&\ \ \int \p_y^\iota v\ \p_y^{\kappa+1}\B \cd\ \p_y^N\B dxdy\\
=&-\int \p_y^{\iota-1} \p_x u\ \p_y^{\kappa+1}\B\cd \p_y^N\B dxdy\\
=&\int \p_y^{\iota-1}  u\ \p_x\p_y^{\kappa+1}\B\cd\p_y^N\B dxdy
+\int \p_y^{\iota-1}  u\ \p_y^{\kappa+1}\B\cd\p_x\p_y^N\B dxdy\\
=&\int \p_y^{\iota-1}  u\ \p_x\p_y^{\kappa+1}\B\cd\p_y^N\B dxdy
-\int \p_y^{\iota}  u\ \p_y^{\kappa+1}\B\cd\p_x\p_y^{N-1}\B dxdy\\
&-\int \p_y^{\iota-1}  u\ \p_y^{\kappa+2}\B\cd\p_x\p_y^{N-1}\B dxdy\\
\lesssim&\ C\|\na \U\|_{H^N}\|\B\|_{H^N}\|\p_x\B\|_{H^{N-1}}\\
\le&\ C\|\B\|_{H^N}^2\|\na \U\|_{H^N}^2+0.001\|\p_x\B\|_{H^{N-1}}^2.
\end{aligned}
$$
When $\kappa={N-1}$, we can bound this integral by $ \|\p_xu\|_{L^\infty}\|\B\|_{H^N}^2$. Hence, we have
$$I_4\le\ C(\|\B\|_{H^N}^2\|\na \U\|_{H^N}^2+\|\p_xu\|_{L^\infty}\|\B\|_{H^N}^2)+0.01\|\p_x\B\|_{H^{N-1}}^2.$$
At last, we bound $I_2+I_3$. Applying the cancelation property
$$\int \vec{b}\cd\na \p^N\vec{u}\cd\p^N \vec{b}dx+\int \vec{b}\cd\na \p^N\vec{b}\cd\p^N \vec{u}dx=0,$$
we have
$$
I_2+I_3=\ \sum_{\iota+\kappa=N,0\le \kappa\le{N-1}}C_N^\kappa
\int \p^\iota \vec{b}\cd\na\p^\kappa \vec{u}
\cdot\p^N\vec{b}dxdy+\sum_{\iota+\kappa=N,0\le \kappa\le{N-1}}C_N^\kappa
\int \p^\iota \vec{b}\cd\na\p^\kappa \vec{b}
\cdot\p^N\vec{u}dxdy.
$$
By a similar analysis of $I_4$, we only need to bound the  integral:
$$\digamma_1=\int \p_y^\iota b\  \p_x\p_y^\kappa u \p_y^Nb dxdy,$$
since other cases can be bounded by the left hand side of (\ref{cha1}).
When $1\le\kappa\le{N-1}$, integrating by parts twice, we have
$$
\begin{aligned}
\digamma_1
=&-\int \p_x\p_y^\iota b\  \p_y^\kappa u \p_y^Nb dxdy
-\int \p_y^\iota b\  \p_y^\kappa u \p_x\p_y^Nb dxdy\\
=&-\int \p_x\p_y^\iota b\  \p_y^\kappa u\p_y^Nb dxdy
+\int \p_y^{\iota+1} b\  \p_y^\kappa u \p_x\p_y^{N-1}b dxdy\\
&+\int \p_y^{\iota} b\  \p_y^{\kappa+1} u \p_x\p_y^{N-1}b dxdy
\lesssim\ \|\B\|_{H^N}\|\na \U\|_{H^N}\|\p_x\B\|_{H^{N-1}}\\
\le&\ C \|\B\|_{H^N}^2\|\na \U\|_{H^N}^2+0.001\|\p_x\B\|_{H^{N-1}}^2.
\end{aligned}
$$
When $\kappa=0$, we have
$$\digamma_1\lesssim\ \|\p_xu\|_{L^\infty}\|\B\|_{H^N}^2.$$
So
$$I_2+I_3\lesssim\ C (\|\p_xu\|_{L^\infty}+\|\na \U\|_{H^N}^2)\|\B\|_{H^N}^2+0.01\|\p_x\B\|_{H^{N-1}}^2.$$
Collecting the above estimates of $I_i$ $i=\overline{1,8}$, we can get
$$
\begin{aligned}
&\ \ \ddt\Big(\|\vec{u}\|_{H^N}^2+\|\vec{b}\|_{H^N}^2
-\frac{1}{4}(\U|\p_x\B)_{H^{N-1}}\Big)+\frac{1}{2}\|\na \vec{u}\|_{H^N}^2
+0.05\|\p_x\B\|_{H^{N-1}}^2\\
\lesssim & (\|\U\|_{H^N}^2+\|\B\|_{H^N}^2)(\|\na \U\|_{H^N}^2
+\|\p_x\B\|_{H^{N-1}}^2+\|v\|_{L^\infty}^2+\|B\|_{L^\infty}^2+\|\p_xu
\|_{L^\infty}).
  \end{aligned}
  $$
Integrating in time, using
$$
\|\vec{u}\|_{H^N}^2+\|\vec{b}\|_{H^N}^2
-\frac{1}{4}(\U|\p_x\B)_{H^{N-1}}\thickapprox \|\vec{u}\|_{H^N}^2+\|\vec{b}\|_{H^N}^2
$$
and
$$
\int_0^t(\|v\|_{L^\infty}^2+\|B\|_{L^\infty}^2+\|\p_xu
\|_{L^\infty})d\tau
\le\ \int_0^t \langle \tau\rangle^{-\frac{5}{4}}d\tau (\|V\|_3^2+\|V\|_3)
\le\ C (\|V\|_3^2+\|V\|_3),
$$
we conclude the proof of (\ref{H8}).
\section{The  estimates on $v$}
\label{sv}
\vskip.1in
In this section, we shall prove
\begin{equation}\label{4.1}
\left\{
\begin{aligned}
\|v(t)\|_{L^2}\lesssim\ &\langle t\rangle^{-\frac{1}{2}}
\big(\|V_0\|_3+\|V\|_3^2\big);\\
\|\p_x v(t)\|_{H^2}\lesssim\ &\langle t\rangle^{-1}
\big(\|V_0\|_3+\|V\|_3^2\big);\\
\|v(t)\|_{\F L^1}\lesssim\ &\langle t\rangle^{-1}
\big(\|V_0\|_3+\|V\|_3^2\big).
\end{aligned}
\right.
\end{equation}
Thanks  to the Plancherel's identity, let us turn to the  estimate of
$\|v\|_{\F L^2}$, $\|\p_x\langle\na \rangle^2 v\|_{\F L^2}$ and $\|v\|_{\F L^1}$, respectively. By (\ref{U}),
the expression of $v$ can be given by
\begin{equation}\label{eprv}
v(t)=\underbrace{M_{3}(\p,t)v_0+M_{1}(\p,t)B_0}_{L_v}
+\underbrace{\int_0^t M_{3}(\p,t-\tau)F^2d\tau}_{NL_{v1}}
+\underbrace{\int_0^t M_{1}(\p,t-\tau)G^2d\tau}_{NL_{v2}},
\end{equation}
where
\begin{equation}\label{f2}
\begin{aligned}
  F^2=&-\U\cd\na v-\p_y\De^{-1}{\rm div}(\B\cd\na\B-\U\cd\na\U)+\B\cd\na B\\
  =&\mathcal{R}_{12}(\B\cd\na b-\U\cd\na u)-\mathcal{R}_{11}
  (\B\cd\na B-\U\cd\na v),\\
  \end{aligned}
\end{equation}
and
\begin{equation}\label{g2}
G^2
=-\U\cd\na B+\B\cd\na v=-\p_x(uB-bv).
\end{equation}
The nonlinear term $\B\cd\na\B$ can be rewritten as
\begin{equation}\label{bbfen}
\B\cd\na\B=(b\p_x b+B\p_y b,\B\cd\na B)=(2b\p_x b+\p_y(bB),\B\cd\na B)=(\p_x (bb)+\p_y(bB),\B\cd\na B).
\end{equation}
\subsection{The estimate of $(\ref{4.1})_1$}
\label{subv1}
Using (\ref{rf1}) and (\ref{decay1})$_1$ for $k=0$, one can get
$$
\|L_v\|_{\F L^2(D_1)}\lesssim\ \|M_{3}(\p,t)v_0\|_{\F L^2(D_1)}+\|M_{1}(\p,t)B_0\|_{\F L^2(D_1)}
\lesssim\ \langle t\rangle^{-\frac{1}{2}}
(\|v_0\|_{L^1}+\|B_0\|_{L^1}).
$$
 By (\ref{rf1}), (\ref{bbfen}), (\ref{decay1})$_1$ for $k=1$ and $k=1/2$, (\ref{ine3}),  (\ref{ine3}) and (\ref{ine1}),  we  infer
 $$
\begin{aligned}
\sum_{i=1,2}\|NL_{vi}\|_{\F L^2(D_1)}
 \lesssim&\ \int_0^t\|M_3(\p,t-\tau)F^2\|_{\F L^2(D_1)}d\tau
+ \int_0^t\|M_1(\p,t-\tau)G^2\|_{\F L^2(D_1)}d\tau
\\
\lesssim&\ \int_0^t \LH^{-\frac{1}{2}}\||\na|^{-1}\Big(\U\cd\na\U,\p_y(Bb),\B\cd\na B,\U\cd\na\B,\B\cd\na\U\Big)\|_{L^2}d\tau
\\
&+\int_0^t \|M_3(\p,t-\tau)\p_x(bb)\|_{\F L^2(D_1)}d\tau\\
\lesssim&\ \int_0^t \langle t-\tau\rangle^{-\frac{1}{2}}
(\|\U\otimes \U\|_{L^2}+\|\B B\|_{L^2}+\|\U \otimes\B\|_{L^2})d\tau\\
&+\int_0^t \langle t-\tau\rangle^{-\frac{3}{4}}\||\p_x|^\frac{1}{2}(bb)
\|_{L^1}d\tau\\
\lesssim&\ \int_0^t \Big(\langle t-\tau\rangle^{-\frac{1}{2}}
\langle \tau\rangle^{-1.1}+\langle t-\tau\rangle^{-\frac{3}{4}}
\langle \tau\rangle^{-\frac{3}{4}}\Big) d\tau
\|V\|_3^2\\
\lesssim&\ \langle t\rangle^{-\frac{1}{2}}\|V\|_3^2.
 \end{aligned}
 $$
So
\begin{equation}\label{v1}
\|\widehat{v}\|_{L^2(D_1)}
\lesssim\ \langle t\rangle^{-\frac{1}{2}}(\|V_0\|_3+\|V\|_3^2).
\end{equation}
By (\ref{rf2}),(\ref{End3})$_1$, (\ref{End1})$_1$,
we have
$$
\begin{aligned}
\|L_v\|_{\F L^2(D_4)}
\lesssim&\ \|M_3(\p,t)v_0\|_{\F L^2(D_4)}+\|M_1(\p,t)B_0\|_{\F L^2(D_4)}\\
\lesssim&\ \langle t\rangle^{-\frac{1}{2}}
\|(v_0,B_0)\|_{L^2\cap L^1}.
\end{aligned}
$$
It follows from (\ref{rf2}), (\ref{bbfen}),  (\ref{End3})$_1$ and (\ref{End3})$_4$ for $r=2$ that
$$
\begin{aligned}
\|NL_{v1}\|_{\F L^2(D_4)}
\lesssim&\ \int_0^t\|M_3(\p,t-\tau)F^2\|_{\F L^2(D_4)}d\tau\\
\lesssim&\ \int_0^t\|M_3(\p,t-\tau)
(\U\cd\na\U,\B\cd\na\B)\|_{\F L^2(D_4)}d\tau\\
\lesssim&\ \int_0^t\|M_3(\p,t-\tau)
(\U\cd\na\U,\B \cd\na B, \p_y(\B B))\|_{\F L^2(D_4)}d\tau\\
&+\int_0^t\|\p_xM_3(\p,t-\tau)
(bb)\|_{\F L^2(D_4)}d\tau\\
\lesssim&\ \int_0^t \LH^{-\frac{1}{2}}
\||\na|^{-1}(\U\cd\na\U,\B \cd\na B, \p_y(\B B))\|_{H^1}d\tau\\
&+\int_0^t \LH^{-1}(\|bb\|_{\F L^2}+\|b\p_xb\|_{\F L^2})d\tau\\
\lesssim&\ \int_0^t \LH^{-\frac{1}{2}}
\big(\|(\U\otimes\U,\U\cd\na\U)\|_{L^2}+\|\B B\|_{H^1}\big)d\tau\\
&+\int_0^t \LH^{-1}(\|bb\|_{L^2}+\|b\p_xb\|_{L^2})d\tau\\
\lesssim&\ \int_0^t
\big(\LH^{-\frac{1}{2}}\langle\tau\rangle^{-1.1}
+\LH^{-1}\langle\tau\rangle^{-\frac{3}{4}}\big)d\tau\|V\|_3^2\\
\lesssim&\ \LHH^{-\frac{1}{2}}\|V\|_3^2.
\end{aligned}
$$
Using (\ref{rf2}), (\ref{End1})$_1$ and (\ref{ine3}), we can  deduce
$$
\begin{aligned}
  \|NL_{v2}\|_{\F L^2(D_4)}\lesssim&\
  \int_0^t \|M_1(\p,t-\tau)G^2\|_{L^2(D_4)}d\tau\\
  \lesssim&\ \int_0^t \langle t-\tau\rangle^{-\frac{1}{2}}
  \||\na|^{-1}(\U\cd\na \B,\B\cd\na \U)\|_{L^2}d\tau \\
  \lesssim&\ \int_0^t \langle t-\tau\rangle^{-\frac{1}{2}}
  \|\U\otimes\B\|_{L^2}d\tau\\
  \lesssim&\ \int_0^t \langle t-\tau\rangle^{-\frac{1}{2}}
  \langle \tau\rangle^{-1.1}d\tau\ \|V\|_3^2
  \lesssim\ \langle t\rangle^{-\frac{1}{2}}\|V\|_3^2.
\end{aligned}
$$
Then we have
\begin{equation}\label{v2}
\|\widehat{v}\|_{L^2(D_4)}
\lesssim\ \LHH^{-\frac{1}{2}}
(\|V_0\|_3+\|V\|_3^2).
\end{equation}
Due to Remark \ref{rpp},
combining with  (\ref{v1}) and (\ref{v2})  leads to the estimate (\ref{4.1})$_1$.
\begin{rem}\label{ru}
In the estimate of $\|v\|_{L^2}$,
$\widehat{F^2} $ and $\widehat{G^2} $ are bounded as  follows:
$$|\widehat{F^2}|\le |\xe|
|\F\{\U\otimes \U,\B B\}|+|\xi||\F\{bb\}|,\ \ |\widehat{G^2}|\le |\xe|
|\F\{\U\otimes \B\}|.$$
Since $F^1$ and $G^1$ can be bounded by the same way,
we  can also obtain the similar estimate of $\|u\|_{L^2}$, see (\ref{4.2})$_1$.
\end{rem}

\subsection{The estimate of $(\ref{4.1})_2$}\label{subv2}
Due to (\ref{proper1}), we have
$$\|\p_x\langle\na\rangle^2v\|_{\F L^2(D_1)}\lesssim\
\|\p_xv\|_{\F L^2(D_1)}.$$
By  (\ref{rf1}), (\ref{decay1})$_1$ for $k=1$,  we have
$$
\|\p_xL_v\|_{\F L^2(D_1)}
\lesssim\ \|\p_xM_{3}(\p,t)v_0\|_{\F L^2(D_1)}
+\|\p_xM_{1}(\p,t)B_0\|_{\F L^2(D_1)}
\lesssim\ \langle t\rangle^{-1}(\|v_0\|_{L^1}+\|B_0\|_{L^1}).
$$
By using (\ref{rf1}), (\ref{bbfen}), (\ref{decay1})$_1$ for $k=2$,  we can also obtain

 $$
\begin{aligned}
 &\ \sum_{i=1,2}\|\p_xNL_{vi}\|_{\F L^2(D_1)}\\
 \lesssim&\ \int_0^t\|\p_xM_3(\p,t-\tau)F^2\|_{\F L^2(D_1)}d\tau
+ \int_0^t\|\p_xM_1(\p,t-\tau)G^2\|_{\F L^2(D_1)}d\tau
\\
\lesssim&\ \int_0^t \LH^{-1}\||\na|^{-1}\Big(\U\cd\na\U,\p_y(Bb),
\B\cd\na B,\U\cd\na\B,\B\cd\na\U\Big)\|_{L^2}d\tau
\\
&+\int_0^t \|\p_x^2M_3(\p,t-\tau)(bb)\|_{\F L^2(D_1)}d\tau\\
\lesssim&\ \int_0^t \LH^{-1}\|\Big(\U\otimes\U,B\B,\U\otimes\B\Big)\|_{L^2}d\tau
+\int_0^t \|\p_x^2M_3(\p,t-\tau)(bb)\|_{\F L^2(D_1)}d\tau.
 \end{aligned}
 $$
 By (\ref{ine3}), the first integral on the right hand side  can be bounded by
 \begin{equation}\label{w1}
C\int_0^t \langle t-\tau\rangle^{-1}\langle \tau\rangle^{-1.1}d\tau\ \|V\|_3^2\lesssim\ \langle t\rangle^{-1}\|V\|_3^2.
\end{equation}
For the second integral, we shall  split it into two integrals. Using (\ref{rf1}), (\ref{decay1})$_1$ for $k=2$ and $k=1$, (\ref{ine1}), (\ref{ine2}) for $p=2$, one can get
$$
\begin{aligned}
&\ \int_0^t \|\p_x^2M_3(\p,t-\tau)(bb)\|_{\F L^2(D_1)}d\tau\\
\lesssim&\ \int_0^{t/2} \|\p_x^2M_3(\p,t-\tau)(bb)\|_{\F L^2(D_1)}d\tau
+\int_{t/2}^t \|\p_xM_3(\p,t-\tau)(b\p_xb)\|_{\F L^2(D_1)}d\tau\\
\lesssim&\ \int_0^{t/2} \LH^{-\frac{3}{2}}\|bb\|_{L^1}d\tau
+\int_{t/2}^t \LH^{-\frac{3}{4}}\|b\p_xb\|_{L^1_y(L^2_x)}d\tau\\
\lesssim&\ \Big(\int_0^{t/2}
\langle t-\tau\rangle^{-\frac{3}{2}}\langle \tau\rangle^{-\frac{1}{2}}d\tau
+\int_{t/2}^t \LH^{-\frac{3}{4}}\langle \tau\rangle^{-\frac{5}{4}}d\tau\Big)\|V\|_3^2\\
\lesssim&\ \LHH^{-1}\|V\|_3^2.
\end{aligned}
$$
So
$$
\sum_{i=1,2}\|\p_x NL_{vi}\|_{\F L^2(D_1)}
\lesssim\ \langle t\rangle^{-1}\|V\|_3^2.
$$
Hence
\begin{equation}\label{v3}
\|\p_x\langle\na\rangle^2 v\|_{\F L^2(D_1)}
\lesssim\ \LHH^{-1}
(\|V_0\|_3+\|V\|_3^2).
\end{equation}
Applying (\ref{rf2}), (\ref{End1})$_3$ and (\ref{End3})$_4$ for $r=2$, (\ref{ine3}),  we have
$$
\begin{aligned}
\|\p_x\langle\na\rangle^2L_v\|_{\F L^2(D_{4})}
\lesssim &\ \|\p_x\langle\na\rangle^2M_3(\p,t)v_0\|_{\F L^2(D_{4})}
+\|\p_x\langle\na\rangle^2M_1(\p,t)B_0\|_{\F L^2(D_{4})}\\
\lesssim&\ \langle t\rangle^{-1}
(\|v_0\|_{H^3}+\|B_0\|_{H^2}).
\end{aligned}
$$
We deduce from (\ref{rf2}), (\ref{End1})$_3$,  (\ref{End3})$_4$ for $r=2$, (\ref{End3})$_8$ for $r=2$ and (\ref{ine3}) that
$$
\begin{aligned}
&\ \ \sum_{i=1,2}\|\p_x\langle\na\rangle^2NL_{vi}\|_{\F L^2(D_{4})}\\
\lesssim&\ \int_0^t
\big(\|\p_x M_3(\p,t-\tau)F^2\|_{\F L^2(D_{4})}
+\|\p_x \De M_3(\p,t-\tau)F^2\|_{\F L^2(D_{4})}\\
&+
\|\p_x\langle\na\rangle^2 M_1(\p,t-\tau)G^2\|_{\F L^2(D_{4})}\big)d\tau\\
\lesssim&\ \int_0^t \LH^{-1}
(\|F^2\|_{H^1}+\|G^2\|_{H^2})d\tau
+\int_0^t \LH^{-\frac{3}{2}}
(\|F^2\|_{H^1}+\|\p_x F^2\|_{H^2})d\tau\\
\lesssim&\ \int_0^t
(\LH^{-1}\langle \tau\rangle^{-1.01}
+\LH^{-\frac{3}{2}}\langle \tau\rangle^{-1})d\tau\|V\|_3^2
\lesssim\ \langle t\rangle^{-1}\|V\|_3^2.
\end{aligned}
$$
Thus we get
\begin{equation}\label{v4}
\|\p_x\langle\na \rangle^2v\|_{\F L^2(D_{4})}
\lesssim\ \LHH^{-1}(\|V_0\|_3+\|V\|_3^2).
\end{equation}
Due to Remark \ref{rpp},
we can get the desired result $(\ref{4.1})_2$ by combining (\ref{v3}) with (\ref{v4}).
\vskip .1in
\subsection{The estimate of $(\ref{4.1})_3$}\label{subv3}
\vskip .1in
Using (\ref{rf1}) and (\ref{decay1})$_3$ for $k=0$, we can infer
$$\|L_v\|_{\F L^1(D_1)}
\lesssim\ \|M_3(\p,t)v_0\|_{\F L^1(D_1)}+\|M_1(\p,t)B_0\|_{\F L^1(D_1)}
\lesssim\ \langle t\rangle^{-1}
(\|v_0\|_{L^1}+\|B_0\|_{L^1}).
$$
It follows from (\ref{rf1}), (\ref{bbfen}) and (\ref{decay1})$_3$ for $k=0$ that

$$
\begin{aligned}
 &\ \ \sum_{i=1,2}\|NL_{vi}\|_{\F L^1(D_1)}\\
 \lesssim&\ \int_0^t\|M_3(\p,t-\tau)F^2\|_{\F L^1(D_1)}d\tau
+ \int_0^t\|M_1(\p,t-\tau)G^2\|_{\F L^1(D_1)}d\tau
\\
\lesssim&\ \int_0^t \LH^{-1}\||\na|^{-1}\Big(\U\cd\na\U,\p_y(Bb),
\B\cd\na B,\U\cd\na\B,\B\cd\na\U\Big)\|_{L^2}d\tau
\\
&+\int_0^t \|\p_xM_3(\p,t-\tau)(bb)\|_{\F L^1(D_1)}d\tau\\
\lesssim&\ \int_0^t \LH^{-1}\|\Big(\U\otimes\U,B\B,\U\otimes\B\Big)\|_{L^2}d\tau
+\int_0^t \|\p_xM_3(\p,t-\tau)(bb)\|_{\F L^1(D_1)}d\tau,
 \end{aligned}
 $$
where
the first  integral can be bounded by (\ref{w1}).
If we use (\ref{End3})$_4$,   the second integral can be bounded by
 $$C \int_0^t \LH^{-1}\|bb\|_{\F L^1}d\tau
 \lesssim\ \int_0^t \LH^{-1}\langle\tau\rangle^{-1}d\tau\|V\|_3
 \lesssim\ \LHH^{-1+\epsilon}\|V\|_3^2$$
 for sufficiently small $\epsilon>0$. However, there is an $\epsilon$-loss in the decay rate. To overcome this difficulty, we split the interval $[0,t]$ into $[0,t/2)$ and $[t/2,t)$ again. In fact, using (\ref{rf1}), (\ref{decay1})$_3$ for $k=1$ and $k=0$,  (\ref{ine1})
 and (\ref{ine2}),
we can bound  the second integral by
$$
\begin{aligned}
&\ \ C(\int_0^{t/2}  \langle t-\tau\rangle^{-\frac{3}{2}}\|bb\|_{L^1}d\tau
 +\int_{t/2}^t  \langle t-\tau\rangle^{-\frac{3}{4}}\|b\p_xb\|_{L^1_y(L^2_x)}d\tau)\\
\lesssim&\ \Big(\int_0^{t/2}  \langle t-\tau\rangle^{-\frac{3}{2}}\langle\tau\rangle^{-\frac{1}{2}}d\tau
 +\int_{t/2}^t  \langle t-\tau\rangle^{-\frac{3}{4}}
 \langle\tau\rangle^{-\frac{5}{4}}d\tau\Big)
\|V\|_3^2\\
\lesssim&\ \langle t\rangle^{-1}\|V\|_3^2.
\end{aligned}
$$
These estimates follows
$$\sum_{i=1,2}\|NL_{vi}\|_{\F L^1(D_1)}
\lesssim\ \langle t\rangle^{-1}\|V\|_3^2.$$
So
\begin{equation}\label{v5}
\|\widehat{v}\|_{L^1(D_1)}
\lesssim\ \LHH^{-1}(\|V_0\|_3+\|V\|_3^2).
\end{equation}
Using (\ref{rf2}), (\ref{End3})$_2$, (\ref{End1})$_4$ for $r=1$, we can deduce
$$
\begin{aligned}
\|L_v\|_{\F L^1(D_{4})}
\lesssim &\ \|M_3(\p,t)v_0\|_{\F L^1(D_4)}+\|M_1(\p,t)B_0\|_{\F L^1(D_4)}\\
\lesssim&\ \langle t\rangle^{-1}
(\|v_0\|_{\F L^1\cap L^1\cap L^1_y(L^\frac{4}{3}_x)}
+\|B_0\|_{L^1}+\|\langle\na\rangle^{1.51} B_0\|_{L^1_x(L^2_y)})\\
\lesssim&\ \langle t\rangle^{-1}\|V_0\|_3.
\end{aligned}
$$
Using (\ref{rf2}), (\ref{f2}), (\ref{bbfen}), (\ref{End3})$_{10}$, and (\ref{End3})$_{11}$, (\ref{ine4}), we infer
$$
\begin{aligned}
  NL_{v1}\lesssim&\ \int_0^t\|M_3(\p,t-\tau)F^2\|_{\F L^1(D_4)}d\tau\\
  \lesssim&\ \int_0^t\|\mathcal{R}_1M_3(\p,t-\tau)(\U\cd\na\U,\B\cd\na\B)
  \|_{\F L^1(D_4)}d\tau\\
  \lesssim&\ \int_0^t\|\mathcal{R}_1M_3(\p,t-\tau)(\U\cd\na\U,\B\cd\na B, \p_y(\B B))
  \|_{\F L^1(D_4)}d\tau\\
  &+\int_0^t \|\p_x\mathcal{R}_1M_3(\p,t-\tau)(bb)\|_{\F L^1(D_4)}d\tau\\
  \lesssim&\ \int_0^t\LH^{-1}\Big(\||\na|^{-1}(\U\cd\na\U,\B\cd\na B, \p_y(\B B))
  \|_{\F L^1}\\
  &+\|(\U\cd\na\U,\B\cd\na B, \p_y(\B B))\|_{\F L^1}\Big) d\tau\\
  &+\int_0^t\LH^{-\frac{3}{2}}\|(bb,b\p_xb)\|_{\F L^1}d\tau\\
  \lesssim&\ \int_0^t(\LH^{-1} \langle\tau\rangle^{-1.1}
  +\LH^{-\frac{3}{2}} \langle\tau\rangle^{-1})d\tau
  \|V\|_3^2\\
  \lesssim&\ \LHH^{-1}\|V\|_3^2
\end{aligned}
$$
and, by (\ref{g2}), (\ref{End1})$_3$ and (\ref{ine4}), we can also get
$$
\begin{aligned}
\|NL_{v2}\|_{\F L^1(D_4)}
\lesssim&\ \int_0^t \|\p_xM_1(\p,t-\tau)
(uB-bv)\|_{\F L^1(D_4)}d\tau\\
\lesssim&\ \int_0^t \langle t-\tau\rangle^{-1}
\|uB-bv\|_{\F L^1}d\tau\\
\lesssim&\  \int_0^t \langle t-\tau\rangle^{-1}\langle \tau\rangle^{-\frac{3}{2}}
d\tau \|V\|_3^2
\lesssim\ \langle t\rangle^{-1}\|V\|_3^2.
\end{aligned}
$$
Thus
\begin{equation}\label{v6}
\|\widehat{v}\|_{ L^1(D_4)}
\lesssim\ \langle t\rangle^{-1}(\|V_0\|_3+\|V\|_3^2).
\end{equation}
Due to Remark \ref{rpp},
collecting (\ref{v5}) and (\ref{v6}) can yield
(\ref{4.1})$_3$. This completes the proof of (\ref{4.1}).
\section{The  estimate on $u$}
\label{su}
\vskip.1in
In this section, we will prove
\begin{equation}\label{4.2}
\left\{
\begin{aligned}
\|u(t)\|_{L^2}\lesssim\ &\langle t\rangle^{-\frac{1}{2}}
\big(\|V_0\|_3+\|V\|_3^2\big);\\
\|\p_yu(t)\|_{L^2}\lesssim\ &\langle t\rangle^{-\frac{3}{4}}
\big(\|V_0\|_3+\|V_0\|_3^2+\|V\|_3^2+\|V\|_3^3\big);\\
\|\p_xu(t)\|_{H^2}\lesssim\ &\langle t\rangle^{-1}
\big(\|V_0\|_3+\|V\|_3^2\big);\\
\|u(t)\|_{\F L^1}\lesssim\ &\langle t\rangle^{-1}
\big(\|V_0\|_3+\|V_0\|_3^2+\|V\|_3^2+\|V\|_3^3\big);\\
\|\p_x u(t)\|_{\F L^1}\lesssim\ &\langle t\rangle^{-\frac{5}{4}}
\big(\|V_0\|_3+\|V\|_3^2\big).
\end{aligned}
\right.
\end{equation}
\subsection{The expression of $u$}
\vskip.2in
By (\ref{U}), we can obtain
$$
\begin{aligned}
u=&M_{3}(\p,t)u_0+M_{1}(\p,t)b_0
+\int_0^t M_{3}(\p,t-\tau)F^1d\tau
+\underbrace{\int_0^t M_{1} (\p,t-\tau)G^1d\tau}_{\mathcal{NL}_u}\\
=& \underbrace{M_{3}(\p,t)u_0+M_{1}(\p,t)b_0}_{L_u}
+\underbrace{\int_0^t M_{3}(\p,t-\tau)F^1d\tau}_{NL_{u1}}
+\underbrace{\int_0^t M_{1} (\p,t-\tau) G^{11}d\tau}_{NL_{u2}}\\
&-\underbrace{\int_0^t  M_{1} (\p,t-\tau)v\p_y bd\tau}_{NL_{u3}},
\end{aligned}
$$
where
$\mathcal{NL}_u=NL_{u2}-NL_{u3}$
and
$$F^1=-\U\cd\na u+\p_xp+\B\cd\na b,\
G^1=G^{11}-v\p_yb,\ G^{11}=-u\p_x b+\B\cd\na u.$$
As Remark \ref{ru},
we can prove (\ref{4.2})$_1$ by the previous arguments
yielding (\ref{4.1})$_1$.
 Since the estimates of $M_i$ $(i=1,2,3)$ on $D_1$ are similar, following the proof of
 $\|\p_xv\|_{\F L^2(D_1)}$ and $\|\widehat{v}\|_{L^1(D_1)}$ yields
\begin{equation}\label{d1}
\|\na\langle\na\rangle^2 u\|_{\F L^2(D_1)}+\|\widehat{u}\|_{L^1(D_1)}
 \lesssim\ \LHH^{-1}
 (\|V_0\|_3+\|V\|_3^2).
 \end{equation}
 Next, we turn to these estimates on $D_4$.
 When we use the  arguments in the  section \ref{sv}, the integral on $v\p_y b$ seems hard to be controlled. It is difficult to get the desired decay  rate at least. So we shall seek some new approaches to overcome this difficulty.
By frequency decomposition technique, we first obtain
\begin{equation}\label{add.3}
\begin{aligned}
NL_{u3}=&\int_0^t
M_{1}(\p,t-\tau)(v_{<\langle \tau\rangle^{-8}}\p_y b
+v_{>2\langle \tau\rangle^{-0.05}}\p_y b)d\tau\\
&+\int_0^t
M_{1}(\p,t-\tau)(P_\backsim v\p_y b)d\tau\\
:=&NL_{u3}'+NL_{u3}'',
\end{aligned}
\end{equation}
where $P_\backsim$  is given by (\ref{ppp}).
To obtain the desired estimate of $NL_{u3}''$, we need rewriting  its expression.
Thanks to (\ref{rrrr}),
then $v=\vec{R'}\cd \p_x \U$. Using the formations of (\ref{mhd})$_2$ and (\ref{mhd})$_1$,  integrating by parts on time, we can get
$$
\begin{aligned}
  NL_{u3}''
  =& \int_0^t M_{1}(\p,t-\tau)
  P_\backsim (\vec{R'}\cd\p_x \vec{u})\p_y bd\tau\\
  =& \int_0^t M_{1}(\p,t-\tau)
  P_\backsim \big(\vec{R'}\cd (\B_\tau+\U\cd\na \B-\B\cd\na \U)\big)\p_y b d\tau\\
  =&\int_0^t \p_\tau\Big[M_{1}(\p,t-\tau)
  P_\backsim (\vec{R'}\cd \B)\p_y b \Big]d\tau\\
&-\int_0^t \Big(\p_\tau M_{1}(\p,t-\tau)\Big)\Big[
  P_\backsim (\vec{R'}\cd \B)\p_y b \Big]d\tau\\
  &-\int_0^tM_{1}(\p,t-\tau)
 \Big[ P_\backsim (\vec{R'}\cd \B)\p_y b_\tau \Big]d\tau
  +\ {\rm ``other\ \  good\ \  parts"}\\
  =& -M_{1}(\p,t)\big(
  P_{\backsim1} (\vec{R'}\cd \B_0)\p_y b_0 \big)
-\int_0^t \Big(\p_\tau M_{1}(\p,t-\tau)\Big)\Big[
  P_\backsim (\vec{R'}\cd \B)\p_y b \Big]d\tau\\
  & -\int_0^tM_{1}(\p,t-\tau)\Big[
  P_\backsim (\vec{R'}\cd \B)\p_y \p_x u \Big]d\tau
  +\ {\rm ``other\ \  good\ \  parts"}\\
  =& J_1+J_2+J_3+\ {\rm ``other\ \  good\ \  parts"},
\end{aligned}
$$
where $\B_\tau=\p_\tau \B$ and $P_{\backsim1}f=P_{1\le \cd\le 2}f$.
\begin{rem}\label{addr1}
 Here ``{\rm other\ \  good\ \  parts}" include  two kinds of ``good" cases: (1) the integral on the term in which  $\p_\tau$ acts on $P_\backsim$;  (2) the integral on the nonlinear term consisting of three unknowns. The process that $\p_\tau$ hits $P_\backsim$  brings  the decay rate $\langle \tau\rangle^{-1}$, while
  three unknowns  shall bring the faster decay rate than two unknowns. So we can easily bound these cases, and  omit the details in the following context.
\end{rem}
Hence, we can get the new expression of $u$:
\begin{equation}\label{expu}
u=L_u+NL_{u1}+NL_{u2}-NL_{u3}'-\sum_{i=1}^3J_i-\ {\rm ``other\ \  good\ \  parts"},
\end{equation}
where, by using (\ref{tm}),
 $J_2$ and $J_3$ can be rewritten as follows
\begin{equation}\label{7.7}
\begin{aligned}
J_2
  =&\int_0^t (\p_y\p_\tau M_{1})(\p,t-\tau)\Big[
  P_\backsim (\vec{R'}\cd \B) b \Big]d\tau
  -\int_0^t (\p_\tau M_{1})(\p,t-\tau)\Big[\p_y
  P_\backsim (\vec{R'}\cd \B)b \Big]d\tau\\
  =&\int_0^t (\p_y M_{3})(\p,t-\tau)\Big[
  \p_xP_\backsim (\vec{R'}\cd \B) b \Big]d\tau
  +\int_0^t (\p_y M_{3})(\p,t-\tau)\Big[
  P_\backsim (\vec{R'}\cd \B)\p_x b \Big]d\tau\\
  &-\int_0^t (\p_x M_{3})(\p,t-\tau)\Big[\p_y
  P_\backsim (\vec{R'}\cd \B)b \Big]d\tau
  =J_{21}+J_{22}+J_{23}
  \end{aligned}
\end{equation}
and
\begin{equation}\label{7.8}
J_3=-\int_0^t\p_xM_{1}(\p,t-\tau)\Big[
  P_\backsim (\vec{R'}\cd \B)\p_y  u \Big]d\tau
  +\int_0^tM_{1}(\p,t-\tau)\Big[\p_x
  P_\backsim (\vec{R'}\cd \B)\p_y  u \Big]d\tau.
\end{equation}
(\ref{expu}) is used to prove the estimate of $\|\p_yu\|_{\F L^2(D_4)}$, but seems useless for the estimate of $\|\widehat{u}\|_{\F L^1(D_4)}$. Motivated by this fact,
 we give another   expression of $NL_{u3}''$. Using
$\p_yv=-\p_xu$, we have
$$
\begin{aligned}
 NL_{u3}''
 =& \int_0^t M_1(\p,t-\tau)(\p_xP_\backsim u b)d\tau
 +\int_0^t \p_yM_1(\p,t-\tau)(P_\backsim v b) d\tau\\
 =&\int_0^t M_1(\p,t-\tau)(\p_xP_\backsim u b)d\tau
 +\int_\frac{t}{2}^t \p_yM_1(\p,t-\tau)(P_\backsim v b) d\tau\\
 &+ \int_0^\frac{t}{2} \p_yM_1(\p,t-\tau)(P_\backsim v b) d\tau
 = O_1+O_2+O_3.
\end{aligned}
$$
Applying the similar techniques yielding the previous expression of $NL_{u3}''$ to $O_3$, we can get
$$
\begin{aligned}
O_3=&\p_yM_1(\p,t-\tau)(P_\backsim(\vec{R'}\cd\B)b)\Big|_{\tau
=0}^{\tau=\frac{t}{2}}-
\int_0^{t/2}(\p_\tau\p_yM_1(\p,t-\tau))\big
[P_\backsim(\vec{R'}\cd\B)b\big]d\tau
\\
&-\int_0^{t/2}\p_y M_1(\p,t-\tau)\big[P_\backsim(\vec{R'}\cd\B)\p_xu\big]d\tau
+{\rm ``other\ good\ parts"}\\
=&O_{31}+O_{32}+O_{33}+{\rm ``other\ good\ parts"}.
\end{aligned}
$$
Thus we have
\begin{equation}\label{expu2}
u=L_u+NL_{u1}+NL_{u2}-NL_{u3}'-O_1-O_2-\sum_{i=1}^3O_{3i}-\ {\rm ``other\ \  good\ \  parts"},
\end{equation}
where, by using (\ref{tm}), $O_{32}$ and $O_{33}$ can be rewritten by
$$
\begin{aligned}
O_{32}=\int_0^{t/2}(\p_\tau\p_yM_1)(\p,t-\tau)
[P_\backsim(\vec{R'}\cd\B)b]d\tau
=\int_0^{t/2}(\p_x\p_yM_3)(\p,t-\tau)
[P_\backsim(\vec{R'}\cd\B)b]d\tau
\end{aligned}
$$
and
$$
\begin{aligned}
O_{33}=&-\int_0^{t/2}\p_x\p_y M_1(\p,t-\tau)\big[P_\backsim(\vec{R'}\cd\B)u\big]d\tau
+\int_0^{t/2}\p_y M_1(\p,t-\tau)\big[\p_xP_\backsim(\vec{R'}\cd\B)u\big]d\tau\\
=&-\int_0^{t/2}\p_x M_1(\p,t-\tau)\big[\p_y P_\backsim(\vec{R'}\cd\B)u\big]d\tau
-\int_0^{t/2}\p_x M_1(\p,t-\tau)\big[ P_\backsim(\vec{R'}\cd\B)\p_yu\big]d\tau\\
&+\int_0^{t/2}\p_y M_1(\p,t-\tau)\big[\p_xP_\backsim(\vec{R'}\cd\B)u\big]d\tau.
\end{aligned}
$$
\subsection{The estimate of (\ref{4.2})$_2$}
Due to (\ref{d1}), it suffices to bound $\|\p_yL_u\|_{\F L^2(D_4)}$.
We shall use (\ref{expu}) to achieve the goal.
Thanks to (\ref{rf2}), (\ref{End3})$_3$ and (\ref{End1})$_2$ for $r=2$ and $\delta=0.01$, we have
$$
\begin{aligned}
\|\p_yL_u\|_{\F L^2(D_4)}
\lesssim&\ \|\p_yM_3(\p,t)u_0\|_{\F L^2(D_4)}+\|\p_yM_1(\p,t)b_0\|_{\F L^2(D_4)}\\
\lesssim&\ \langle t\rangle^{-1}
\|u_0\|_{L^1\cap H^1}+\langle t\rangle^{-\frac{3}{4}}\|\langle\na\rangle^{0.51} b_0\|_{L^1_x(L^2_y)})\\
\lesssim&\ \langle t\rangle^{-\frac{3}{4}}\|V_0\|_3.
\end{aligned}
$$
We deduce from  (\ref{rf2}), (\ref{End3})$_3$, (\ref{ine1}) and (\ref{ine3})  that
$$
\begin{aligned}
\|\p_yNL_{u1}\|_{\F L^2(D_4)}
\lesssim&\ \int_0^t
\|\p_yM_3(\p,t-\tau)(\U\cd\na \U,
\B\cd\na B,\p_y(Bb))\|_{\F L^2(D_4)}d\tau\\
&+\int_0^t
\|\p_yM_3(\p,t-\tau)(b\p_x b)\|_{\F L^2(D_4)}d\tau\\
\lesssim&\ \int_0^t
\langle t-\tau\rangle^{-1}(
\|\B B\|_{H^2}+\|\U\otimes \U\|_{H^2})d\tau\\
&+\int_0^t\langle t-\tau\rangle^{-1} \|b\p_x b\|_{H^1\cap L^1}
d\tau\\
\lesssim&\ \int_0^t (\langle t-\tau\rangle^{-1}
\langle \tau\rangle^{-1.01}+\LH^{-1}\langle
\tau\rangle^{-1})d\tau \|V\|_3^2\\
\lesssim&\ \langle t\rangle^{-0.89}\|V\|_3^2.
\end{aligned}
$$
By (\ref{rf2}), (\ref{End1})$_2$ for $r=2$ and $\delta=0.01$, (\ref{ine2}),  one has
$$
\begin{aligned}
  &\ \ \ \|\p_y NL_{u2}\|_{\F L^2(D_4)}+\|\p_yNL_{u3}'\|_{\F L^2(D_4)}\\
  \lesssim&\ \int_0^t \|\p_yM_1(\p,t-\tau)\big(u\p_x b,\B\cd\na u,v_{<\langle \tau\rangle^{-8}}\p_yb,v_{>2\langle \tau\rangle^{-0.05}}\p_yb\big)\|_{\F L^2(D_4)}d\tau\\
  \lesssim&\ \int_0^t \LH^{-\frac{3}{4}}
  \|\langle \na\rangle^{0.51}
  \big(u\p_xb,\B\cd\na u,v_{<\langle \tau\rangle^{-8}}\p_yb,v_{>2\langle \tau\rangle^{-0.05}}\p_yb\big)\|_{L^1_x(L^2_y)}
  d\tau\\
  \lesssim&\ \int_0^t \LH^{-\frac{3}{4}}\langle \tau\rangle^{-1.01}d\tau\|V\|_3^2\\
  \lesssim&\ \langle t\rangle^{-\frac{3}{4}}\|V\|_3^2.
\end{aligned}
$$
Similarly, by (\ref{rf2}), (\ref{End1})$_2$ for $r=2$ and $\delta=0.01$, we infer
$$
\begin{aligned}
\|\p_yJ_1\|_{\F L^2(D_4)}\lesssim&\
\langle t\rangle^{-\frac{3}{4}}
\|\langle \na\rangle^{0.51}\{P_{\backsim1}
(\vec{R'}\cd\B_0)\p_yb_0\}\|_{L^1_x(L^2_y)}\\
\lesssim&\ \langle t\rangle^{-\frac{3}{4}}
\|\langle \na\rangle^{2}\{P_{\backsim1}
(\vec{R'}\cd\B_0)\p_yb_0\}\|_{L^1}\\
\lesssim&\ \langle t\rangle^{-\frac{3}{4}}
\|P_{\backsim1}(\vec{R'}\cd\B_0)\|_{H^2}
\|\p_yb_0\|_{H^2}\\
\lesssim&\ \langle t\rangle^{-\frac{3}{4}}
\|\B_0\|_{H^2}
\|\p_yb_0\|_{H^2}
\lesssim\ \langle t\rangle^{-\frac{3}{4}}\|V_0\|_3^2.
\end{aligned}
$$
For the estimate on $J_2$,  it is sufficient to bound $J_{21}$, $J_{22}$ and $J_{23}$. We obtain, by using (\ref{rf2}), (\ref{End3})$_6$, and (\ref{ine3}), that
$$
\begin{aligned}
&\ \ \|\p_y J_{21}\|_{\F L^2(D_4)}
+\|\p_y J_{22}\|_{\F L^2(D_4)}\\
\lesssim&\ \int_0^t
\|\p_y^2M_3(\p,t-\tau)
(\p_xP_\backsim (\vec{R'}\cd \B)b,
P_\backsim (\vec{R'}\cd \B)\p_xb)\|_{\F L^2(D_4)}d\tau\\
\lesssim&\ \int_0^t \LH^{-1}
(\|\p_xP_\backsim (\vec{R'}\cd \B)b\|_{H^2}
+\|P_\backsim (\vec{R'}\cd \B)\p_xb\|_{H^2})d\tau\\
\lesssim&\ \int_0^t \LH^{-1}\langle \tau\rangle^{-1.01}d\tau\ \|V\|_3^2\\
\lesssim&\ \langle t\rangle^{-1}\|V\|_3^2.
\end{aligned}
$$
For $J_{23}$, using (\ref{rf2}), (\ref{End3})$_8$ and (\ref{ine3}), we have
$$
\begin{aligned}
\|\p_y J_{23}\|_{\F L^2(D_4)}
\lesssim&\ \int_0^t \|\p_x\p_yM_3(\p,t-\tau)(\p_yP_\backsim (\vec{R'}\cd \B)b)\|_{\F L^2(D_4)}d\tau\\
\lesssim&\ \int_0^t \LH^{-\frac{3}{2}}
(\|\p_x\big(\p_yP_\backsim (\vec{R'}\cd \B)b\big)\|_{H^1}
+\|\p_yP_\backsim (\vec{R'}\cd \B)b\|_{L^2})d\tau\\
\lesssim&\ \int_0^t \LH^{-\frac{3}{2}}\langle \tau\rangle^{-\frac{3}{4}}d\tau\ \|V\|_3^2
\lesssim\ \langle t\rangle^{-\frac{3}{4}}\|V\|_3^2.
\end{aligned}
$$
It follows from (\ref{rf2}), (\ref{End1})$_3$ for $r=2$, (\ref{End1})$_2$ for $r=2$ and $\delta=0.01$, (\ref{ine3}) and (\ref{ine2}) that
$$
\begin{aligned}
\|\p_yJ_3\|_{\F L^2(D_4)}
\lesssim&\ \int_0^t
\|\p_x\p_yM_1(\p,t-\tau)(P_\backsim(\vec{R'}\cd \B)\p_y u)\|_{\F L^2(D_4)}d\tau\\
&+\int_0^t
\|\p_yM_1(\p,t-\tau)(\p_xP_\backsim(\vec{R'}\cd \B)\p_y u)\|_{\F L^2(D_4)}d\tau\\
\lesssim&\ \int_0^t \LH^{-1}
\|P_\backsim(\vec{R'}\cd \B)\p_y u\|_{H^1}d\tau\\
&+\int_0^t \LH^{-\frac{3}{4}}\|\langle \na\rangle^{0.51}\big(\p_x
P_\backsim(\vec{R'}\cd \B)\p_y u\big)\|_{L^1_x(L^2_y)}d\tau\\
\lesssim&\ \int_0^t (\LH^{-1} \langle \tau\rangle^{-1.01}
+\LH^{-\frac{3}{4}}\langle \tau\rangle^{-1.01})d\tau
\|V\|_3^2\\
\lesssim&\ \langle t\rangle^{-\frac{3}{4}}\|V\|_3^2.
\end{aligned}
$$
``Other good parts" on $D_4$ can be bounded by
$
C\langle t\rangle^{-\frac{3}{4}}(\|V\|_3^2+\|V\|_3^3)
$ by using similar arguments, so
 $$\|\p_y u\|_{\F L^2(D_4)}
 \lesssim\ \LHH^{-\frac{3}{4}}
 (\|V_0\|_3+\|V_0\|_3^2+\|V\|_3^2+\|V\|_3^3).$$
Hence, we have proved (\ref{4.2})$_2$.
\subsection{The estimate of (\ref{4.2})$_3$}
Thanks to (\ref{d1}),
it is sufficient  to estimate  $\|\p_x\langle\na\rangle^2 u\|_{\F L^2(D_4)}$.
Using  (\ref{rf2}), (\ref{End3})$_4$ and (\ref{End1})$_3$ for $r=2$, we get
$$
\begin{aligned}
\|\p_x\langle\na\rangle^2 L_u\|_{\F L^2(D_4)}
\lesssim&\ \|\p_x\langle\na\rangle^2M_3(\p,t) u_0\|_{\F L^2(D_4)}
+\|\p_x\langle\na\rangle^2M_1(\p,t) b_0\|_{\F L^2(D_4)}\\
\lesssim&\ \langle t\rangle^{-1}(\|u_0\|_{H^3}+\|b_0\|_{H^2})
\lesssim\ \langle t\rangle^{-1}\|V_0\|_3.
\end{aligned}
$$
Using (\ref{rf2}), (\ref{End3})$_4$ for $r=2$ and (\ref{End3})$_6$, and (\ref{ine3}), we can obtain
$$
\begin{aligned}
\|\p_x\langle\na\rangle^2 NL_{u1}\|_{\F L^2(D_4)}
\lesssim&\ \|\p_x NL_{u1}\|_{\F L^2(D_4)}
+\|\p_x\De NL_{u1}\|_{\F L^2(D_4)}
\\
\lesssim&\ \int_0^t \LH^{-1}\|(\U\cd\na \U,\B\cd\na \B)\|_{H^1}d\tau\\
&+\int_0^t \LH^{-\frac{3}{2}}\|\p_x(\U\cd\na \U,\B\cd\na \B)\|_{H^2}d\tau\\
\lesssim&\ \int_0^t (\LH^{-1}\langle\tau\rangle^{-1.01}
+\LH^{-\frac{3}{2}}\langle\tau\rangle^{-1})d\tau\|V\|_3^2\\
\lesssim&\ \langle t\rangle^{-1}\|V\|_3^2.
\end{aligned}
$$
We can deduce from (\ref{rf2}), (\ref{End1})$_3$ for $r=2$ and  (\ref{ine3}) that
$$
\begin{aligned}
\|\p_x\langle\na\rangle^2 \mathcal{NL}_{u}\|_{\F L^2(D_4)}
\lesssim&\ \int_0^t
\|\p_x\langle\na\rangle^2M_1(\p,t-\tau)(\U\cd\na b, \B\cd\na u)\|_{\F L^2(D_{4})}d\tau\\
\lesssim&\ \int_0^t \LH^{-1}(\|\U\cd\na b\|_{H^2}
+\|\B\cd\na u\|_{H^2})d\tau\\
\lesssim&\ \int_0^t \LH^{-1}\langle\tau\rangle^{-1.1}d\tau\|V\|_3^2
\lesssim\ \langle t\rangle^{-1}\|V\|_3^2.
\end{aligned}
$$
As a consequence, we have
$$\|\p_x\langle\na\rangle^2 u\|_{\F L^2(D_4)}
\lesssim
\LHH^{-1}(\|V_0\|_3+\|V\|_3^2).$$
Then we  complete the proof of
(\ref{4.2})$_3$.
\subsection{The estimate of (\ref{4.2})$_4$}
In this subsection, we shall use (\ref{expu2}) to bound $\|\widehat{u}\|_{L^1(D_4)}$. Like the derivation of the estimate of $\|L_v\|_{\F L^1(D_4)}$, one can get
$$
\|L_u\|_{\F L^1(D_4)}
\lesssim\ \langle t\rangle^{-1}\|V_0\|_3.
$$
Using (\ref{bbfen}), we have
$$
\begin{aligned}
\|NL_{u1}\|_{\F L^1(D_4)}
\lesssim&\ \int_0^t \|M_3(\p,t-\tau) F^1\|_{\F L^1(D_4)}d\tau\\
\lesssim&\ \int_0^t \|M_3(\p,t-\tau) (\U\cd\na\U,\B\cd\na B,
\p_y(Bb))\|_{\F L^1(D_4)}d\tau\\
&+\int_0^t \|\p_xM_3(\p,t-\tau)(bb) \|_{\F L^1(D_4)}d\tau.
\end{aligned}
$$
Using (\ref{rf2}), (\ref{End3})$_2$, (\ref{End3})$_5$, (\ref{ine1})-(\ref{ine4}),  the first integral can be bounded by
$$
\begin{aligned}
&\ \ C\int_0^t \LH^{-1}(\||\na|^{-1}(\U\cd\na\U,\B\cd\na B,
\p_y(Bb))\|_{L^2\cap \F L^1}+\|(\U\cd\na\U,\B\cd\na B,
\p_y(Bb))\|_{\F L^1})d\tau\\
\lesssim&\ \int_0^t \LH^{-1}\langle \tau\rangle^{-1.1}d\tau\|V\|_3^2
\lesssim\ \LHH^{-1}\|V\|_3^2,
\end{aligned}
 $$
 while the second integral can be bounded by
 $$
\begin{aligned}
&\ \ C(\int_0^\frac{t}{2}\LH^{-\frac{3}{2}}
(\|bb\|_{L^2\cap \F L^1}+\|b\p_xb\|_{\F L^1})d\tau\\
&+\int_{\frac{t}{2}}^t\LH^{-\frac{3}{4}}
(\|b\p_xb\|_{L^1_y(L^2_x)}+\|b\p_xb\|_{\F L^1})d\tau)\\
\lesssim&\ (
\int_0^\frac{t}{2}\LH^{-\frac{3}{2}}\langle\tau\rangle^{-\frac{1}
{2}}d\tau+\int_{\frac{t}{2}}^t\LH^{-\frac{3}{4}}
\langle\tau\rangle^{-\frac{5}{4}}d\tau)\|V\|_3^2\\
\lesssim&\ \LHH^{-1}\|V\|_3^2.
\end{aligned}
 $$
Thus
$$\|NL_{u1}\|_{\F L^1(D_4)}\lesssim
\ \langle t\rangle^{-1}\|V\|_3^2.$$
It follows from (\ref{rf2}), (\ref{End1})$_4$ and (\ref{ine1}), (\ref{ine2}) that
$$
\begin{aligned}
 \|NL_{u2}\|_{\F L^1(D_4)}
 \lesssim&\ \int_0^t
 \|M_1(\p,t-\tau)(u\p_x b,\B\cd\na u)\|_{\F L^1(D_4)}d\tau\\
 \lesssim&\ \int_0^t \LH^{-1}
 (\|\langle\na \rangle^{1.51}(u\p_xb,\B\cd\na u) \|_{L^1_x(L^2_y)}+\|(u\p_xb,\B\cd\na u)\|_{L^1})d\tau\\
 \lesssim&\ \int_0^t \LH^{-1}\langle \tau\rangle^{-1.01}d\tau\|V\|_3^2\\
 \lesssim&\ \langle t\rangle^{-1}\|V\|_3^2.
\end{aligned}
$$
Using (\ref{rf2}),(\ref{End1})$_2$ for $r=1$ and $\delta=0.01$, (\ref{End1})$_4$,
(\ref{ine1}) and (\ref{ine2}) that
$$
\begin{aligned}
 \|NL_{u3}'\|_{\F L^1(D_4)}
 \lesssim&\ \int_0^t\|M_1(\p,t-\tau)(v_{<\langle\tau\rangle^{-8}}\p_yb,v_{>2
 \langle\tau\rangle^{-0.05}}\p_yb)\|_{\F L^1(D_4)}d\tau\\
 \lesssim&\ \int_0^t\|\p_yM_1(\p,t-\tau)(v_{<\langle\tau\rangle^{-8}}b,v_{>2
 \langle\tau\rangle^{-0.05}}b)\|_{\F L^1(D_4)}d\tau\\
 &+\int_0^t\|M_1(\p,t-\tau)(\p_yv_{<\langle\tau\rangle^{-8}}b,\p_yv_{>2
 \langle\tau\rangle^{-0.05}}b)\|_{\F L^1(D_4)}d\tau\\
 \lesssim&\ \int_0^t
 \LH^{-1}
 \|\langle \na\rangle^{1.51}(v_{<\langle\tau\rangle^{-8}}b,
v_{>2
 \langle\tau\rangle^{-0.05}}b)\|_{L^1_x(L^2_y)}d\tau\\
 &+\int_0^t
 \LH^{-1}
 \Big(\|(\p_xu_{<\langle\tau\rangle^{-8}}b,
\p_xu_{>2
 \langle\tau\rangle^{-0.05}}b)\|_{L^1}\\
& +\|\langle \na\rangle^{1.51}(\p_xu_{<\langle\tau\rangle^{-8}}b,
\p_xu_{>2
 \langle\tau\rangle^{-0.05}}b)\|_{L^1_x(L^2_y)}\Big)d\tau\\
 \lesssim&\ \int_0^t \LH^{-1}\langle \tau\rangle^{-1.01}d\tau
 \|V\|_3^2\lesssim\ \langle t\rangle^{-1}\|V\|_3^2.
\end{aligned}
$$
We can get by using (\ref{rf2}), (\ref{End1})$_4$,
 $$\|\p_xP_\backsim u b\|_{ L^1}
 \lesssim\ \|\p_xu\|_{L^2}\|b\|_{L^2}
 \lesssim\ \LHH^{-\frac{5}{4}}\|V\|_3^2$$
 and (\ref{ine2}) that
$$
\begin{aligned}
\|\widehat{O_1}\|_{L^1(D_4)}
\lesssim&\ \int_0^t \|M_1(\p,t-\tau)(\p_xP_\backsim ub)\|_{\F L^1(D_4)}d\tau\\
\lesssim&\ \int_0^t
\LH^{-1}(\|\p_xP_\backsim u b\|_{ L^1}
+\|\langle\na\rangle^{1.51}(\p_xP_\backsim ub)\|_{L^1_x(L^2_y)})d\tau\\
\lesssim&\ \int_0^t \LH^{-1}\langle\tau\rangle^{-1.01}d\tau
\|V\|_3^2
\lesssim\ \LHH^{-1}\|V\|_3^2.
\end{aligned}
$$
Applying (\ref{rf2}), (\ref{End1})$_2$, we have
$$
\begin{aligned}
\|\widehat{O_2}\|_{L^1(D_4)}
\lesssim&\ \int^t_{t/2} \|\p_yM_1(\p,t-\tau)(P_\backsim vb)\|_{\F L^1(D_4)}d\tau\\
\lesssim&\ \int^t_{t/2} \LH^{-\frac{1}{2}}
\|\widehat{v}\|_{L^1}\|\widehat{b}\|_{L^1}d\tau\\
\lesssim&\ \int^t_{t/2} \LH^{-\frac{1}{2}} \langle\tau\rangle^{-\frac{3}{2}}d\tau
\|V\|_3^2\\
\lesssim&\ \LHH^{-1}\|V\|_3^2.
\end{aligned}
$$
It follows from (\ref{rf2}), (\ref{End1})$_2$, and (\ref{ine4}), (\ref{ine2}) that
$$
\begin{aligned}
\|\widehat{O_{31}}\|_{L^1(D_4)}
\lesssim&\
\|\p_yM_1(\p,\frac{t}{2})
\big(P_{\thickapprox}
(\vec{R'}\cd\B) b\big)(\frac{t}{2})\|_{\F L^1(D_4)}
+\|\p_yM_1(\p,t)(P_{\backsim1}
(\vec{R'}\cd\B_0) b_0)\|_{\F L^1(D_4)}\\
\lesssim&\ \LHH^{-\frac{1}{2}}\|\big[P_{\thickapprox}
(\vec{R'}\cd\B) b\big](\frac{t}{2})\|_{\F L^1}
+\LHH^{-1}\|\langle\na\rangle^{1.51}(P_{\backsim1}(\vec{R'}\cd \B_0)b_0)\|_{L^1_x(L^2_y)}\\
\lesssim&\ \LHH^{-1}(\|V\|_3^2+\|V_0\|_3^2).
\end{aligned}
$$
Using (\ref{rf2}), (\ref{End3})$_8$, and (\ref{ine4}), we can deduce
$$
\begin{aligned}
\|\widehat{O_{32}}\|_{L^1(D_4)}
\lesssim&\ \int_0^{t/2}
\|\p_x\p_yM_3(\p,t-\tau)(P_\backsim(\vec{R'}\cd\B)b)\|_{\F L^1(D_4)}d\tau\\
\lesssim&\ \int_0^{t/2}
\LH^{-\frac{3}{2}}
(\|P_\backsim(\vec{R'}\cd\B)b\|_{\F L^1}+
\|\p_x\p_y(P_\backsim(\vec{R'}\cd\B)b)\|_{\F L^1})d\tau\\
\lesssim&\ \int_0^{t/2}
\LH^{-\frac{3}{2}}\langle \tau\rangle^{-0.99}d\tau\|V\|_3^2
\lesssim\ \LHH^{-1.4}\|V\|_3^2.
\end{aligned}
$$
Thanks to (\ref{rf2}), (\ref{End1})$_3$ for $r=1$ and (\ref{End1})$_2$ for $r=1$, one has
$$
\begin{aligned}
\|\widehat{O_{33}}\|_{L^1(D_4)}
\lesssim&\ \int_0^{t/2}
\|\p_xM_1(\p,t-\tau)(\p_yP_\backsim(\vec{R'}\cd\B)u
)\|_{\F L^1(D_4)}d\tau\\
&+\int_0^{t/2}
\|\p_xM_1(\p,t-\tau)(P_\backsim(\vec{R'}\cd\B)\p_yu
)\|_{\F L^1(D_4)}d\tau\\
&+\int_0^{t/2}
\|\p_yM_1(\p,t-\tau)(\p_xP_\backsim(\vec{R'}\cd\B)u
)\|_{\F L^1(D_4)}d\tau\\
\lesssim&\ \int_0^{t/2} \LH^{-1}\big(\|\p_yP_\backsim(\vec{R'}\cd\B)u\|_{\F L^1}
+\|P_\backsim(\vec{R'}\cd\B)\p_yu\|_{\F L^1}\\
&+\|\langle\na\rangle^{1.51}(\p_xP_\backsim(\vec{R'}\cd\B)
u)\|_{L^1_x(L^2_y)}\big)d\tau\\
\lesssim&\ \int_0^{t/2} \LH^{-1}\langle \tau\rangle^{-1.01}
d\tau\|V\|_3^2\lesssim\ \LHH^{-1}\|V\|_3^2.
\end{aligned}
$$
The estimate on ``other good parts"
can be bounded by
$C\langle t\rangle^{-1}(\|V\|_3^2+\|V\|_3^3)$. Hence,
 we can obtain
$$\|u\|_{\F L^1(D_4)}\lesssim\ \langle t\rangle^{-1}
(\|V_0\|_3+\|V_0\|_3^2+\|V\|_3^2+\|V\|_3^3),$$
which completes the proof of (\ref{4.2})$_4$.
\subsection{The estimate of (\ref{4.2})$_5$}
Using (\ref{rf1}),
$|\xe|\lesssim 1$ in $D_1$,
(\ref{decay1})$_3$ for $k=1$, we have
$$
\|\p_xL_u\|_{\F L^1(D_1)}\le\
\|\p_xM_3(\p,t)u_0\|_{\F L^1(D_1)}+\|\p_xM_1(\p,t)b_0\|_{\F L^1(D_1)}\lesssim
\langle t\rangle^{-\frac{3}{2}}(\|u_0\|_{L^1}+\|b_0\|_{L^1}).
$$
Thanks to (\ref{rf1}), (\ref{decay1})$_3$ for $k=1$, (\ref{ine2}) and (\ref{ine3}), one can get
$$
\begin{aligned}
 \|\p_xNL_{u1}\|_{\F L^1(D_1)}
 \le&\ \int_0^t \|\p_xM_3(\p,t-\tau)F^1\|_{\F L^1(D_1)}d\tau\\
 \lesssim&\ \int_0^t \|\p_xM_3(\p,t-\tau)(\U\cd\na\U,\B\cd\na B,\p_y(Bb))\|_{\F L^1(D_1)}d\tau\\
 &+\int_0^t \|\p_xM_3(\p,t-\tau)(b\p_xb)\|_{\F L^1(D_1)}d\tau\\
\lesssim&\ \int_0^t \LH^{-\frac{3}{2}}
\||\na|^{-1}(\U\cd\na\U,\B\cd\na B,\p_y(Bb))\|_{L^2}d\tau\\
&+\int_0^t \LH^{-\frac{5}{4}}\|b\p_xb\|_{L^1_y(L^2_x)}d\tau\\
\lesssim&\ \int_0^t \LH^{-\frac{3}{2}}\langle \tau\rangle^{-\frac{5}{4}}
d\tau\|V\|_3^2
\lesssim\ \LHH^{-\frac{5}{4}}\|V\|_3^2.
\end{aligned}
$$
It follows from (\ref{rf1}), (\ref{decay1})$_3$ for $k=1$ and
$$\|\p_x(b\U-u\B)\|_{L^1}\lesssim\ \|\B\|_{L^2}\|\p_x\U\|_{L^2}
+\|\U\|_{L^2}\|\p_x\B\|_{L^2}
\lesssim\ \LHH^{-\frac{5}{4}}\|V\|_3^2,$$
 that
$$
\begin{aligned}
 \|\p_x\mathcal{NL}_{u}\|_{\F L^1(D_1)}
 \le&\ \int_0^t \|\p_xM_1(\p,t-\tau){\rm div}(b\U-u\B)\|_{\F L^1(D_1)}d\tau\\
 \lesssim&\ \int_0^t \|M_1(\p,t-\tau){\rm div}\big[\p_x(b\U-u\B)\big]\|_{\F L^1(D_1)}d\tau\\
 \lesssim&\ \int_0^t \LH^{-\frac{3}{2}}\|\p_x(b\U-u\B)\|_{L^1}d\tau\\
 \lesssim&\ \int_0^t \LH^{-\frac{3}{2}}\langle \tau\rangle^{-\frac{5}{4}}d\tau\|V\|_3^2
 \lesssim\ \LHH^{-\frac{5}{4}}\|V\|_3^2.
\end{aligned}
$$
Thus
\begin{equation}\label{f11}
\|\p_xu\|_{\F L^1(D_1)}
\lesssim\ \LHH^{-\frac{5}{4}}(\|V_0\|_3+\|V\|_3^2).
\end{equation}
Next, we bound the estimate on $D_4$. Using (\ref{rf2}), (\ref{End3})$_5$, and (\ref{End1})$_3$ for $r=1$, we  find
$$
\begin{aligned}
\|\p_xL_u\|_{\F L^1(D_4)}\lesssim&\
\|\p_xM_3(\p,t)u_0\|_{\F L^1(D_4)}
+\|\p_xM_1(\p,t)b_0\|_{\F L^1(D_4)}\\
\lesssim&\ \langle t\rangle^{-\frac{3}{2}}(\|u_0\|_{L^2\cap \F L^1}
+\|\p_xu\|_{\F L^1}+\|\widehat{b_0}\|_{L^1})\\
\lesssim&\
\langle t\rangle^{-\frac{3}{2}}
\|V_0\|_3.
\end{aligned}
$$
Applying \ref{rf2}), (\ref{End3})$_8$ for $r=1$, (\ref{End3})$_9$, (\ref{End3})$_4$ for $r=1$, and (\ref{ine4}), one has
$$
\begin{aligned}
\|\p_xNL_{u1}\|_{\F L^1(D_4)}
\lesssim&\ \int_0^t
\|\p_xM_3(\p,t-\tau)F^1\|_{\F L^1(D_4)}d\tau\\
\lesssim&\ \int_0^t
\|\p_xM_3(\p,t-\tau)(\U\cd\na \U,\B\cd\na B,\p_y(Bb))\|_{\F L^1(D_4)}d\tau\\
&+\int_0^t
\|\p_x^2M_3(\p,t-\tau)(bb)\|_{\F L^1(D_4)}d\tau\\
\lesssim&\ \int_0^t \LH^{-\frac{3}{2}}
(\|\U\otimes\U\|_{\F L^1}+\|\p_x(\U\cd\na \U)\|_{\F L^1}\\
&+\|B\B\|_{\F L^1}+\|\p_x\na (\B B)\|_{\F L^1})d\tau\\
&+\int_0^\frac{t}{2}\LH^{-2} (\|bb\|_{\F L^1}+\|\p_x(b\p_xb)\|_{\F L^1})d\tau\\
&+\int_{\frac{t}{2}}^t\LH^{-1} (\|b\p_xb\|_{\F L^1}+\|\p_x(b\p_xb)\|_{\F L^1})d\tau\\
\lesssim&\ (\int_0^t\LH^{-\frac{3}{2}}
\langle \tau\rangle^{-\frac{5}{4}}
d\tau
+\int_0^\frac{t}{2}\LH^{-2} \langle \tau\rangle^{-1}d\tau\\
&+\int_{\frac{t}{2}}^t\LH^{-1} \langle \tau\rangle^{-1.3}d\tau )\|V\|_3^2
\lesssim\ \langle t\rangle^{-\frac{5}{4}}\|V\|_3^2.
\end{aligned}
$$
We can infer from (\ref{rf2}), (\ref{End1})$_3$ for $r=1$ and $\delta=0.01$ that
$$
\begin{aligned}
\|\p_x\mathcal{NL}_{u}\|_{\F L^1(D_4)}
\lesssim&\ \int_0^t
\|\p_xM_1(\p,t-\tau)(\U\cd\na \B,\B\cd\na \U)\|_{\F L^1(D_4)}d\tau\\
\lesssim&\ \int_{\frac{t}{2}}^t\LH^{-1}
\|(\U\cd\na \B,\B\cd\na \U)\|_{\F L^1(D_4)}d\tau\\
&+\int_0^{\frac{t}{2}}\LH^{-\frac{3}{2}}
\|\langle \na\rangle^{1.51}(\U\cd\na\U,\B\cd\na\B)\|_{L^1_x(L^2_y)}
d\tau\\
\lesssim&\ (\int_{\frac{t}{2}}^t
\LH^{-1}\langle \tau\rangle^{-1.3} d\tau
+\int_0^{\frac{t}{2}}
\LH^{-\frac{3}{2}}
\langle \tau\rangle^{-0.75} d\tau) \|V\|_3^2\\
\lesssim&\ \langle t\rangle^{-\frac{5}{4}}\|V\|_3^2.
\end{aligned}
$$
Therefore, it comes out
\begin{equation}\label{f12}
\|\p_xu\|_{\F L^1(D_4)}
\lesssim\ \LHH^{-\frac{5}{4}}(\|V_0\|_3+\|V\|_3^2).
\end{equation}
Collecting the above estimates (\ref{f11}) and (\ref{f12}) can yield (\ref{4.2})$_5$.
\vskip .3in
\section{The  estimate on $b$}
\label{s4}
\vskip.1in
In this section, we will prove
\begin{equation}\label{4.3}
\left\{
\begin{aligned}
\|b(t)\|_{L^2}\lesssim\ &\langle t\rangle^{-\frac{1}{4}}
\big(\|V_0\|_3+\|V_0\|_3^2+\|V\|_3^2+\|V\|_3^3\big);\\
\|\p_xb(t)\|_{H^1}\lesssim\ &\langle t\rangle^{-\frac{3}{4}}
\big(\|V_0\|_3+\|V_0\|_3^2+\|V\|_3^2+\|V\|_3^3\big);\\
\||\na|^{-1}\langle\na\rangle b(t)\|_{\F L^1}\lesssim\ &\langle t\rangle^{-\frac{1}{2}}
\big(\|V_0\|_3+\|V_0\|_3^2+\|V\|_3^2+\|V\|_3^3\big);\\
\|\mathcal{R}_1\langle\na\rangle b(t)\|_{\F L^1}\lesssim\ &\langle t\rangle^{-1}
\big(\|V_0\|_3+\|V_0\|_3^2+\|V\|_3^2+\|V\|_3^3\big).
\end{aligned}
\right.
\end{equation}
\subsection{The expression of $b$}
Recalling (\ref{B}),
it is easy to get
$$
\begin{aligned}
b=&M_{1}(\p,t)u_0+M_{2}(\p,t)b_0
+\underbrace{\int_0^t M_{1}(\p,t-\tau)F^1d\tau
+\int_0^t M_{2} (\p,t-\tau)G^1d\tau}_{\mathcal{NL}_b}\\
=&\underbrace{M_{1}(\p,t)u_0+M_{2}(\p,t)b_0}_{L_b}+ \underbrace{\int_0^tM_{1}(\p,t-\tau)F^1d\tau}_{NL_{b1}}\\
&+\underbrace{\int_0^t M_{2}(\p,t-\tau)G^{11}d\tau}_{NL_{b2}}
-\underbrace{\int_0^t M_{2}(\p,t-\tau)v\p_ybd\tau}_{NL_{b3}}.
\end{aligned}
$$
As the same reason in the previous section, we need some new expression of $b$ to overcome the difficulty coming from the estimate of $v\p_yb$. Similar to the derivation of
  the expression of $NL_{u3}$,  we can get the expression of $NL_{b3}$. As a matter of fact,
$$NL_{b3}=NL_{b3}'+NL_{b3}'',$$
where
$$
\begin{aligned}
NL_{b3}'=&\int_0^t M_{2}(\p,t-\tau)(v_{<\langle \tau\rangle^{-8}}\p_yb)d\tau
+\int_0^t M_{2}(\p,t-\tau)(v_{>2\langle \tau\rangle^{-0.05}}\p_yb)d\tau\\
=&\int_0^t \p_yM_{2}(\p,t-\tau)(v_{<\langle \tau\rangle^{-8}}b+v_{>2\langle \tau\rangle^{-0.05}}b)d\tau\\
&-\int_0^t M_{2}(\p,t-\tau)(\p_yv_{<\langle \tau\rangle^{-8}}b+\p_yv_{>2\langle \tau\rangle^{-0.05}}b)d\tau,
\end{aligned}
$$
and
$$
\begin{aligned}
NL_{b3}''=&  -M_{2}(\p,t)\big(
  P_{\backsim1} (\vec{R'}\cd \B_0)\p_y b_0 \big)
-\int_0^t \Big(\p_\tau M_{2}(\p,t-\tau)\Big)\Big[
  P_\backsim (\vec{R'}\cd \B)\p_y b \Big]d\tau\\
  & -\int_0^tM_{2}(\p,t-\tau)\Big[
  P_\backsim (\vec{R'}\cd \B)\p_y \p_x u \Big]d\tau
  +\ {\rm ``other\ \  good\ \  parts"}\\
  =& L_1+L_2+L_3+\ {\rm ``other\ \  good\ \  parts"}.
\end{aligned}
$$
Thus we have
$$b=L_b+NL_{b1}+NL_{b2}-NL_{b3}'-\sum_{i=1}^3L_i-{\rm ``other\ \  good\ \  parts"},$$
where, by (\ref{tm}),
$L_2$ and $L_3$ can be rewrite as follows
$$
\begin{aligned}
L_2=&\int_0^t \Big(\p_x M_{1}\Big)(\p,t-\tau)\Big[
  P_\backsim (\vec{R'}\cd \B)\p_y b \Big]d\tau\\
  =&\int_0^t \Big(\p_y\p_x M_{1}\Big)(\p,t-\tau)\Big[
  P_\backsim (\vec{R'}\cd \B) b \Big]d\tau
  -\int_0^t \Big(\p_x M_{1}\Big)(\p,t-\tau)\Big[
  \p_yP_\backsim (\vec{R'}\cd \B) b \Big]d\tau\\
  =&\int_0^t \Big(\p_y M_{1}(\p,t-\tau)\Big)\p_x\Big[
  P_\backsim (\vec{R'}\cd \B) b \Big]d\tau
-\int_0^t \Big(\p_x M_{1}(\p,t-\tau)\Big)\Big[
  \p_yP_\backsim (\vec{R'}\cd \B) b \Big]d\tau
  \end{aligned}
  $$
  and
$$L_3=-\int_0^t\p_xM_{2}(\p,t-\tau)\Big[
  P_\backsim (\vec{R'}\cd \B)\p_y u \Big]d\tau
  +\int_0^tM_{2}(\p,t-\tau)\Big[
  \p_xP_\backsim (\vec{R'}\cd \B)\p_y u \Big]d\tau.$$
  Following the idea dealing with $\|\widehat{v}\|_{L^2(D_1)}$ and $\|\p_xv\|_{\F L^2(D_1)}$, one can easily obtain
  \begin{equation}\label{000}
  \begin{aligned}
\|\widehat{b}\|_{L^2(D_1)}
\lesssim&\ \LHH^{-\frac{1}{2}}(\|V_0\|_3+\|V\|_3^2),\\
\|\p_x\langle\na\rangle b\|_{\F L^2(D_1)}
\lesssim&\ \|\p_x b\|_{\F L^2(D_1)}\lesssim\ \langle t\rangle^{-1}
(\|V_0\|_3+\|V\|_3^2).
\end{aligned}
\end{equation}
\subsection{The estimate of (\ref{4.3})$_1$}
Thanks to (\ref{000}), it suffices  to give the estimate of $\|\widehat{b}\|_{L^2(D_4)}$.
Using (\ref{rf2}), (\ref{End1})$_1$, (\ref{End2})$_1$ for $r=2$, we can get
$$
\begin{aligned}
\|L_b\|_{\F L^2(D_4)}\lesssim&\
\|M_1(\p,t)u_0\|_{\F L^2(D_4)}+\|M_2(\p,t)b_0\|_{\F L^2(D_4)}\\
\lesssim&\ \LHH^{-\frac{1}{4}}
(\|u_0\|_{L^1\cap L^2}+\|\langle\na\rangle b_0\|_{L^1_x(L^2_y)})\\
\lesssim&\
\langle t\rangle^{-\frac{1}{4}}\|V_0\|_3.
\end{aligned}
$$
One can get from  (\ref{rf2}), (\ref{End1})$_1$, (\ref{End2})$_1$ for $r=2$,
(\ref{ine2}) and (\ref{ine3}) that
$$
\begin{aligned}
\sum_{i=1,2}\|NL_{bi}\|_{\F L^2(D_4)}
\lesssim&\ \int_0^t
(\|M_1(\p,t-\tau)F^1\|_{\F L^2(D_4)}+
\|M_2(\p,t-\tau)G^{11}\|_{\F L^2(D_4)})d\tau\\
\lesssim&\ \int_0^t(\LH^{-\frac{1}{2}}
\||\na|^{-1}F^1\|_{L^2}
+\LH^{-\frac{1}{4}}\|\langle\na\rangle^{0.51}G^{11}\|_{L^1_x(L^2_y)}
)d\tau\\
\lesssim&\ \int_0^t \LH^{-\frac{1}{2}}\|(\U\otimes\U,\B\otimes\B)\|_{L^2}d\tau\\
&+\int_0^t \LH^{-\frac{1}{4}}
\|\langle\na\rangle^{0.51}(u\p_xb,\B\cd\na u)\|_{L^1_x(L^2_y)}d\tau\\
 \lesssim&\ \int_0^t
 (\LH^{-\frac{1}{2}}\langle \tau\rangle^{-\frac{3}{4}}
 +\LH^{-\frac{1}{4}}\langle \tau\rangle^{-1.01})d\tau
 \|V\|_3^2\\
 \lesssim&\ \langle t\rangle^{-\frac{1}{4}}\|V\|_3^2.
\end{aligned}
$$
Thanks to  (\ref{rf2}), (\ref{End2})$_1$ for $r=2$, and (\ref{ine2}), we can obtain
$$
\begin{aligned}
\|NL_{b3}'\|_{\F L^2(D_4)}
\lesssim&\ \int_0^t
\|M_2(\p,t-\tau)(v_{<\langle \tau\rangle^{-8}}\p_yb,v_{>2\langle \tau\rangle^{-0.05}}\p_yb)\|_{\F L^2(D_4)}d\tau\\
\lesssim&\
\int_0^t \LH^{-\frac{1}{4}}
\|\langle \na\rangle^{0.51}(v_{<\langle \tau\rangle^{-8}}\p_yb,v_{>2\langle \tau\rangle^{-0.05}}\p_yb)\|_{L^1_x(L^2_y)}d\tau\\
\lesssim
&\ \int_0^t \LH^{-\frac{1}{4}}
\langle \tau\rangle^{-1.01}d\tau \|V\|_3^2\\
\lesssim&\ \langle t\rangle^{-\frac{1}{4}}\|V\|_3^2,
\end{aligned}
$$
and
$$
\begin{aligned}
\|L_1\|_{\F L^2(D_4)}
\lesssim&\ \|M_2(\p,t)(P_{\backsim1}
(\vec{R'}\cd\B_0)\p_yb_0)\|_{\F L^2(D_4)}
\lesssim\ \langle t\rangle^{-\frac{1}{4}}
\|\langle \na\rangle^{0.51}(P_{\backsim1}
(\vec{R'}\cd\B_0)\p_yb_0)\|_{L^1_x(L^2_y)}\\
\lesssim&\ \langle t\rangle^{-\frac{1}{4}}
\|\langle \na\rangle^{2}(P_{\backsim1}
(\vec{R'}\cd\B_0)\p_yb_0)\|_{L^1}
\lesssim\ \langle t\rangle^{-\frac{1}{4}}
\|P_{\backsim1}
(\vec{R'}\cd\B_0)\|_{H^2}\|\p_yb_0\|_{H^2}\\
\lesssim&\ \langle t\rangle^{-\frac{1}{4}}
\|V_0\|_3^2.
\end{aligned}
$$
It follows from  (\ref{rf2}), (\ref{End1})$_3$ for $r=2$ and (\ref{ine3}) that
$$
\begin{aligned}
\|L_2\|_{\F L^2(D_4)}
\lesssim&\ \int_0^t
\|\p_xM_1(\p,t-\tau)
(P_\backsim(\vec{R'}\cd\B)\p_yb)\|_{\F L^2(D_4)}d\tau\\
\lesssim&\ \int_0^t
\LH^{-1} \|P_\backsim(\vec{R'}\cd\B)\p_yb\|_{L^2}d\tau\\
\lesssim&\ \int_0^t \LH^{-1}\langle \tau\rangle^{-0.6}d\tau
\|V\|_3^2\\
\lesssim&\ \langle t\rangle^{-0.39}\|V\|_3^2.
\end{aligned}
$$
Using (\ref{rf2}), (\ref{End2})$_1$ for $r=2$, (\ref{End2})$_2$ for $r=2$, (\ref{ine3})
and (\ref{ine2}), one can deduce
$$
\begin{aligned}
  \|L_3\|_{\F L^2(D_4)}\lesssim&
  \ \int_0^t \|\p_xM_2(\p,t-\tau) (P_\backsim(\vec{R'}\cd\B)\p_yu)\|_{L^2(D_4)}d\tau\\
  &+\int_0^t \|M_2(\p,t-\tau)(\p_xP_\backsim(\vec{R'}\cd\B)\p_yu)
  \|_{L^2(D_4)}d\tau\\
  \lesssim&\ \int_0^t \LH^{-\frac{1}{2}}
  \|\na (P_\backsim(\vec{R'}\cd\B)\p_yu)\|_{L^2}d\tau\\
  &+\int_0^t \LH^{-\frac{1}{4}}\|\langle \na\rangle^{0.51}
  (\p_xP_\backsim(\vec{R'}\cd\B)\p_yu)\|_{L^1_x(L^2_y)}d\tau\\
  \lesssim&\
  \int_0^t (\LH^{-\frac{1}{2}}\langle \tau\rangle^{-1.1}+\LH^{-\frac{1}{4}}\langle \tau\rangle^{-1.1}
  )d\tau
  \|V\|_3^2\\
\lesssim&
\ \langle t\rangle^{-\frac{1}{4}}\|V\|_3^2.
\end{aligned}
$$
Like the previous arguments, the associated estimate of ``other good parts" on $D_4$  can be bounded by
$C\langle t\rangle^{-\frac{1}{4}}(\|V\|_3^2+\|V\|_3^3)$.
Collecting the above estimate yields
$$\|\widehat{b}\|_{L^2(D_4)}
\lesssim\ \langle t\rangle^{-\frac{1}{4}}(\|V_0\|_3+\|V_0\|_3^2+
\|V\|_3^2+\|V\|_3^3).$$
which
 completes the estimate of (\ref{4.3})$_1$.
\subsection{The estimate of (\ref{4.3})$_2$}

Due to
$$\|\p_xb\|_{H^1}\thickapprox \|\p_x\langle\na\rangle b\|_{L^2} $$
and Plancherel's identity, it is sufficient to bound
$\|\p_x\langle\na\rangle b\|_{\F L^2}$. Thanks to (\ref{000}),
 we only aim at estimating $\|\p_x\langle\na\rangle b\|_{\F L^2(D_4)}$.
 Using (\ref{rf2}), (\ref{End1})$_3$ for $r=2$, (\ref{End2})$_2$ for $r=2$ and $\delta=0.01$, we have
$$
\begin{aligned}
\|\p_x\langle\na\rangle L_b\|_{\F L^2(D_4)}
\lesssim&\
\|\p_x\langle\na\rangle M_1(\p,t)u_0\|_{\F L^2(D_4)}
+\|\p_x\langle\na\rangle M_2(\p,t)b_0\|_{\F L^2(D_4)}\\
\lesssim&\ \LHH^{-1}\|u_0\|_{H^1}+\LHH^{-\frac{3}{4}}
\|\langle \na\rangle^{2.51} b_0\|_{L^1_x(L^2_y)}\\
\lesssim&\ \langle t\rangle^{-\frac{3}{4}}\|V_0\|_3.
\end{aligned}
$$
We can deduce from (\ref{rf2}), (\ref{End1})$_3$ for $r=2$, and (\ref{ine3}) that
$$
\begin{aligned}
\|\p_x\langle\na\rangle NL_{b1}\|_{\F L^2(D_4)}
\lesssim&\ \int_0^t
\|\p_xM_1(\p,t-\tau)
\langle\na\rangle F^1\|_{\F L^2(D_4)}d\tau\\
\lesssim&\ \int_0^t \LH^{-1}\|\langle\na\rangle(\U\cd\na \U,\B\cd\na \B)\|_{L^2}d\tau\\
\lesssim&\ \int_0^t \LH^{-1}\langle \tau\rangle^{-1.01}d\tau\|V\|_3^2
\lesssim\ \langle t\rangle^{-1}\|V\|_3^2.
\end{aligned}
$$
Applying (\ref{rf2}), (\ref{End2})$_2$ for $r=2$ and $\delta=0.01$, and (\ref{ine2}), we get
$$
\begin{aligned}
&\ \ \ \|\p_x\langle\na\rangle NL_{b2}\|_{\F L^2(D_4)}
+\|\p_x\langle\na\rangle NL_{b3}'\|_{\F L^2(D_4)}\\
\lesssim&\ \int_0^t
\|\p_xM_2(\p,t-\tau) \langle\na\rangle
(G^{11},v_{<\langle \tau\rangle^{-8}}\p_yb,v_{>2\langle \tau\rangle^{-0.05}}\p_yb)\|_{\F L^2(D_4)}d\tau\\
\lesssim&\ \int_0^t \LH^{-\frac{3}{4}}
\|\langle\na\rangle^{2.51}(u\p_xb,\B\cd\na u,v_{<\langle \tau\rangle^{-8}}\p_yb,v_{>2\langle \tau\rangle^{-0.05}}\p_yb)\|_{L^1_x(L^2_y)}d\tau\\
\lesssim&\ \int_0^t
\LH^{-\frac{3}{4}}\langle \tau\rangle^{-1.01}d\tau
\|V\|_3^2
\lesssim\ \langle t\rangle^{-\frac{3}{4}}\|V\|_3^2.
\end{aligned}
$$
Similarly, by (\ref{rf2}), (\ref{End2})$_2$ for $r=2$ and $\delta=0.01$, one has
$$
\begin{aligned}
\|\p_x\langle\na\rangle L_1\|_{\F L^2(D_4)}
\lesssim&\ \|\p_x\langle\na\rangle M_2(\p,t)
(P_{\backsim1}
(\vec{R'}\cd\B_0)\p_yb_0)\|_{\F L^2(D_4)}\\
\lesssim&\ \langle t\rangle^{-\frac{3}{4}} \|\langle\na\rangle^{2.51}(P_{\backsim1}
(\vec{R'}\cd\B_0)\p_yb_0)\|_{L^1_x(L^2_y)}\\
\lesssim&\ \langle t\rangle^{-\frac{3}{4}}\|\langle\na\rangle^{4}(P_{\backsim1}
(\vec{R'}\cd\B_0)\p_yb_0)\|_{L^1}\\
\lesssim&\ \langle t\rangle^{-\frac{3}{4}}\|P_{\backsim1}
(\vec{R'}\cd\B_0)\|_{H^4}\|\p_yb_0\|_{H^4}\\
\lesssim&\ \langle t\rangle^{-\frac{3}{4}}\|V_0\|_3^2.
\end{aligned}
$$
By using (\ref{rf2}), (\ref{End1})$_3$ for $r=2$, (\ref{ine3}), we obtain,
$$
\begin{aligned}
\|\p_x\langle\na\rangle L_2\|_{\F L^2(D_4)}
\lesssim&\ \int_0^t
\|\p_xM_1(\p,t-\tau)
\langle\na\rangle\p_x(P_\backsim(\vec{R'}\cd\B)\p_yb)\|_{\F L^2(D_4)}d\tau\\
\lesssim&\ \int_0^t
\LH^{-1} \|\langle\na\rangle\p_x(P_\backsim(\vec{R'}\cd\B)\p_yb)\|_{L^2}d\tau\\
\lesssim&\ \int_0^t \LH^{-1}\langle \tau\rangle^{-1.01}d\tau
\|V\|_3^2\\
\lesssim&\ \langle t\rangle^{-1}\|V\|_3^2.
\end{aligned}
$$
One deduces from (\ref{rf2}), (\ref{End2})$_3$, (\ref{End2})$_2$ for $r=2$  and $\delta=0.01$ that
$$
\begin{aligned}
  \|\p_x\langle\na\rangle L_3\|_{\F L^2(D_4)}\lesssim&
  \ \int_0^t \|\p_x^2M_2(\p,t-\tau) \langle\na\rangle (P_\backsim(\vec{R'}\cd\B)\p_yu)\|_{\F L^2(D_4)}\\
  &+\int_0^t \| \p_xM_2(\p,t-\tau)\langle\na\rangle (\p_xP_\backsim(\vec{R'}\cd\B)\p_yu)
  \|_{\F L^2(D_4)}d\tau\\
  \lesssim&\ \int_0^t \LH^{-1}
  \|\langle\na\rangle^3 \Big(P_\backsim(\vec{R'}\cd\B)\p_yu\Big)\|_{L^2}d\tau\\
  &+\int_0^t \LH^{-\frac{3}{4}}\|\langle \na\rangle^{2.51}
  \Big(\p_xP_\backsim(\vec{R'}\cd\B)\p_yu\Big)\|_{L^1_x(L^2_y)}d\tau\\
  \lesssim&\
  \int_0^t (\LH^{-1}\langle \tau\rangle^{-0.9}+\LH^{-\frac{3}{4}}\langle \tau\rangle^{-1.01})d\tau
  \|V\|_3^2\\
  \lesssim&
\ \langle t\rangle^{-\frac{3}{4}}\|V\|_3^2.
\end{aligned}
$$
``Other good parts" on $D_4$ can be bounded by $C\langle t\rangle^{-\frac{3}{4}}(\|V\|_3^2+\|V\|_3^3)$. Hence, there holds
$$\|\p_x\langle\na\rangle b\|_{\F L^2(D_4)}
\lesssim\ \LHH^{-\frac{3}{4}}
(\|V_0\|_3+\|V_0\|_3^2+\|V\|_3^2+\|
V\|_3^3).$$
\subsection{The estimate of (\ref{4.3})$_3$}
Using $|\xe|\lesssim1$ in $D_1$, one can get
$$\||\na|^{-1}\langle\na\rangle b\|_{\F L^1(D_1)}
\lesssim\ \||\na|^{-1} b\|_{\F L^1(D_1)}.
$$
Thanks to (\ref{rf1}) and (\ref{decay1})$_5$ for $\alpha=1$, we have
$$\||\na|^{-1}L_b\|_{\F L^1(D_1)}
\lesssim\
\||\na|^{-1}\big(M_1(\p,t)u_0,M_2(\p,t)b_0\big)\|_{\F L^1(D_1)}
\lesssim\ \LHH^{-\frac{1}{2}}
\|(u_0,b_0)\|_{L^1}.
$$
Applying (\ref{rf1}), (\ref{decay1})$_5$ for $\alpha=1$, (\ref{bbfen}), we have
$$
\begin{aligned}
 \||\na|^{-1}\mathcal{NL}_b\|_{\F L^1(D_1)}
  \lesssim&\ \int_0^t \||\na|^{-1}(M_1(\p,t-\tau)F^1)\|_{\F L^1(D_1)}d\tau\\
  &+\int_0^t \||\na|^{-1}(M_2(\p,t-\tau)G^1)\|_{\F L^1(D_1)}d\tau\\
  \lesssim&\ \int_0^t \|M_1(\p,t-\tau)(\U\otimes\U,\B B)\|_{\F L^1(D_1)}d\tau\\
  &+\int_0^t \||\na|^{-1}M_1(\p,t-\tau)(b\p_xb)\|_{\F L^1(D_1)}d\tau\\
  &+\int_0^t \|M_2(\p,t-\tau)(\U\otimes\B)\|_{\F L^1(D_1)}d\tau\\
 \end{aligned}
$$
Due to  (\ref{decay1})$_3$ for $k=0$, (\ref{decay1})$_5$ for $\alpha=1/2$, (\ref{ine1}) and (\ref{ine3}),
the first and third  integral can be bounded by
$$
\begin{aligned}
&\ \ C\int_0^t \LH^{-\frac{1}{2}}\|(\U\otimes\U,\U\otimes\B,\B B)\|_{L^2}d\tau\\
\lesssim&\ \int_0^t \LH^{-\frac{1}{2}}
\langle\tau\rangle^{-1.1}d\tau\|V\|_3^2
\lesssim\ \LHH^{-\frac{1}{2}}\|V\|_3^2,
\end{aligned}
$$
while the second integral can be bounded by
$$
\begin{aligned}
&\ \ C\int_0^t\||\na|^{-\frac{1}{2}}M_1(\p,t-\tau)|\p_x|^\frac{1}{2}(bb)
\|_{\F L^1(D_1)}d\tau\\
\lesssim&
\ \int_0^t \LH^{-\frac{3}{4}}
\||\p_x|^\frac{1}{2}(bb)\|_{L^1}d\tau\\
\lesssim&\  \int_0^t \LH^{-\frac{3}{4}}
\langle\tau\rangle^{-\frac{3}{4}}d\tau\|V\|_3^2
\lesssim\ \LHH^{-\frac{1}{2}}\|V\|_3^2.
\end{aligned}
$$
So
\begin{equation}\label{010}
\||\na|^{-1}\langle\na\rangle b\|_{\F L^1(D_1)}
\lesssim\ \LHH^{-\frac{1}{2}}(\|V_0\|_3+\|V\|_3^2).
\end{equation}
Using
$$\||\na|^{-1}\langle\na\rangle b\|_{\F L^1(D_4)}
\lesssim\ \| b\|_{\F L^1(D_{41})}+\||\na|^{-1} b\|_{\F L^1(D_{42})},
$$
is is sufficient to show  the estimates of the terms on the right hand side.
Using (\ref{rf2}), (\ref{End1})$_1$ and (\ref{End2})$_1$ for $r=1$, one can easily get
$$
\|L_b\|_{\F L^1(D_{41})}
\lesssim
\ \langle t\rangle^{-\frac{1}{2}}\|V_0\|_3.
$$
One can infer by (\ref{rf2}), (\ref{bbfen}), (\ref{End1})$_2$  for $r=1$ and (\ref{End1})$_3$ for $r=1$, (\ref{ine4}) that
$$
\begin{aligned}
\|NL_{b1}\|_{\F L^1(D_{41})}
\lesssim&\
\int_0^t \|M_1(\p,t-\tau)F^1\|_{\F L^1(D_{41})}d\tau\\
\lesssim&\ \int_0^t \|M_1(\p,t-\tau)(\U\cd\na\U,\B\cd\na B,\p_y(Bb))\|_{\F L^1(D_{41})}d\tau\\
&+\int_0^t \|\p_xM_1(\p,t-\tau)(bb)\|_{\F L^1(D_{41})}d\tau\\
\lesssim&\ \int_0^t (\LH^{-\frac{1}{2}}\|(\U\otimes\U,B\B)\|_{\F L^1}+\LH^{-1}\|bb\|_{\F L^1})d\tau\|V\|_3^2\\
\lesssim&\ \int_0^t (\LH^{-\frac{1}{2}}\langle \tau\rangle^{-\frac{3}{2}}+\LH^{-1}\langle \tau\rangle^{-1})d\tau\|V\|_3^2\\
\lesssim&\ \langle t\rangle^{-\frac{1}{2}}\|V\|_3^2.
\end{aligned}
$$
Applying (\ref{rf2}),  (\ref{End2})$_1$ for $r=1$ and  (\ref{ine2}), we can get
$$
\begin{aligned}
&\ \ \ \|NL_{b2}\|_{\F L^1(D_{41})}+\|NL_{b3}'\|_{\F L^1(D_{41})}\\
\lesssim&\ \int_0^t \|M_2(\p,t-\tau)\{u\p_xb,\B\cd\na u,v_{<\langle\tau\rangle^{-8}}
\p_yb,v_{>2\langle\tau\rangle^{-0.05}}
\p_yb\}\|_{\F L^1(D_{41})}d\tau\\
\lesssim&\ \int_0^t
\LH^{-\frac{1}{2}}
\|\langle\na\rangle^{1.51}(u\p_xb,\B\cd\na u,v_{<\langle\tau\rangle^{-8}}
\p_yb,v_{>2\langle\tau\rangle^{-0.05}}
\p_yb)\|_{L^1_x(L^2_y)}d\tau\\
\lesssim&\ \int_0^t
\LH^{-\frac{1}{2}}\langle \tau\rangle^{-1.01}d\tau\|V\|_3^2
\lesssim\ \langle t\rangle^{-\frac{1}{2}}\|V\|_3^2.
\end{aligned}
$$
Using (\ref{rf2}) and (\ref{End2})$_1$ again, it is easy to get
$$
\|L_1\|_{\F L^1(D_{41})}
\lesssim\
\|M_2(\p,t)
(P_{\backsim1}(\vec{R'}\cd\B_0)\p_yb_0)\|_{\F L^1(D_4)}
\lesssim\ \langle t\rangle^{-\frac{1}{2}}\|V_0\|_3^2.
$$
Using  (\ref{rf2}),  (\ref{End1})$_3$ and (\ref{ine4}), we can get
$$
\begin{aligned}
\|L_2\|_{\F L^1(D_{41})}
\lesssim&\ \int_0^t
\|\p_xM_1(\p,t-\tau)(P_\backsim(\vec{R'}\cd\B)\p_yb)
\|_{\F L^1(D_{41})}d\tau\\
\lesssim&\ \int_0^t \LH^{-1}
\|P_\backsim(\vec{R'}\cd\B)\|_{\F L^1}\|\p_yb\|_{\F L^1}d\tau\\
\lesssim&\  \int_0^t \LH^{-1}\langle\tau\rangle^{-0.9}d\tau\|V\|_3^2\\
\lesssim&\ \langle t\rangle^{-\frac{1}{2}}\|V\|_3^2.
\end{aligned}
$$
  By (\ref{rf2}),  (\ref{End2})$_2$ for $r=1$, (\ref{End2})$_1$ for $r=1$,   (\ref{ine2}) and (\ref{ine4}), it follows
$$
\begin{aligned}
\|L_3\|_{\F L^1(D_{41})}
\lesssim&\ \int_0^t \|\p_xM_2(\p,t-\tau)
(P_\backsim(\vec{R'}\cd\B)\p_y u)\|_{\F L^1(D_{41})}d\tau\\
&+\int_0^t \|M_2(\p,t-\tau)
(\p_x P_\backsim(\vec{R'}\cd\B)\p_y u)\|_{\F L^1(D_{41})}d\tau\\
\lesssim&\ \int_0^t \LH^{-\frac{1}{2}}
\|\na(P_\backsim(\vec{R'}\cd\B)\p_y u)\|_{\F L^1}d\tau\\
&+\int_0^t \LH^{-\frac{1}{2}}
\|\langle\na\rangle^{1.51}(\p_xP_\backsim(\vec{R'}\cd\B)\p_y u)\|_{L^1_x(L^2_y)}d\tau\\
\lesssim&\ \int_0^t \LH^{-\frac{1}{2}}
\langle\tau\rangle^{-1.01}d\tau\|V\|_3^2
\lesssim \ \langle t\rangle^{-\frac{1}{2}}
\|V\|_3^2.
\end{aligned}
$$
 ``Other good parts" can be bounded by
by $C\langle t\rangle^{-\frac{1}{2}}(\|V\|_3^2+\|V\|_3^3)$.
Finally,
\begin{equation}\label{bb1}
\|b\|_{\F L^1(D_{41})}
\lesssim\ \LHH^{-\frac{1}{2}}(\|V_0\|_3+\|V_0\|_3^2
+\|V\|_3^2+\|V\|_3^3).
\end{equation}
By (\ref{rf2}),  (\ref{End1})$_5$,  (\ref{End2})$_4$, we have
$$
\begin{aligned}
\||\na|^{-1}L_b\|_{\F L^1(D_{42})}
\lesssim&\ \||\na|^{-1} M_1(\p,t)u_0
\|_{\F L^1(D_{42})}+\||\na|^{-1}M_2(\p,t)b_0
\|_{\F L^1(D_{42})}\\
\lesssim&\ \LHH^{-\frac{1}{2}}(\|(u_0,b_0)\|_{L^1}
+\|u_0\|_{L^1_y(L^\frac{4}{3}_x)}).
\end{aligned}
$$
It follows from  (\ref{rf2}), (\ref{bbfen}),  (\ref{End1})$_5$ for $\delta=0.01$, and (\ref{ine1}) that
$$
\begin{aligned}
\||\na|^{-1}NL_{b1}\|_{\F L^1(D_{42})}
\lesssim&\ \int_0^t
\||\na|^{-1}M_1(\p,t-\tau)(\U\cd\na\U,\B\cd\na B,\p_y(B\B))\|_{\F L^1(D_{42})}d\tau\\
&+\int_0^t \||\na|^{-1}M_1(\p,t-\tau)(b\p_xb)\|_{\F L^1(D_{42})}d\tau\\
\lesssim&\ \int_0^t
\LH^{-1}(\||\na|^{0.99}(\U\otimes\U,\B B)\|_{L^1}
+\||\p_x|^{0.99}(bb)\|_{L^1})d\tau\\
\lesssim&\ \int_0^t
\LH^{-1}\langle \tau\rangle^{-0.6}d\tau\|V\|_3^2
\lesssim\ \LHH^{-\frac{1}{2}}\|V\|_3^2.
\end{aligned}
$$
By (\ref{rf2}),  (\ref{End2})$_4$, (\ref{End2})$_1$ for $r=1$ and $\delta=0.01$,    (\ref{ine1}) and (\ref{ine2}), we infer
$$
\begin{aligned}
&\ \ \ \||\na|^{-1}NL_{b2}\|_{\F L^1(D_{42})}
+\||\na|^{-1}NL_{b3}'\|_{\F L^1(D_{42})}\\
\lesssim&\ \int_0^t
\||\na|^{-1}M_2(\p,t-\tau)(u\p_xb,\B\cd\na u,
\p_yv_{<\langle \tau\rangle^{-8}}b,\p_yv_{>\langle\tau\rangle^{-0.05}}b)\|_{\F L^1(D_{42})}d\tau\\
&+\int_0^t
\|M_2(\p,t-\tau)(v_{<\langle \tau\rangle^{-8}}b,v_{>\langle\tau\rangle^{-0.05}}b)\|_{\F L^1(D_{42})}d\tau\\
\lesssim&\ \int_0^t\LH^{-\frac{1}{2}}
\|(u\p_xb,\B\cd\na u,\p_yv_{<\langle \tau\rangle^{-8}}b,\p_yv_{>\langle\tau\rangle^{-0.05}}
b)\|_{L^1}d\tau\\
&+\int_0^t \LH^{-\frac{1}{2}}\|\langle\na\rangle^{1.51}(v_{<\langle \tau\rangle^{-8}}b,v_{>\langle\tau\rangle^{-0.05}}b)
\|_{L^1_x(L^2_y)}d\tau\\
\lesssim&\ \int_0^t\LH^{-\frac{1}{2}}\langle \tau\rangle^{-1.01}
d\tau\|V\|_3^2
\lesssim\ \LHH^{-\frac{1}{2}}\|V\|_3^2.
\end{aligned}
$$
Similarly, we can get by using (\ref{End2})$_4$ that
$$\||\na|^{-1}L_1\|_{\F L^1(D_{42})}
\lesssim\ \LHH^{-\frac{1}{2}}\|P_{\backsim1}(\vec{R'}\cd\B_0)\p_yb_0\|_{L^1}
\lesssim\ \LHH^{-\frac{1}{2}}\|V_0\|_3^2.$$
By (\ref{rf2}),  (\ref{End1})$_6$, (\ref{End1})$_8$, (\ref{ine1}) and (\ref{ine4}), we achieve
$$
\begin{aligned}
\||\na|^{-1}L_2\|_{\F L^1(D_{42})}
\lesssim&\ \int_0^t
\||\na|^{-1}\p_x\p_yM_1(\p,t-\tau)(
P_{\backsim}(\vec{R'}\cd\B)b)\|_{\F L^1(D_{42})}d\tau\\
&+\int_0^t
\||\na|^{-1}\p_xM_1(\p,t-\tau)(\p_y
P_{\backsim}(\vec{R'}\cd\B)b)\|_{\F L^1(D_{42})}d\tau\\
\lesssim&\ \int_0^t \LH^{-1}\|P_{\backsim}(\vec{R'}\cd\B)b\|_{\F L^1}d\tau
+\int_0^t \LH^{-\frac{3}{2}}\|\p_yP_{\backsim}(\vec{R'}\cd\B)b\|_{ L^1}d\tau\\
\lesssim&\ \int_0^t
(\LH^{-1}\langle \tau\rangle^{-0.99}
+\LH^{-\frac{3}{2}}\langle \tau\rangle^{-\frac{1}{2}})
d\tau\|V\|_3^2\\
\lesssim&\ \LHH^{-\frac{1}{2}}\|V\|_3^2.
\end{aligned}
$$
We can obtain from  (\ref{rf2}),  (\ref{End2})$_4$, (\ref{End2})$_5$, (\ref{ine1}) and (\ref{ine4}) that
$$
\begin{aligned}
\||\na|^{-1}L_3\|_{\F L^1(D_{42})}
\lesssim&\ \int_0^t \||\na|^{-1}\p_xM_2(\p,t-\tau)(
P_{\backsim}(\vec{R'}\cd\B)\p_yu)\|_{\F L^1(D_{42})}d\tau\\
&+\int_0^t \||\na|^{-1}M_2(\p,t-\tau)(\p_x
P_{\backsim}(\vec{R'}\cd\B)\p_yu)\|_{\F L^1(D_{42})}d\tau\\
\lesssim&\ \int_0^t
\LH^{-\frac{1}{2}}
(\|P_{\backsim}(\vec{R'}\cd\B)\p_yu\|_{\F L^1}
+\|\p_x
P_{\backsim}(\vec{R'}\cd\B)\p_yu\|_{L^1})d\tau\\
\lesssim&\ \int_0^t
\LH^{-\frac{1}{2}}\langle \tau\rangle^{-1.01}d\tau
\|V\|_3^2
\lesssim\ \LHH^{-\frac{1}{2}}\|V\|_3^2.
\end{aligned}
$$
Hence, it comes out
\begin{equation}\label{bb2}
\||\na|^{-1}b\|_{\F L^1(D_{42})}
\lesssim\ \LHH^{-\frac{1}{2}}
(\|V_0\|_3+\|V_0\|_3^2+\|V\|_3^2+\|V\|_3^3).
\end{equation}
Combining with (\ref{010}), (\ref{bb1}) and (\ref{bb2}) follows (\ref{4.3})$_3$.
\subsection{The estimate of (\ref{4.3})$_4$}
Direct computations yield
\begin{equation}\label{bb3}
\|\mathcal{R}_1\langle\na\rangle b\|_{\F L^1(D_1)}\lesssim\ \||\na|^{-1}\langle\na\rangle b\|_{\F L^1(D_1)}\lesssim\
\langle t\rangle^{-1}(\|V_0\|_3+\|V\|_3^2),
\end{equation}
So it suffices to bounding  $\|\mathcal{R}_1\langle\na\rangle b\|_{\F L^1(D_4)}$.
By (\ref{rf2}), (\ref{End1})$_4$, and (\ref{End2})$_5$ for $l=1$, we have
$$
\begin{aligned}
\|\mathcal{R}_1\langle\na\rangle L_b\|_{\F L^1(D_4)}
\lesssim&\
\|\mathcal{R}_1\langle\na\rangle M_1(\p,t)u_0
\|_{\F L^1(D_4)}+\|\mathcal{R}_1\langle\na\rangle M_2(\p,t)b_0
\|_{\F L^1(D_4)}\\
\lesssim&\ \||\na|\langle\na\rangle M_1(\p,t)u_0
\|_{\F L^1(D_4)}+\|\mathcal{R}_1\langle\na\rangle M_2(\p,t)b_0
\|_{\F L^1(D_4)}\\
\lesssim&\ \langle t\rangle^{-1}
(\|\langle\na\rangle^4(u_0,b_0)\|_{L^1_x(L^2_y)}
+\|\langle\na\rangle^2 u_0\|_{L^1})\\
\lesssim&\ \langle t\rangle^{-1}\|V_0\|_3.
\end{aligned}
$$
It follows by (\ref{rf2}), (\ref{End1})$_8$, (\ref{End1})$_9$ and  (\ref{ine4})  that
$$
\begin{aligned}
\|\mathcal{R}_1\langle \na\rangle NL_{b1}\|_{\F L^1(D_4)}
\lesssim&\ \int_0^t \|\mathcal{R}_1\langle \na\rangle M_1(\p,t-\tau)
(\U\cd\na\U,\B\cd\na B,\p_y(Bb))\|_{\F L^1(D_4)}d\tau\\
&+\int_0^t\|\p_x\mathcal{R}_1 M_1(\p,t-\tau) (bb) \|_{\F L^1(D_4)}d\tau\\
&+\int_0^t\||\na|\mathcal{R}_1 M_1(\p,t-\tau) \p_x(bb) \|_{\F L^1(D_4)}d\tau\\
\lesssim&\ \int_0^t \LH^{-1}
\||\na|^{-1}\langle \na\rangle(\U\cd\na\U,\B\cd\na B,\p_y(Bb))\|_{\F L^1(D_4)}d\tau\\
&+\int_0^t (\LH^{-\frac{3}{2}}\|bb\|_{\F L^1}
+\LH^{-1}\|b\p_xb\|_{\F L^1})d\tau\\
\lesssim&\ \int_0^t (\LH^{-1}
\langle \tau\rangle^{-1.1}+\LH^{-\frac{3}{2}}
\langle \tau\rangle^{-1})d\tau\|V\|_3^2\\
\lesssim&\ \langle t\rangle^{-1}\|V\|_3^2.
\end{aligned}
$$
Using  (\ref{rf2}), (\ref{End2})$_5$ for $l=1$, (\ref{ine2}), one has
$$
\begin{aligned}
&\ \ \ \|\mathcal{R}_1\langle \na\rangle NL_{b2}\|_{\F L^1(D_4)}
+\|\mathcal{R}_1\langle\na\rangle NL_{b3}'\|_{\F L^1(D_4)}\\
\lesssim&\ \int_0^t \|\mathcal{R}_1\langle\na\rangle M_2(\p,t-\tau)(u\p_xb,\B\cd\na u,v_{<\langle\tau\rangle^{-8}}
\p_yb,v_{>2\langle\tau\rangle^{-0.05}}
\p_yb)\|_{\F L^1(D_4)}d\tau\\
\lesssim&\ \int_0^t \LH^{-1}
\|\langle \na\rangle^{2.51}(u\p_xb,\B\cd\na u,v_{<\langle\tau\rangle^{-8}}
\p_yb,v_{>2\langle\tau\rangle^{-0.05}}
\p_yb)\|_{L^1_x(L^2_y)}d\tau\\
\lesssim&\ \int_0^t \LH^{-1}
\langle \tau\rangle^{-1.01}d\tau\|V\|_3^2
\lesssim\ \langle t\rangle^{-1}\|V\|_3^2.
\end{aligned}
$$
Thanks to  (\ref{rf2}),  (\ref{End2})$_5$ for $l=1$, we deduce
$$
\|\mathcal{R}_1\langle\na\rangle L_1\|_{\F L^1(D_4)}
\lesssim\
\|\mathcal{R}_1M_2(\p,t)\langle\na\rangle(P_{\backsim1}(\vec{R'}\cd\B_0)\p_y
b_0)\|_{\F L^1(D_4)}
\lesssim\ \langle t\rangle^{-1}\|V_0\|_3^2.
$$
 Using (\ref{rf2}), (\ref{End1})$_8$,  (\ref{End1})$_9$ and (\ref{ine4}), we achieve
$$
\begin{aligned}
\|\mathcal{R}_1\langle\na\rangle L_2\|_{\F L^1(D_4)}
\lesssim&\ \int_0^t
\|\p_y\mathcal{R}_1\langle\na\rangle M_1(\p,t-\tau) \p_x\big(P_\backsim(\vec{R'}\cd\B)bb\big)
\|_{\F L^1(D_4)}d\tau\\
&+\int_0^t
\|\p_x\mathcal{R}_1\langle\na\rangle M_1(\p,t-\tau) (\p_yP_\backsim(\vec{R'}\cd\B)b)\|_{\F L^1(D_4)}d\tau\\
\lesssim&\ \int_0^t (\LH^{-1} \|\Big[\p_x\langle\na\rangle(P_\backsim(\vec{R'}\cd\B)b),
\p_x(\p_y P_\backsim(\vec{R'}\cd\B)b)\Big]\|_{\F L^1}\\
&+\LH^{-\frac{3}{2}}
\|\p_y P_\backsim(\vec{R'}\cd\B)b\|_{\F L^1})d\tau\\
\lesssim&\ \int_0^t (\LH^{-1} \langle\tau\rangle^{-1.01}
+\LH^{-\frac{3}{2}} \langle\tau\rangle^{-1})d\tau \|V\|_3^2\\
\lesssim&\ \langle t\rangle^{-1}\|V\|_3^2,
\end{aligned}
$$
where we have used $\|\langle\na\rangle f\|_{\F L^1}
\lesssim\ \|f\|_{\F L^1}+\||\na|f\|_{\F L^1}$ for the second inequality.
We infer from  (\ref{rf2}), (\ref{End2})$_6$, (\ref{End2})$_5$ for $l=1$,
and (\ref{ine4}), (\ref{ine2}) that
$$
\begin{aligned}
\|\mathcal{R}_1\langle\na\rangle L_3\|_{\F L^1(D_4)}
\lesssim&\ \int_0^t \|\p_x\mathcal{R}_1\langle\na\rangle M_2(\p,t-\tau)(P_\backsim(\vec{R'}\cd\B)\p_y u)\|_{\F L^1(D_4)}d\tau\\
&+\int_0^t \|\mathcal{R}_1\langle\na\rangle M_2(\p,t-\tau) (\p_xP_\backsim(\vec{R'}\cd\B)\p_y u)\|_{\F L^1(D_4)}d\tau\\
\lesssim&\
\int_0^t \LH^{-1}\|\langle \na\rangle^2(P_\backsim(\vec{R'}\cd\B)\p_y u)\|_{\F L^1}d\tau\\
&+\int_0^t \LH^{-1}\|\langle \na\rangle^{2.51}(\p_xP_\backsim(\vec{R'}\cd\B)\p_y u)\|_{L^1_x(L^2_y)}d\tau\\
\lesssim&\ \int_0^t \langle t-\tau\rangle^{-1}
\langle\tau\rangle^{-1.01}d\tau\|V\|_3^2
\lesssim\ \langle t\rangle^{-1}\|V\|_3^2.
\end{aligned}
$$
We can bound ``other good parts" by
$C\langle t\rangle^{-1}(\|V\|_3^2+\|V\|_3^3)$. Thus
\begin{equation*}
\|\mathcal{R}_1\langle\na\rangle b\|_{\F L^1(D_4)}\lesssim\
\langle t\rangle^{-1}(\|V_0\|_3+\|V_0\|_3^2+\|V\|_3^2+\|V\|_3^3),
\end{equation*}
which,
together with (\ref{bb3})
leads to (\ref{4.3})$_4$.
\section{The  estimate on $B$ and proof of (\ref{endend})}
\label{s4}
\vskip.2in
In this section, we will prove
\begin{equation}\label{4.4}
\left\{
\begin{aligned}
\|B(t)\|_{L^2}\lesssim\ &\langle t\rangle^{-\frac{1}{2}}
\big(\|V_0\|_3+\|V\|_3^2\big);\\
\|\p_xB(t)\|_{L^2}\lesssim\ &\langle t\rangle^{-1}
\big(\|V_0\|_3+\|V\|_3^2\big);\\
\|B(t)\|_{FL^1}\lesssim\ &\langle t\rangle^{-1}
\big(\|V_0\|_3+\|V_0\|_3^2+\|V\|_3^2+\|V\|_3^3\big).
\end{aligned}
\right.
\end{equation}
\subsection{The expression of $B$}
Thanks to (\ref{B}), we can get
$$
B=\underbrace{M_{1}(\p,t)v_0+M_{2}(\p,t)B_0}_{L_B}
+\underbrace{\int_0^t M_{1}(\p,t-\tau)F^2d\tau}_{L_{B1}}
+\underbrace{\int_0^t M_{2}(\p,t-\tau)G^2d\tau}_{L_{B2}}.
$$
\subsection{The estimate  of (\ref{4.4})$_1$}
Like the estimate of $\|\widehat{v}\|_{L^2(D_1)}$, we can also get
\begin{equation}\label{BB1}
\|\widehat{B}\|_{L^2(D_1)}\lesssim\ \langle t\rangle^{-\frac{1}{2}}(\|V_0\|_3+\|V\|_3^2).
\end{equation}
Using (\ref{rf2}), (\ref{End1})$_1$, (\ref{D21})$_1$ for $k=1$,
and
$B_0=-\mathcal{R}_{11}B_0+\mathcal{R}_{12}b_0$, we have
$$
\begin{aligned}
\|L_B\|_{\F L^2(D_4)}
\lesssim&\ \|M_1(\p,t)v_0\|_{\F L^2(D_4)}
+\|\mathcal{R}_1M_2(\p,t)(B_0,b_0)\|_{\F L^2(D_4)}\\
\lesssim&\ \langle t\rangle^{-\frac{1}{2}}\|v_0\|_{L^1\cap L^2}
+\|\G_{1,1}e^{-\G_{2,2}t}\F\{\B_0\}\|_{L^2(D_4)}\\
\lesssim&\ \langle t\rangle^{-\frac{1}{2}}\|V_0\|_3.
\end{aligned}
$$
Using (\ref{f2}), (\ref{D21})$_1$ for $k=2$, and (\ref{ine3}), one can obtain
$$
\begin{aligned}
\|NL_{B1}\|_{\F L^2(D_4)}
\lesssim&\ \int_0^t\||\na|\mathcal{R}_1M_1(\p,t-\tau)(\B\otimes\B,\U\otimes\U
)\|_{\F L^2(D_4)}d\tau\\
\lesssim&\ \int_0^t\|\G_{2,2}e^{-\G_{2,2}(t-\tau)}
\F\{\B\otimes\B,\U\otimes\U\}\|_{L^2(D_4)}d\tau\\
\lesssim&\ \int_0^t \LH^{-1}
(\| \B\otimes\B\|_{L^2}+\| \U\otimes\U\|_{L^2})d\tau\\
\lesssim&\ \int_0^t \LH^{-1}
\langle\tau\rangle^{-\frac{3}{4}}d\tau\|V\|_3^2\\
\lesssim&\ \langle t\rangle^{-\frac{1}{2}}\|V\|_3^2.
\end{aligned}
$$
We can deduce from (\ref{g2}), (\ref{End2})$_2$ for $r=2$,  (\ref{g2}) and (\ref{ine3}) that
$$
\begin{aligned}
\|NL_{B2}\|_{\F L^2(D_4)}
\lesssim&\ \int_0^t\|\p_xM_2(\p,t-\tau)(uB-bv)\|_{\F L^2(D_4)}d\tau\\
\lesssim&\ \int_0^t \LH^{-\frac{1}{2}}
(\| \na(uB)\|_{L^2}+\| \na(bv)\|_{L^2})d\tau\\
\lesssim&\ \int_0^t \LH^{-\frac{1}{2}}
\langle\tau\rangle^{-1.01}d\tau\|V\|_3^2\\
\lesssim&\ \langle t\rangle^{-\frac{1}{2}}\|V\|_3^2.
\end{aligned}
$$
As a result, there holds
\begin{equation}\label{BB2}
\|\widehat{B}\|_{L^2(D_4)}
\lesssim\ \langle t\rangle^{-\frac{1}{2}}
(\|V_0\|_3+\|V\|_3^2).
\end{equation}
It follows (\ref{4.4})$_1$ by using (\ref{BB1}) and (\ref{BB2}).
\subsection{The estimate  of (\ref{4.4})$_2$}
As the estimate of $\|\p_x v\|_{L^2}$, it is easy to obtain
\begin{equation}\label{BB3}
\|\p_x B\|_{\F L^2(D_1)}\lesssim
\ \langle t\rangle^{-1}(\|V_0\|_3+\|V\|_3^2).
\end{equation}
Similar to the estimate of   $\|L_B\|_{\F L^2(D_4)}$, by (\ref{rf2}),
(\ref{End1})$_3$ for $r=2$, and (\ref{D21})$_1$ for $k=2$,  we infer
$$
\begin{aligned}
\|\p_xL_B\|_{\F L^2(D_4)}
\lesssim&\
\|\p_xM_1(\p,t)v_0
\|_{\F L^2(D_4)}
+\|\p_xM_2(\p,t)B_0
\|_{\F L^2(D_4)}\\
\lesssim&\ \langle t\rangle^{-1}\|v_0
\|_{ L^2}
+\|\mathcal{R}_1^2|\na|M_2(\p,t)\B_0
\|_{\F L^2(D_4)}\\
\lesssim&\ \langle t\rangle^{-1}\|v_0
\|_{ L^2}
+\|\G_{2,2}e^{-\G_{2,2}t}\F\{|\na| \B_0\}
\|_{L^2(D_4)}\\
\lesssim&\ \langle t\rangle^{-1}\|V_0\|_3.
\end{aligned}
$$
It follows by (\ref{rf2}),
(\ref{End1})$_3$ for $r=2$, and (\ref{ine3}) that
$$
\begin{aligned}
\|\p_xNL_{B1}\|_{\F L^2(D_4)}
\lesssim&\ \int_0^t\|\p_xM_1(\p,t-\tau)(\B\cd\na\B,\U\cd\na\U
)\|_{\F L^2(D_4)}d\tau\\
\lesssim&\ \int_0^t \LH^{-1}
(\| \B\cd\na\B\|_{L^2}+\| \U\cd\na\U\|_{L^2})d\tau\\
\lesssim&\ \int_0^t \LH^{-1}
\langle\tau\rangle^{-1.1}d\tau\|V\|_3^2\\
\lesssim&\ \langle t\rangle^{-1}\|V\|_3^2.
\end{aligned}
$$
Using (\ref{g2}), (\ref{rf2}), (\ref{End2})$_3$, and (\ref{ine3}), one has
$$
\begin{aligned}
\|\p_xNL_{B2}\|_{\F L^2(D_4)}
\lesssim&\ \int_0^t\|\p_x^2M_2(\p,t-\tau)(uB-bv)\|_{\F L^2(D_4)}d\tau\\
\lesssim&\ \int_0^t \LH^{-1}
(\| \na^2(uB)\|_{L^2}+\| \na^2(bv)\|_{L^2})d\tau\\
\lesssim&\ \int_0^t \LH^{-1}
\langle\tau\rangle^{-1.01}d\tau\|V\|_3^2\\
\lesssim&\ \langle t\rangle^{-1}\|V\|_3^2.
\end{aligned}
$$
Hence, we have
\begin{equation}\label{BB4}
\|\p_x B\|_{\F L^2(D_4)}
\lesssim\ \langle t\rangle^{-1}(\|V_0\|_3+\|V\|_3^2).
\end{equation}
Combining with (\ref{BB3}) and (\ref{BB4}) implies (\ref{4.4})$_2$.
\subsection{The estimate  of (\ref{4.4})$_3$}
\vskip .2in
Like the previous way dealing with $\|\widehat{B}\|_{L^1(D_1)}$, one can get
\begin{equation}\label{BB7}
\|B\|_{\F L^1(D_1)}\lesssim\ \langle t\rangle^{-1}(\|V_0\|_3+\|V\|_3^2).
\end{equation}
For the estimate on $D_4$, by
$B=-\mathcal{R}_{11}B+\mathcal{R}_{12}b$,
and
\begin{equation}\label{011}
\|\mathcal{R}_{12}b\|_{\F L^1(D_4)}
\lesssim\ \|\mathcal{R}_{1}\langle\na\rangle b\|_{\F L^1(D_4)}
\lesssim\ \langle t\rangle^{-1}(\|V_0\|_3+\|V_0\|_3^2+\|V\|_3^2+\|V\|_3^3),
\end{equation}
it suffices to bound $\|\mathcal{R}_{11}B\|_{\F L^1(D_4)}$.
By (\ref{rf2}), (\ref{End1})$_8$,  (\ref{End2})$_5$ for $l=2$ and (\ref{ine4}),  we deduce
$$
\begin{aligned}
\|\mathcal{R}_{11}L_B\|_{\F L^1(D_4)}
\lesssim&\ \|\mathcal{R}_{11}M_1(\p,t)v_0\|_{\F L^1(D_4)}+\|\mathcal{R}_{11}M_2(\p,t)b_0\|_{\F L^1(D_4)}\\
\lesssim&\ \langle t\rangle^{-1}\|(v_0,b_0)\|_{\F L^1},
\end{aligned}
$$
$$
\begin{aligned}
  \|\mathcal{R}_{11}NL_{B1}\|_{\F L^1(D_4)}
  \lesssim&\ \int_0^t
  \|\mathcal{R}_{11}M_1(\p,t-\tau)(\U\cd\na \U,\B\cd\na\B)\|_{\F L^1(D_4)}d\tau\\
  \lesssim&\ \int_0^t \LH^{-1}\|(\U\cd\na \U,\B\cd\na\B)\|_{\F L^1}d\tau\\
  \lesssim&\ \int_0^t
  \LH^{-1}\langle \tau\rangle^{-1.1}d\tau\|V\|_3^2
  \lesssim\ \langle t\rangle^{-1}\|V\|_3^2,
\end{aligned}
$$
and
$$
\begin{aligned}
  \|\mathcal{R}_{11}NL_{B2}\|_{\F L^1(D_4)}
  \lesssim&\ \int_0^t \|\mathcal{R}_{11}M_2(\p,t-\tau)(\U\cd\na B,\B\cd\na v)\|_{\F L^1(D_4)}d\tau\\
  \lesssim&\ \int_0^t \LH^{-1}
  \|(\U\cd\na B,\B\cd\na v)\|_{\F L^1}d\tau\\
  \lesssim&\ \int_0^t \LH^{-1}
  \langle \tau\rangle^{-1.1}d\tau\|V\|_3^2
  \lesssim\ \langle t\rangle^{-1}\|V\|_3^2.
\end{aligned}
$$
So we achieve
\begin{equation}\label{BB8}
\|\mathcal{R}_{11}B\|_{\F L^1(D_4)}
\lesssim\ \langle t\rangle^{-1}(\|V_0\|_3+\|V\|_3^2).
\end{equation}
Finally, it follows (\ref{4.4})$_3$ by using (\ref{BB7}), (\ref{BB8}) and (\ref{011}).
\subsection{Proof of (\ref{endend})}
(\ref{endend}) is a direct consequence by adding (\ref{H8}), (\ref{4.1}), (\ref{4.2}), (\ref{4.3}), (\ref{4.4}) and $\|V\|_3^2\lesssim\ \|V\|_3^\frac{3}{2}+\|V\|_3^3$.

\vskip .4in
\section*{Acknowledgements}
We would like to thank Professor Zhifei Zhang for his  suggestions and helpful comments. This work was supported by the NSF of the Jiangsu Higher Education Institutions of China
(18KJB110018), the NSF of Jiangsu Province BK20180721.

\appendix
\section{}
\label{app}
In this section, we give the proof of some lemmas.
\begin{proof}[Proof of Lemma \ref{M1}]
Thanks to (\ref{rf2}), (\ref{D21})$_7$ for $k=0$ and (\ref{D21})$_1$ for $k=1$ and $r=2$, we  get
$$
\begin{aligned}
\|M_1(\p,t)f\|_{\F L^2(D_4)}
\lesssim&\ \min\{\|\G_{1,2}e^{-\G_{2,2}t}\widehat{f}\|_{L^2(D_4)},
\|\G_{1,1}e^{-\G_{2,2}t}\widehat{|\na|^{-1}f}\|_{L^2(D_4)}\}\\
\lesssim&\ \LHH^{-\frac{1}{2}}\min\{\|f\|_{L^1\cap L^2},\||\na|^{-1}f\|_{L^2}\},
\end{aligned}$$
which yields  (\ref{End1})$_1$.
It follows (\ref{End1})$_2$ by
using (\ref{rf2}),
$$\||\na|M_1(\p,t)f\|_{\F L^r(D_4)}
\lesssim\ \|\G_{1,1}e^{-\G_{2,2}t}\widehat{f}\|_{L^r(D_4)}
$$
and (\ref{D21})$_1$ for $k=1$.
(\ref{End1})$_3$ can be obtained by
using (\ref{rf2}),
$$\|\p_xM_1(\p,t)f\|_{\F L^r(D_4)}
\lesssim\ \|\G_{2,2}e^{-\G_{2,2}t}\widehat{f}\|_{L^r(D_4)}
$$
and (\ref{D21})$_1$ for $k=2$.
Combining with (\ref{rf2}),
$$\|M_1(\p,t)f\|_{\F L^r(D_4)}
\lesssim\ \|\G_{1,2}e^{-\G_{2,2}t}\widehat{f}\|_{L^r(D_4)},
$$
(\ref{D21})$_2$ for $k=1$ and $\delta=0.01$
  leads  (\ref{End1})$_4$.
Combining with (\ref{rf2}),
$$\||\na|^{-1}M_1(\p,t)f\|_{\F L^1(D_{42})}
\lesssim\ \|\G_{1,3}e^{-\G_{2,2}t}\widehat{f}\|_{L^1(D_{42})}
$$
and (\ref{D21})$_6$ for $k=0$, $q=p=4/3$,
we can get (\ref{End1})$_5$.
It follows (\ref{End1})$_6$ by
using (\ref{rf2}),
$$\|\mathcal{R}_1M_1(\p,t)f\|_{\F L^1(D_{42})}
\lesssim\ \|\G_{2,3}e^{-\G_{2,2}t}\widehat{f}\|_{L^1(D_{42})},
$$
and
(\ref{D21})$_4$ for $k=1$.
It follows (\ref{End1})$_7$ by
using (\ref{rf2}),
$$\|\mathcal{R}_1M_1(\p,t)f\|_{\F L^1(D_{4})}
\lesssim\ \|\G_{2,3}e^{-\G_{2,2}t}\widehat{f}\|_{L^1(D_4)},
$$
and
(\ref{D21})$_3$ for $k=1$. Applying
$$\||\na|\mathcal{R}_1M_1(\p,t)f\|_{\F L^1(D_{4})}
+\|\mathcal{R}_{11}M_1(\p,t)f\|_{\F L^1(D_{4})}
\lesssim\ \|\G_{2,2}e^{-\G_{2,2}t}\widehat{f}\|_{L^1(D_4)},
$$
and (\ref{D21})$_1$ for $k=2$ and $r=1$ can lead (\ref{End1})$_8$. Similarly, using
$$\|\p_x\mathcal{R}_1M_1(\p,t)f\|_{\F L^1(D_{4})}
+\||\na|\mathcal{R}_{11}M_1(\p,t)f\|_{\F L^1(D_{4})}
\lesssim\ \|\G_{3,3}e^{-\G_{2,2}t}\widehat{f}\|_{L^1(D_4)},
$$
and (\ref{D21})$_1$ for $k=3$ and $r=1$ can lead (\ref{End1})$_9$.
\end{proof}

\begin{proof}[Proof of Lemma \ref{M2}]
Using (\ref{rf2}) and (\ref{D21})$_1$ for $k=0$, one can easily get
(\ref{End2})$_1$. By (\ref{rf2}),
$$\|\p_xM_2(\p,t)f\|_{\F L^r(D_{4})}
\lesssim\ \||\na|\G_{1,1}e^{-\G_{2,2}t}\widehat{f}\|_{L^r(D_4)},
$$
and (\ref{D21})$_1$ for $k=1$, we can obtain (\ref{End2})$_2$.
Similarly, one can also get (\ref{End2})$_3$. Using a similar way leading (\ref{I22}) for $k=0$, we can get (\ref{End2})$_4$. Thanks to
$$\|\mathcal{R}_1^lM_2(\p,t)f\|_{\F L^1(D_{4})}
\lesssim\ \LHH^{-\frac{l}{2}}\|M_2(\p,t)f\|_{\F L^1(D_{4})}
\lesssim\ \LHH^{-\frac{l}{2}} \|\G_{0,0}e^{-\G_{2,2}t}\widehat{f}\|_{L^1(D_4)},
$$
it follows (\ref{End2})$_5$ by
using (\ref{D21})$_1$ for $k=0$ and $r=1$. Due to
$$\|\p_x\mathcal{R}_1M_2(\p,t)f\|_{\F L^1(D_{4})}
\lesssim\ \||\na|\mathcal{R}_1^2M_2(\p,t)f\|_{\F L^1(D_{4})},$$
it follows (\ref{End2})$_6$ by using (\ref{End2})$_5$.
\end{proof}

\begin{proof}[Proof of Lemma \ref{M3}]
Using (\ref{rf2}),  (\ref{decay1})$_2$ for $k=1$,  (\ref{D21})$_1$ for $k=1$ and $r=2$, we have
$$
\|M_3(\p,t)f\|_{\F L^2(D_{4})}
\lesssim\ \|e^{\frac{1}{2}t\De}f\|_{\F L^2(D_{4})}+
\||\xe|^{-1}\G_{1,1}e^{-\G_{2,2}t}\widehat{f}\|_{L^2(D_4)}
\lesssim\ \LHH^{-\frac{1}{2}}\||\na|^{-1}f\|_{H^1}.
$$
Using (\ref{rf2}),  (\ref{decay1})$_2$ for $k=0$ and (\ref{D21})$_7$ for $k=0$, we have
$$
\|M_3(\p,t)f\|_{\F L^2(D_{4})}
\lesssim\ \|e^{\frac{1}{2}t\De}f\|_{\F L^2(D_{4})}+
\|\G_{1,2}e^{-\G_{2,2}t}\widehat{f}\|_{L^2(D_4)}
\lesssim\ \LHH^{-\frac{1}{2}}\|f\|_{L^1\cap L^2}.
$$
Thus, we complete the proof of (\ref{End3})$_1$. Using
(\ref{rf2}), one can get the first estimate of (\ref{End3})$_2$ by applying
$$
\|M_3(\p,t)f\|_{\F L^1(D_{4})}
\lesssim\ \|e^{\frac{1}{2}t\De}f\|_{\F L^1(D_{4})}+
\|\G_{2,4}e^{-\G_{2,2}t}\widehat{f}\|_{L^1(D_4)},
$$
(\ref{decay1})$_4$ for $k=0$, and (\ref{D21})$_5$ for $k=0$.
Using
(\ref{rf2}), one can get the second estimate of (\ref{End3})$_2$ by applying
$$
\|M_3(\p,t)f\|_{\F L^1(D_{4})}
\lesssim\ \||\na|e^{\frac{1}{2}t\De}|\na|^{-1}f\|_{\F L^1(D_{4})}+
\|\G_{2,3}e^{-\G_{2,2}t}|\xe|^{-1}\widehat{f}\|_{L^1(D_4)},
$$
(\ref{decay1})$_4$ for $k=1$, and (\ref{D21})$_3$ for $k=1$.
Using
(\ref{rf2}), one can get the first estimate of (\ref{End3})$_3$ by applying
$$
\|\p_yM_3(\p,t)f\|_{\F L^2(D_{4})}
\lesssim\ \||\na|e^{\frac{1}{2}t\De}f\|_{\F L^2(D_{4})}+
\|\G_{1,1}e^{-\G_{2,2}t}\widehat{f}\|_{L^2(D_4)},
$$
(\ref{decay1})$_2$ for $k=1$, and (\ref{D21})$_1$ for $k=1$.
Using
(\ref{rf2}), one can get the second estimate of (\ref{End3})$_3$ by applying
$$
\|\p_yM_3(\p,t)f\|_{\F L^2(D_{4})}
\lesssim\ \||\na|^2e^{\frac{1}{2}t\De}|\na|^{-1}f\|_{\F L^2(D_{4})}+
\|\G_{2,2}e^{-\G_{2,2}t}|\xe|^{-1}\widehat{f}\|_{L^2(D_4)},
$$
 (\ref{decay1})$_2$ for $k=2$, and (\ref{D21})$_1$ for $k=2$.
Using
(\ref{rf2}), one can get the third estimate of (\ref{End3})$_3$ by applying
$$
\|\p_yM_3(\p,t)f\|_{\F L^2(D_{4})}
\lesssim\ \||\na|e^{\frac{1}{2}t\De}f\|_{\F L^2(D_{4})}+
\|\G_{2,3}e^{-\G_{2,2}t}\widehat{f}\|_{L^2(D_4)},
$$
(\ref{decay1})$_1$ for $k=1$, and (\ref{D21})$_7$ for $k=1$.
So we conclude the proof of (\ref{End3})$_3$. For (\ref{End3})$_4$, we only show the case $r=2$, and other cases can be bounded similarly.
Using $|\xi|\lesssim |\xe|^2$,
$$
\|\p_xM_3(\p,t)f\|_{\F L^2(D_{4})}
\lesssim\ \|\p_xe^{\frac{1}{2}t\De}f\|_{\F L^2(D_{4})}+
\|\G_{2,2}e^{-\G_{2,2}t}\widehat{f}\|_{L^2(D_4)},
$$
one can get (\ref{End3})$_4$ for $r=2$ by (\ref{decay1})$_2$ for $k=2$ and (\ref{D21})$_1$ for $k=2$.
Using $|\xi|\lesssim |\xe|^2$,
$$
\|\p_xM_3(\p,t)f\|_{\F L^1(D_{4})}
\lesssim\ \|\p_xe^{\frac{1}{2}t\De}f\|_{\F L^1(D_{4})}+
\|\G_{2,4}e^{-\G_{2,2}t}\widehat{\p_xf}\|_{L^1(D_4)},
$$
one can get the first bound of (\ref{End3})$_5$  by (\ref{decay1})$_4$ for $k=2$ and (\ref{D21})$_5$ for $k=0$.
Using $|\xi|\lesssim |\xe|^2$,
$$
\|\p_xM_3(\p,t)f\|_{\F L^1(D_{4})}
\lesssim\ \|\p_xe^{\frac{1}{2}t\De}f\|_{\F L^1(D_{4})}+
\|\G_{3,4}e^{-\G_{2,2}t}\widehat{f}\|_{L^1(D_4)},
$$
one can get the second bound of (\ref{End3})$_5$  by (\ref{decay1})$_4$ for $k=2$ and (\ref{D21})$_3$ for $k=2$.
Using
$$\|e^{ct\De}f\|_{\F L^1}
\lesssim\ \LHH^{-\frac{3}{4}}(\|f\|_{L^1_y(L^2_x)}
+\|\widehat{f}\|_{L^1}),$$
which can be proved by using the similar arguments yielding (\ref{q1}), and
$$\|\G_{2,4}e^{-\G_{2,2}t}\widehat{f}\|_{L^1(D_4)}
\le\ \|\G_{2,2}e^{-\G_{2,2}t}\widehat{f}\|_{L^1(D_{41})}
+\|\G_{2,4}e^{-\G_{2,2}t}\widehat{f}\|_{L^1(D_{42})}
\le\ \LHH^{-\frac{3}{4}}(\|\widehat{f}\|_{L^1}+\|f\|_{L^1_y(L^2_x)}),
$$
which can be obtained by using (\ref{D21})$_1$ for $k=2$ and (\ref{D21})$_6$ for $k=1$, $p=\frac{4}{3}$ and $q=2$, we have
$$
\|\p_xM_3(\p,t)f\|_{\F L^1(D_{4})}
\lesssim\ \|e^{\frac{1}{2}t\De}\p_xf\|_{\F L^1(D_{4})}+
\|\G_{2,4}e^{-\G_{2,2}t}\widehat{\p_xf}\|_{L^1(D_4)}
\lesssim\ \LHH^{-\frac{3}{4}}(\|\widehat{\p_xf}\|_{L^1}+\|\p_xf\|_{L^1_y(L^2_x)}),
$$
which completes the proof of the third bound of (\ref{End3})$_5$.
Using (\ref{decay1})$_2$ for $k=2$, (\ref{D21})$_1$ for $k=2$
and
$$
\|\De M_3(\p,t)f\|_{\F L^2(D_{4})}
\lesssim\ \|\De e^{\frac{1}{2}t\De}f\|_{\F L^2(D_{4})}+
\|\G_{2,2}e^{-\G_{2,2}t}\widehat{f}\|_{L^2(D_4)},
$$
we can get (\ref{End3})$_6$.
Using $|\xi|\lesssim|\xe|^2$, (\ref{decay1})$_4$ for $k=5$, (\ref{D21})$_1$ for $k=5$
and
$$
\|\p_x^2\mathcal{R}_1 M_3(\p,t)f\|_{\F L^1(D_{4})}
\lesssim\ \|\p_x^2\mathcal{R}_1 e^{\frac{1}{2}t\De}f\|_{\F L^1(D_{4})}+
\|\G_{5,5}e^{-\G_{2,2}t}\widehat{f}\|_{L^1(D_4)},
$$
we can get (\ref{End3})$_7$. Using $|\xi|\lesssim|\xe|^2$, (\ref{decay1})$_4$ for $k=3$, (\ref{D21})$_1$ for $k=3$
and
$$
\|\p_x|\na| M_3(\p,t)f\|_{\F L^1(D_{4})}
\lesssim\ \|\p_x |\na| e^{\frac{1}{2}t\De}f\|_{\F L^1(D_{4})}+
\|\G_{3,3}e^{-\G_{2,2}t}\widehat{f}\|_{L^1(D_4)},
$$
we can get (\ref{End3})$_8$ for $r=1$. Other cases $1<r\le2$ can be bounded similarly.
Using $|\xi|\lesssim|\xe|^2$, (\ref{decay1})$_4$ for $k=4$, (\ref{D21})$_1$ for $k=4$
and
$$
\|\p_x^2 M_3(\p,t)f\|_{\F L^1(D_{4})}
\lesssim\ \|\p_x^2  e^{\frac{1}{2}t\De}f\|_{\F L^1(D_{4})}+
\|\G_{4,4}e^{-\G_{2,2}t}\widehat{f}\|_{L^1(D_4)},
$$
we can get (\ref{End3})$_9$. Using $|\xi|\lesssim|\xe|^2$, (\ref{decay1})$_4$ for $k=2$, (\ref{D21})$_1$ for $k=2$
and
$$
\|\mathcal{R}_1 M_3(\p,t)f\|_{\F L^1(D_{4})}
\lesssim\ \|\p_x  e^{\frac{1}{2}t\De}|\na|^{-1}f\|_{\F L^1(D_{4})}+
\|\G_{2,2}e^{-\G_{2,2}t}|\xe|^{-1}\widehat{f}\|_{L^1(D_4)},
$$
we can get (\ref{End3})$_{10}$. Using $|\xi|\lesssim|\xe|^2$, (\ref{decay1})$_4$ for $k=3$, (\ref{D21})$_1$ for $k=3$
and
$$
\|\p_x\mathcal{R}_1 M_3(\p,t)f\|_{\F L^1(D_{4})}
\lesssim\ \|\p_x\mathcal{R}_1  e^{\frac{1}{2}t\De}f\|_{\F L^1(D_{4})}+
\|\G_{3,3}e^{-\G_{2,2}t}\widehat{f}\|_{L^1(D_4)},
$$
we can get (\ref{End3})$_{11}$.
\end{proof}
\begin{proof}[Proof of Lemma \ref{ne1}]
\underline{(\ref{ine1})$_1$} By H\"{o}lder's inequality, product estimate in one dimension, interpolation inequality and
\begin{equation}\label{p222}
\|\na P_\backsim(\vec{R'}\cd\B)\|_{L^2}
\lesssim\ \|\B\|_{L^2}\lesssim\ \LHH^{-\frac{1}{4}}\|V\|_3,
\end{equation}
  we have
$$\|\na P_\backsim(\vec{R'}\cd\B) b\|_{L^1}+\|\B b\|_{L^1}\le\ \|\B\|_{L^2}\|b\|_{L^2}
\lesssim\ \LHH^{-\frac{1}{2}}\|V\|_3^2,$$
$$\||\p_x|^\beta(bb)\|_{L^1_x}
\lesssim\ \|b\|_{L^2_x}\||\p_x|^\beta b\|_{L^2_x}
\lesssim\ \|b\|_{L^2_x}^{2-\beta}
\|\p_xb\|_{L^2_x}^\beta,$$
which yields
$$\||\p_x|^\beta(bb)\|_{L^1}
\lesssim\ \|b\|_{L^2}^{2-\beta}
\|\p_xb\|_{L^2}^\beta
\lesssim\ \LHH^{-\frac{1+\beta}{2}}\|V\|_3^2.$$
\underline{(\ref{ine1})$_2$} Using H\"{o}lder's inequality and $\p_y v=-\p_xu$, we have
$$\|\p_yv_{<\LHH^{-8}}b\|_{L^1}
+\|\p_yv_{>2\LHH^{-0.05}}b\|_{L^1}
\lesssim\ \|\p_xu\|_{L^2}\|b\|_{L^2}
\lesssim\ \LHH^{-\frac{5}{4}}\|V\|_3^2.$$
Thanks to $\|\p_x P_\backsim (\vec{R'}\cd\B)\|_{L^2}
\lesssim\ \|B\|_{L^2}$, we can get
$$\|\p_xP_\backsim(\vec{R'}\cd\B)\p_yu\|_{L^1}
\le\ \|\p_xP_\backsim(\vec{R'}\cd\B)\|_{L^2}
\|\p_yu\|_{L^2}
\lesssim\ \|B\|_{L^2}\|\p_yu\|_{L^2}
\lesssim\ \LHH^{-\frac{5}{4}}\|V\|_3^2.$$
\underline{(\ref{ine1})$_3$} Using product estimate and interpolation inequality, we have
$$\||\na|^\beta(\U\otimes\U)\|_{L^1}
\lesssim\ \|\U\|_{L^2}\||\na|^\beta \U\|_{L^2}
\lesssim\ \|\U\|_{L^2}^{2-\beta}\|\na\U\|_{L^2}^\beta
\lesssim\ \LHH^{-1-\frac{\beta}{4}}\|V\|_3^2$$
and
$$
\begin{aligned}
\||\na|^{0.99}(\B B)\|_{L^1}
\lesssim&\  \||\na|^{0.99}\B\|_{L^2} \|B\|_{L^2}
+\||\na|^{0.99} B\|_{L^2} \|\B\|_{L^2}\\
\lesssim&\ \LHH^{-0.21-0.5}\|V\|_3^2
+\LHH^{-0.25-0.6}\|V\|_3^2\\
\lesssim&\ \LHH^{-0.7}\|V\|_3^2,
\end{aligned}
$$
where we have used
$$\||\na|^{0.99}\B\|_{L^2}
\le\ \||\na|^{0.99}\B_{\le\LHH^{0.03}}\|_{L^2}
+\||\na|^{0.99}\B_{>\LHH^{0.03}}\|_{L^2}\lesssim
\LHH^{-0.21} \|V\|_3$$
and
$$\||\na|^{0.99}B\|_{L^2}
\lesssim\ \|B\|_{L^2}^{0.01}\|\na B\|_{L^2}^{0.99}
\lesssim\ \LHH^{-0.6}\|V\|_3.$$
This concludes the proof of Lemma \ref{ne1}.
\end{proof}
\begin{proof}[Proof of Lemma \ref{ne2}]
\underline{(\ref{ine2})$_1$}  Using interpolation inequality
$$\|f\|_{L^\frac{2p}{2-p}_x}\lesssim\ \|f\|_{L^2_x}^\frac{1}{p}\|\p_xf\|_{L^2_x}^{1-\frac{1}{p}},$$
 we obtain
  $$
  \begin{aligned}
\|b\p_xb\|_{L^1_y(L^p_x)}
\lesssim&\ \big\| \|b\|_{L^\frac{2p}{2-p}_x}\|\p_xb\|_{L^2_x} \big\|_{L^1_y}
\lesssim\ \big\|  \|b\|_{L^2_x}^\frac{1}{p}\|\p_xb\|_{L^2_x}^{2-\frac{1}{p}}  \big\|_{L^1_y}\\
\lesssim&\
\|b\|_{L^2}^\frac{1}{p}
\|\p_xb\|_{L^2}^{2-\frac{1}{p}}
\lesssim\
\LHH^{-\frac{3}{2}+\frac{1}{2p}}\|V\|_3^2.
  \end{aligned}
  $$
\underline{(\ref{ine2})$_2$} We only give the estimate of $\|\langle\na\rangle^3(\B\cd\na u)\|_{L^1_x(L^2_y)}$, since other terms can be bounded similarly.
Using
\begin{equation}\label{f3}
\|\langle\na\rangle^k f\|_{L^1_x(L^2_y)}
\lesssim\  \| f\|_{L^1_x(L^2_y)}
+\|\na^k f\|_{L^1_x(L^2_y)},\ k\ge0,
\end{equation}
we have
$$
\begin{aligned}
\|\langle\na\rangle^3(\B\cd\na u)\|_{L^1_x(L^2_y)}
\lesssim&\ \|\langle\na\rangle^3(B\p_y u)\|_{L^1_x(L^2_y)}
+\|\langle\na\rangle^3(b\p_x u)\|_{L^1_x(L^2_y)}\\
\lesssim&\ \|B\p_y u\|_{L^1_x(L^2_y)}
+\|B\p_y \na^3u\|_{L^1_x(L^2_y)}+\ other\  similar\   terms.
\end{aligned}
$$
By interpolation inequality, we can get
$$\|B\p_y u\|_{L^1_x(L^2_y)}
\lesssim\ \|B\|_{L^2_x(L^\infty_y)}\|u\|_{L^2}
\lesssim\ \|B\|_{L^2}^\frac{1}{2}
\|\p_xb\|_{L^2}^\frac{1}{2}\|u\|_{L^2}
\lesssim\ \LHH^{-\frac{9}{8}}\|V\|_3^2.$$
Using
\begin{equation}\label{A2}
\|\na\langle \na\rangle^3 u\|_{L^2}
\le\ \|\na\langle \na\rangle^3 u_{<\LHH^{\frac{3}{28}}}\|_{L^2}
+\|\na\langle \na\rangle^3 u_{\ge\LHH^{\frac{3}{28}}}\|_{L^2}
\lesssim\ \LHH^{-\frac{3}{7}}\|V\|_3^2,
\end{equation}
then
$$
\begin{aligned}
\|B\p_y \na^3u\|_{L^1_x(L^2_y)}
\lesssim&\ \|B\|_{L^2_x(L^\infty_y)}
\|\p_y \na^3u\|_{L^2}
\lesssim\ \LHH^{-\frac{3}{7}-\frac{5}{8}}\|V\|_3^2
\lesssim\ \LHH^{-1.01}\|V\|_3^2.
\end{aligned}
$$
Thus
$$\|\langle\na\rangle^3(\B\cd\na u)\|_{L^1_x(L^2_y)}\lesssim\ \LHH^{-1.01}\|V\|_3^2.
$$
\underline{(\ref{ine2})$_3$}
Use (\ref{f3}), we only show the estimates of
$$\|v_{<\LHH^{-8}}\p_y\na^3b\|_{L^1_x(L^2_y)},\ {\rm and}\
\ \|v_{>2\LHH^{-0.05}}\p_y\na^3b\|_{L^1_x(L^2_y)},
$$
and other terms can be bounded similarly.
Using H\"{o}lder's inequality and interpolation inequality, one can get
$$
\begin{aligned}
\|v_{<\LHH^{-8}}\p_y\na^3b\|_{L^1_x(L^2_y)}
\lesssim&\ \|v_{<\LHH^{-8}}\|_{L^2_x(L^\infty_y)}\|\p_y\na^3 b\|_{L^2}\\
\lesssim&\ \|v\|_{L^2}^\frac{1}{2}\|\p_yv_{<\LHH^{-8}}\|_{L^2}^\frac{1}{2}
\|b\|_{H^4}\\
\lesssim&\ \LHH^{-2}\|V\|_3^2,\\
\|v_{>2\LHH^{-0.05}}\p_y\na^3b\|_{L^1_x(L^2_y)}
\lesssim&\ \|v_{>2\LHH^{-0.05}}\|_{L^2_x(L^\infty_y)}
\|\p_y\na^3b\|_{L^2}\\
\lesssim&\ \|v_{>2\LHH^{-0.05}}\|_{L^2}^\frac{1}{2}
\|\p_yv_{>2\LHH^{-0.05}}\|_{L^2}^\frac{1}{2}\\
&\ \times(\|\p_y\na^3b_{<\LHH^{0.05}}\|_{L^2}+
\|\p_y\na^3b_{\ge\LHH^{0.05}}\|_{L^2})\\
\lesssim&\ \LHH^{0.025}\|\na v\|_{L^2}
(\|\p_y\na^3b_{<\LHH^{0.05}}\|_{L^2}+
\|\p_y\na^3b_{\ge\LHH^{0.05}}\|_{L^2})\\
\lesssim&\ \LHH^{0.025-1}(\LHH^{0.2-0.25}+\LHH^{-0.2})\|V\|_3^2\\
\lesssim&\ \LHH^{-1.01}\|V\|_3^2.
\end{aligned}
$$
\underline{(\ref{ine2})$_4$} Using (\ref{f3}), we only estimate
$\|B\p_y \na^2b\|_{L^1_x(L^2_y)}$, and other terms can be controlled  similarly.
 Using interpolation  inequality, and
 \begin{equation}\label{py}
\|\langle\na\rangle^3b\|_{L^2}\lesssim\
\|\langle\na\rangle^3b_{<\LHH^{\frac{1}{32}}}\|_{L^2}+
\|\langle\na\rangle^3b_{\ge \LHH^{\frac{1}{32}}}\|_{L^2}
\lesssim\ \LHH^{-\frac{5}{32}}\|V\|_3,
\end{equation}
 we have
 $$\|B\p_y \na^2b\|_{L^1_x(L^2_y)}
 \lesssim\ \|B\|_{L^2_x(L^\infty_y)}\|\langle\na\rangle^3b\|_{L^2}
 \lesssim\ \LHH^{-\frac{5}{8}-\frac{5}{32}}\|V\|_3^2
 \lesssim\ \LHH^{-0.75}\|V\|_3^2.$$
 \underline{(\ref{ine2})$_5$} Using (\ref{f3}) again, we only bound
 $\|\p_x\langle\na\rangle^2P_\backsim u b \|_{L^1_x(L^2_y)} $, while other terms can be bounded similarly.
 Using interpolation inequality,  and
 $$\|\p_x\langle\na\rangle^3 u\|_{L^2}
 \le\ \|\p_x\langle\na\rangle^3 u_{<\LHH^{0.2}}\|_{L^2}
 +\|\p_x\langle\na\rangle^3 u_{\ge\LHH^{0.2}}\|_{L^2}
 \lesssim\ \LHH^{-0.8}\|V\|_3,$$
 we have
 $$
 \begin{aligned}
 \|\p_x\langle\na\rangle^2P_\backsim u b \|_{L^1_x(L^2_y)}
 \lesssim&\ \|\p_x\langle\na\rangle^2 u\|_{L^2_x(L^\infty_y)}
 \|b\|_{L^2}\\
 \lesssim&\  \|\p_x\langle\na\rangle^3 u\|_{L^2}
 \|b\|_{L^2}\\
 \lesssim&\ \LHH^{-1.05}\|V\|_3^2.
 \end{aligned}
 $$
\end{proof}

\begin{proof}[Proof of Lemma \ref{ne3}] By H\"{o}lder's inequality, and (\ref{p222}),
it is easy to  get the estimate of (\ref{ine3})$_1$ and  (\ref{ine3})$_2$. Let us begin with the estimate of (\ref{ine3})$_3$.\\
\underline{(\ref{ine3})$_3$} We only give the estimate of
$\|\B\cd\na \B\|_{H^2}$, while one can bound other terms by the similar way.  We have
$$\|\B\cd\na \B\|_{H^2}
\le\ \|b\p_x\na^2b\|_{L^2}+\ other\  similar\  terms.$$
Using
\begin{equation}\label{A1}
\|\p_x\na^2b\|_{L^2}
\lesssim\ \|\p_x\na^2b_{<\LHH^{\frac{1}{8}}}\|_{L^2}
+\|\p_x\na^2b_{\ge\LHH^{\frac{1}{8}}}\|_{L^2}
\lesssim\ \LHH^{-\frac{5}{8}}\|V\|_3^2,
\end{equation}
we have
$$\|b\p_x\na^2b\|_{L^2}
\lesssim\ \|b\|_{L^\infty}\|\p_x\na^2b\|_{L^2}
\lesssim\ \LHH^{-1.1}\|V\|_3^2.$$
Thus
$$\|\B\cd\na\B\|_{H^2}
\lesssim\ \LHH^{-1.1}\|V\|_3^2.$$
\underline{(\ref{ine3})$_4$}  Using
$$\|\p_xP_\backsim(\vec{R'}\cd\B)\|_{L^\infty}\lesssim\ \|\widehat{B}\|_{L^1}\lesssim\ \LHH^{-1}\|V\|_3$$
and (\ref{py}),
 one can get
 $$\|\p_xP_\backsim(\vec{R'}\cd\B)\p_yb\|_{L^2}
 \lesssim\ \langle t\rangle^{-1.1}\|V\|_3^2.$$
Then we can get the estimates of other terms by the similar way.\\
  \underline{(\ref{ine3})$_5$} Here we only bound
  $\|P_\backsim(\vec{R'}\cd\B)\p_x\na^2b\|_{L^2}$. Using (\ref{A1}),
  and
  \begin{equation}\label{PB}
  \begin{aligned}
  \|P_\backsim(\vec{R'}\cd\B)\|_{\F L^1}
  \lesssim&\ \||\na|^{-1}b\|_{\F L^1}+\sum_{\LHH^{-8}\le M\le 2\LHH^{-0.05}}
  \||\na|^{-1}P_M B\|_{\F L^1}\\
  \lesssim&\ \LHH^{-0.5}\|V\|_3+\sum_{\LHH^{-8}\le M\le 2\LHH^{-0.05}} M^{-1}
    \|P_M B\|_{\F L^1}\\
    \lesssim&\ \LHH^{-0.5}\|V\|_3+\sum_{\LHH^{-8}\le M\le 2\LHH^{-0.05}}
    \|P_M B\|_{ L^2}\\
    \lesssim&\ \LHH^{-0.5}\|V\|_3+\|B\|_{L^2}\sum_{\LHH^{-8}\le M\le 2\LHH^{-0.05}}
     M^{0.001}M^{-0.001}\\
     \lesssim&\ \LHH^{-0.5}\|V\|_3+\LHH^{0.008-0.5}\|V\|_3
     \lesssim\ \LHH^{-0.492}\|V\|_3.
  \end{aligned}
  \end{equation}
  we have
$$
\|P_\backsim(\vec{R'}\cd\B)\p_x\na^2b\|_{L^2}
\lesssim \|P_\backsim(\vec{R'}\cd\B)\|_{\F L^1}
\|\p_x\na^2b\|_{L^2}
\lesssim\ \LHH^{-1.1}\|V\|_3^2.
$$
  \underline{(\ref{ine3})$_6$} Here we only show the estimates of  $\|b\p_x^2\na^2b\|_{L^2}$  and $\|P_\backsim(\vec{R'}\cd\B)\p_y\na^3u\|_{L^2}$.
  Using
  $$\|\p_x\langle\na\rangle^3b\|_{L^2}
  \lesssim\ \|\p_x\langle\na\rangle^3b_{<\LHH^\frac{1}{8}}\|_{L^2}
  +\|\p_x\langle\na\rangle^3b_{\ge\LHH^\frac{1}{8}}\|_{L^2}
  \lesssim
  \ \LHH^{-\frac{1}{2}}\|V\|_3^2,$$
  we have
  $$\|b\p_x^2\na^2b\|_{L^2}
  \lesssim\ \|b\|_{L^\infty}\|\p_x^2\na^2b\|_{L^2}
  \lesssim\ \LHH^{-1}\|V\|_3^2.$$
  Thanks to (\ref{PB}) and (\ref{A2}),
  we can get the  estimate of $\|P_\backsim(\vec{R'}\cd\B)\p_y\na^3u\|_{L^2}$ by using H\"{o}lder's inequality.\\
   \underline{(\ref{ine3})$_7$} By using  (\ref{PB}) and (\ref{py}), one can easily get this estimate.
\end{proof}

\begin{proof}[Proof of Lemma \ref{ne4}]
\underline{(\ref{ine4})$_1$} The first three terms can be bounded easily. By
$$\|\na P_\backsim(\vec{R'}\cd\B)\|_{\F L^1}\lesssim\ \|\B\|_{\F L^1}
\lesssim\ \LHH^{-\frac{1}{2}}\|V\|_3,$$
we can get
$$\|[\p_x(\p_yP_\backsim(\vec{R'}\cd\B)b),\na P_\backsim(\vec{R'}\cd\B)u]\|_{\F L^1}
\lesssim\ \|\B\|_{\F L^1}(\|\p_xb\|_{\F L^1}+\|\widehat{u}\|_{L^1})
\lesssim\ \LHH^{-\frac{3}{2}}\|V\|_3^2.$$
\underline{(\ref{ine4})$_2$}
Like the previous procedure, it suffices to estimate $\|b\p_{xx}u\|_{\F L^1}$ and $\|\p_yb\p_{x}v\|_{\F L^1}$. Using
\begin{equation}\label{ub}
\begin{aligned}
\|\langle\na\rangle\p_{x}u\|_{\F L^1}
\le&\ \|\langle\na\rangle\p_{x}u_{\ge \LHH^{0.5}}\|_{\F L^1}
+\|\langle\na\rangle\p_{x}u_{< \LHH^{0.5}}\|_{\F L^1}
\lesssim\ \LHH^{-0.75}\|V\|_3,\\
\|\langle\na\rangle b\|_{\F L^1}\le&\
\|\langle\na\rangle b_{<\LHH^{\frac{5}{69}}}\|_{\F L^1}+\|\langle\na\rangle b_{\ge\LHH^\frac{5}{69}}\|_{\F L^1}
\lesssim\ \LHH^{-0.42}\|V\|_3,\\
\|\p_xv\|_{\F L^1}
\lesssim&\ \|\p_xv_{<\LHH^{0.15}}\|_{\F L^1}+\|\p_xv_{\ge\LHH^{0.15}}\|_{\F L^1}
\lesssim\ \LHH^{-0.85}\|V\|_3,
\end{aligned}
\end{equation}
one gets
$$\|b\p_{xx}u\|_{\F L^1}
\lesssim\ \|\widehat{b}\|_{L^1}\|\p_{xx}u\|_{\F L^1}
\lesssim\ \LHH^{-\frac{5}{4}}\|V\|_3^2$$
and
$$
\|\p_yb\p_xv\|_{\F L^1}
\lesssim\ \|\p_yb\|_{\F L^1}\|\p_x v\|_{\F L^1}
\lesssim\ \LHH^{-\frac{5}{4}}\|V\|_3^2.
$$
\underline{(\ref{ine4})$_3$}  Using (\ref{PB}) and the fact that
 $P_\thickapprox$ can be bounded by the  process dealing with  $P_\thicksim$, we can get the desired estimate by H\"{o}lder's inequality.\\
\underline{(\ref{ine4})$_4$} Here we only show the estimate of $\|P_\backsim(\vec{R'}\cd\B)\p_y\na^2u\|_{\F L^1}$.
Using
$$
\|\langle\na\rangle^3u\|_{\F L^1}
\lesssim\ \|\langle\na\rangle^3u_{<\LHH^\frac{10}{69}}\|_{\F L^1}
+\|\langle\na\rangle^3u_{\ge\LHH^\frac{10}{69}}\|_{\F L^1}
\lesssim\ \LHH^{-\frac{39}{69}}\|V\|_3
$$
and (\ref{PB}),
we have
$$\|P_\backsim(\vec{R'}\cd\B)\p_y\na^2u\|_{\F L^1}
\lesssim\ \|P_\backsim(\vec{R'}\cd\B)\|_{\F L^1}
\|\p_y\na^2u\|_{\F L^1}
\lesssim\ \LHH^{-1.01}\|V\|_3^2.$$
\underline{(\ref{ine4})$_5$} We only give the estimate of $\|\p_x(b\p_xb)\|_{\F L^1}$. Using
$$\|\p_x^2b\|_{\F L^1}
\le\ \|\p_x^2b_{<\LHH^\frac{10}{59}}\|_{\F L^1}
+\|\p_x^2b_{\ge\LHH^\frac{10}{59}}\|_{\F L^1}
\lesssim\ \LHH^{-0.83}\|V\|_3,$$
we have
$$\|\p_x(b\p_xb)\|_{\F L^1}
\le\ \|\p_xb\p_xb\|_{\F L^1}+\|b\p_x^2b\|_{\F L^1}
\lesssim\ \LHH^{-1.3}\|V\|_3^2.$$
\underline{(\ref{ine4})$_6$} Using the same way yielding (\ref{ub})$_3$, we can get
$$\|\na B\|_{\F L^1}\lesssim\ \LHH^{-0.85}\|V\|_3,$$
which, together
with (\ref{ub})$_2$ yields
$$\|\na(B\B)\|_{\F L^1}
\lesssim\ \LHH^{-1.2}\|V\|_3^2.$$
Other terms can be bounded similarly.\\
\underline{(\ref{ine4})$_7$}  We only estimate  $\|\na^2 B\p_y b\|_{\F L^1}$, while other terms can be bounded similarly. Using
the same way yielding (\ref{ub})$_3$, one has
$$\|\na^2 B\|_{\F L^1}\lesssim\ \LHH^{-0.71}\|V\|_3,$$
which,  with (\ref{ub})$_2$ leads
$$\|\na^2 B\p_y b\|_{\F L^1}
\lesssim\ \|\na^2 B\|_{\F L^1}
\|\p_y b\|_{\F L^1}
\lesssim\ \LHH^{-1.1}\|V\|_3^2.$$
Using (\ref{ub})$_2$ and $(\ref{PB})$,
we have
$$\|P_\backsim(\vec{R'}\cd\B)\p_yb\|_{\F L^1}
\lesssim\ \|P_\backsim(\vec{R'}\cd\B)\|_{\F L^1}
\|\p_yb\|_{\F L^1}
\lesssim\ \LHH^{-0.9}\|V\|_3^2.$$
\end{proof}

\end{document}